\documentclass[letterpaper]{amsart}
\pdfoutput=1
\newtheorem{theorem}{Theorem}[section]

\newtheorem{proposition}[theorem]{Proposition}
\newtheorem{corollary}[theorem]{Corollary}

\theoremstyle{definition}
\newtheorem{definition}[theorem]{Definition}

\theoremstyle{remark}

\usepackage{booktabs}
\usepackage[pdftex]{graphicx}
\usepackage{color}
\usepackage{tikz}
\usepackage{amsmath}
\usepackage{amsfonts}
\usepackage{amssymb}
\usepackage{epsfig}
\usepackage{rotating}
\usepackage{graphics}
\makeatletter
\renewcommand*\env@matrix[1][*\c@MaxMatrixCols c]{%
  \hskip -\arraycolsep
  \let\@ifnextchar\new@ifnextchar
  \array{#1}}
\makeatother
\newcommand{\be}{\begin{equation}}

\newcommand{\ee}{\end{equation}}

\newcommand{\al}{\alpha}

\newcommand{\Om}{\Omega}

\newcommand{\om}{\omega}

\newcommand{\dz}{\wedge}

\newcommand{\ba}{\begin{array}}

\newcommand{\ea}{\end{array}}

\newcommand{\beq}{\begin{eqnarray}}

\newcommand{\eeq}{\end{eqnarray}}

\newtheorem{lm}{lemma}

\newtheorem{thee}{theorem}

\newtheorem{proo}{proposition}

\newtheorem{co}{corollary}

\newtheorem{rem}{remark}

\newtheorem{deff}{definition}

\newcommand{\bd}{\begin{deff}}

\newcommand{\ed}{\end{deff}}

\newcommand{\bl}{\begin{lm}}

\newcommand{\el}{\end{lm}}

\newcommand{\bp}{\begin{proo}}

\newcommand{\ep}{\end{proo}}

\newcommand{\bt}{\begin{thee}}

\newcommand{\et}{\end{thee}}

\newcommand{\bc}{\begin{co}}

\newcommand{\ec}{\end{co}}

\newcommand{\brm}{\begin{rem}}

\newcommand{\erm}{\end{rem}}

\newcommand{\der}{{\rm d}}

\usepackage{t1enc}
\def\frak{\mathfrak}

\newcommand{\newc}{\newcommand}

\let\ccdot\cdot
\def\cdot{\hbox to 2.5pt{\hss$\ccdot$\hss}}

\newc{\aR}{\mbox{\boldmath{$ R$}}}
\newc{\aS}{\mbox{\boldmath{$ S$}}}
\newc{\aT}{\mbox{\boldmath{$ T$}}}
\newc{\aW}{\mbox{\boldmath{$ W$}}}

\newc{\aK}{\mbox{\boldmath{$ K$}}}
\newc{\aL}{\mbox{\boldmath{$ L$}}}

\usepackage{amssymb}
\usepackage{amscd}

\newcommand{\hook}{\raisebox{-0.35ex}{\makebox[0.6em][r]
{\scriptsize $-$}}\hspace{-0.15em}\raisebox{0.25ex}{\makebox[0.4em][l]{\tiny
 $|$}}}

\def\bbZ{{\mathbb{Z}}}

\newcommand{\X}{\mbox{\boldmath{$ X$}}}

\newc{\obstrn}[2]{B^{#1}_{#2}}

\newcommand{\rpl}                         
{\mbox{$
\begin{picture}(12.7,8)(-.5,-1)
\put(0,0.2){$+$}
\put(4.2,2.8){\oval(8,8)[r]}
\end{picture}$}}

\newcommand{\lpl}                         
{\mbox{$
\begin{picture}(12.7,8)(-.5,-1)
\put(2,0.2){$+$}
\put(6.2,2.8){\oval(8,8)[l]}
\end{picture}$}}

\usepackage{ifthen}
\newcommand{\bbR}{\mathbb{R}}

\newcommand{\glg}{\mathbf{GL}}

\newcommand{\gla}{\frak{gl}}

\newcommand{\spg}{\mathbf{Sp}}
\newcommand{\spa}{\mathfrak{sp}}
\newc{\tensor}[1]{#1}
\newc{\Mvariable}[1]{\mbox{#1}}
\newc{\down}[1]{{}_{#1}}
\newc{\up}[1]{{}^{#1}}

\newc{\JulyStrut}{\rule{0mm}{6mm}}
\newc{\midtenPan}{\mbox{\sf S}}
\newc{\midten}{\mbox{\sf T}}
\newc{\midtenEi}{\mbox{\sf U}}
\newc{\ATen}{\mbox{\sf E}}
\newc{\BTen}{\mbox{\sf F}}
\newc{\CTen}{\mbox{\sf G}}

\def\sideremark#1{\ifvmode\leavevmode\fi\vadjust{\vbox to0pt{\vss
 \hbox to 0pt{\hskip\hsize\hskip1em
 \vbox{\hsize3cm\tiny\raggedright\pretolerance10000
 \noindent #1\hfill}\hss}\vbox to8pt{\vfil}\vss}}}%
                        
\newcommand{\Span}{\mathrm{Span}}

\numberwithin{equation}{section}

\newcounter{romenumi}
\newcommand{\labelromenumi}{(\roman{romenumi})}

\newcommand{\bma}{\begin{pmatrix}}
\newcommand{\ema}{\end{pmatrix}}

\begin{document}
\title{$\spg(3,\bbR)$ Monge geometries in dimension 8} 
\vskip 1.truecm \author{Ian M Anderson} \address{Department of Mathematics and Statistics, Utah State University, Logan Utah, 84322}
\email{\tt Ian.Anderson@usu.edu}
\author{Pawe\l~ Nurowski} \address{Center for Theoretical Physics, Polish Academy of Sciences, Al. Lotnikow 32/46, 02-688 Warszawa, Poland} 
\email{nurowski@fuw.edu.pl} \thanks{This reaserch was supported by the
Polish National Science Centre under the grant DEC-2013/09/B/ST1/01799.}
\date{\today}

\begin{abstract} We study a geometry associated with rank 3 distributions in dimension 8, whose symbol algebra is constant and has a simple Lie algebra $\spa(3,\bbR)$ as Tanaka prolongation. We restict our considerations to only those distributions that are defined in terms of a systems of ODEs of the form
  $\dot{z}_{ij}=\frac{\partial^2 f(\dot{x}_1,\dot{x}_2)}{\partial \dot{x}_i\partial \dot{x}_j}$, $i\leq j=1,2$. For them we built the full system of local differential invariants, by solving an equivalence problem a'la Cartan, in the spirit of his 1910's five variable paper. The considered geometry is a parabolic geometry, and we show that its main invariant - the harmonic curvature - is a certain quintic. In the case when this quintic is maximally degenerate but nonzero, we use Cartan's reduction procedure and reduce the EDS governing the invariants to 11, 10 and 9 dimensions. As a byproduct all homogeneous models having maximally degenerate harmonic curvature quintic are found. They have symmetry algebras of dimension 11 (a unique structue), 10 (a 1-parameter family of nonequivalent structures) or 9 (precisely two nonequivalent structures).  
\end{abstract}
\maketitle
\tableofcontents
\newcommand{\gat}{\tilde{\gamma}}
\newcommand{\Gat}{\tilde{\Gamma}}
\newcommand{\thet}{\tilde{\theta}}
\newcommand{\Thet}{\tilde{T}}
\newcommand{\rt}{\tilde{r}}
\newcommand{\st}{\sqrt{3}}
\newcommand{\kat}{\tilde{\kappa}}
\newcommand{\kz}{{K^{{~}^{\hskip-3.1mm\circ}}}}

\section{Introduction}\label{se1}
A Monge system of ODEs, introduced in \cite{IZP}, is a generalization of the celebrated ODE 
\be \dot{z}=\tfrac12 \ddot{x}^2.\label{ch}\ee
This is a differential equation for two real functions $z=z(t)$ and $x=x(t)$ of one real variable $t$, first considered by Hilbert \cite{Hilbert} and Cartan \cite{Cartan1910,Cartan1914} at the begining of XXth century. The natural space associated with the Cartan-Hilbert equation (\ref{ch}) is a 5-dimensional jet manifold $M$ parametrized by $(t,x,\dot{x},\ddot{x},z)$. On this manifold the solutions to (\ref{ch}) are curves
$$\gamma(s)=(t(s),x(s),\dot{x}(s),\ddot{x}(s),z(s))$$ tangent to
a rank 2 distribution
$${\mathcal D}~=~\Span_\bbR(~X=\partial_t+\dot{x}\partial_x+\ddot{x}\partial_{\dot{x}}+\tfrac12\ddot{x}^2\partial_z,~~X_1=\partial_{\ddot{x}}~).$$
It is this distribution that encodes the geometry of the ODE (\ref{ch}). As one moves from one point of $M$ to another the \emph{symbol algebra} of $\mathcal D$ is constant and isomorphic to a \emph{3-step nilpotent} Lie algebra
$$\mathfrak{m}={\mathfrak g}_{-3}\oplus{\mathfrak g}_{-2}\oplus{\mathfrak g}_{-1}$$
with
$$\begin{aligned}
  &{\mathfrak g}_{-1}={\mathcal D}=\Span_\bbR(X,X_1),\\
&{\mathfrak g}_{-2}=\Span_\bbR(Y=[X,X_1])=\Span_\bbR(\partial_{\dot{x}}+\ddot{x}\partial_z),\\ 
  &{\mathfrak g}_{-3}=\Span_\bbR(Z_1=[X,Y],Z_2=[X_1,Y])=\Span_\bbR(\partial_x,\partial_z).\end{aligned}$$
Rank 2 distributions in dimension five, with constant symol algebra isomorphic to $\mathfrak{m}$ above, are called $(2,3,5)$-distributions. If one considers them modulo local diffeomorphisms, they are in one-to-one correspondence with \emph{Monge} ODEs
\be \dot{z}=F(t,x,\dot{x},\ddot{x},z),\label{ch1}\ee
defined in terms of a real differentiable function $F=F(t,x,\dot{x},\ddot{x},z)$ of its five variables $(t,x,\dot{x},\ddot{x},z)$, 
such that $F_{\ddot{y}\ddot{y}}\neq 0$. The correspondence is given via
\be\begin{aligned}\dot{z}=F(t,x,\dot{x},&\ddot{x},z)\quad\longleftrightarrow\\
  &\quad{\mathcal D}_F~=~\Span_\bbR(~X=\partial_t+\dot{x}\partial_x+\ddot{x}\partial_{\dot{x}}+F\partial_z,~~X_1=\partial_{\ddot{x}}~).\end{aligned}\label{ch2}\ee

A remarkable property of Monge ODEs (\ref{ch1}) with $F_{\ddot{y}\ddot{y}}\neq 0$ is that their geometry, encoded in the corresponding $(2,3,5)$ distributions ${\mathcal D}_F$, is describable in terms of a $\tilde{\mathfrak{g}}_2$-valued \emph{Cartan connection} \cite{Cartan1910,Nur}. Here $\tilde{\mathfrak{g}}_2$ denotes the split real form of the exceptional Lie algebra $\mathfrak{g}_2$. In particular, this connection is flat if $F=\tfrac12\ddot{y}^2$, and in such case, which is the case of the Cartan-Hilbert equation (\ref{ch}), the distribution ${\mathcal D}_{F=\tfrac12\ddot{y}^2}$ has $\tilde{\mathfrak{g}}_2$ as its \emph{Lie algebra of infinitesimal symmetries}. We recall here, that a vector field $X\neq 0$ on a manifold $M$ is an \emph{infinitesimal symmetry} of a distribution $\mathcal D$ on $M$ if $[X,{\mathcal D}]\subset{\mathcal D}$, and that infinitesimal symmetries of a distribution $\mathcal D$ on $M$ form a Lie algebra, called the Lie algebra of infinitesimal symmetries of $\mathcal D$.

In Reference \cite{IZP} we generalized the Monge ODE (\ref{ch1}) to a certain class of systems of ODEs, which are defined in terms of a certain class of distributions, which we called \emph{Monge distributions}. The formal definition of them is as follows: 
\begin{definition}
A distribution $\mathcal D$ on an $n$-dimensional manifold $M$ is \emph{Monge} if the following conditions occurs: 
\begin{itemize}
\item[i)] $\mathcal D$ is bracket generating, i.e.
  $${\mathcal D}+[{\mathcal D},{\mathcal D}]+[{\mathcal D},[{\mathcal D},{\mathcal D}]+...=\mathrm{T}M,$$
  and the sum on the left hand side has a finite number of terms,
\item[ii)] $\mathcal D$ has constant symbol algebra $\mathfrak{m}=\oplus_{i=1}^{p\geq2}\mathfrak{g}_{-i}$ of dimension  $\mathrm{dim}(\mathfrak{m})=n$,
  \item[iii)] the grade minus one part $\mathfrak{g}_{-1}$ of the symbol algebra $\mathfrak{m}$ contains an \emph{abelian} subalgebra $\mathfrak{h}\subset\mathfrak{g}_{-1}$ of codimension \emph{one} in $\mathfrak{g}_{-1}$, and $[\mathfrak{g}_{-1},\mathfrak{h}]=\mathfrak{g}_{-2}$,
\item[iv)] the Tanaka prolongation, $\mathfrak{g}_T$, of the symbol algebra $\mathfrak{m}$ is a simple real Lie algebra.
\end{itemize}
\end{definition}
It follows from property iv) that the maximal algebra of infinitesimal symmetries of Monge distribution $\mathcal D$ is then $\mathfrak{g}_T$, and that the geometry associated with Monge geometries are \emph{parabolic} \cite{cap}. It also follows (from property iii)) \cite{IZP}, that a Monge distribution $\mathcal D$ defines a system of ODEs, called the \emph{Monge system} of $\mathcal D$, whose solutions are curves in $M$ tangent to $\mathcal D$.

In particular, the distributions ${\mathcal D}_F$ as in (\ref{ch2}) are Monge for all $F$'s such that $F_{\ddot{x}\ddot{x}}\neq 0$. For them we have $\mathfrak{h}=\Span_\bbR(X_1)$, $\mathfrak{g}_T=\tilde{\mathfrak{g}}_2$, and it follows that their corresponding system of Monge ODEs consists of a single \emph{classical} Monge ODE (\ref{ch1}). 

The pourpose of the present paper is to look closer at a particular class of Monge geometries, different from the Cartan-Hilber ones, which we have chosen from the list of \emph{non-rigid} Monge geometries given in \cite{IZP}. We call them $\spg(3,\bbR)$ \emph{parabolic Monge geometries in dimension 8}. In the terminology of \cite{IZP}, Theorem B, these are \emph{non-rigid} parabolic geometries from the infinite series of geometries called type IIa, with the labelling parameter $\ell=3$. We will establish here all local differential invariants of these geometries, and give all homogeneous examples of them in the case when their \emph{harmonic curvature} is most degenerate but not zero.

Thus the Monge geometries we are going to study here have $M$ of dimension 8. In addition they have $\mathcal D$ of rank 3, and $\mathfrak{g}_T=\spa(3,\bbR)$. In the case of harmonic curvature equal to zero the geometry we are going to study has the following corresponding system of Monge ODEs:
\be\dot{z}_{11}=\dot{x}_1^2,\quad\dot{z}_{12}=\dot{x}_1\dot{x}_2,\quad\dot{z}_{22}=\dot{x}_2^2.\label{flatc3}\ee
Here the unknowns are real functions $(x_1(t),x_2(t),z_{11}(t),z_{12}(t),z_{22}(t))$ of one real variable $t$.

This (flat) system of ODEs naturally lives on an 8-dimensional manifold $M$ parametrized by $(t,x_1,x_2,\dot{x}_1,\dot{x}_2,z_{11},z_{12},z_{22})$. Its Monge distribution $\mathcal D$ has rank 3 and is given by:
\be {\mathcal D}~=~\Span_\bbR(X,X_1,X_2),\label{dm1}\ee
with
\be \begin{aligned}
&X=\partial_t+\dot{x}_1^2\partial_{z_{11}}+\dot{x}_1\dot{x}_2\partial_{z_{12}}+\dot{x}_2^2\partial_{z_{22}}+\dot{x}_1\partial_{x_1}+\dot{x}_2\partial_{x_2},\\
  &X_1=\partial_{\dot{x}_1},\\
  &X_2=\partial_{\dot{x}_2}.\\
\end{aligned}
\label{dm2}\ee
One checks that the symbol algebra of $\mathcal D$ is now:
\be {\mathfrak m}={\mathfrak g}_{-3}\oplus{\mathfrak g}_{-2}\oplus{\mathfrak g}_{-1}.\label{em1}\ee
with
$$\begin{aligned}
&{\mathfrak g}_{-1}={\mathcal D}=\Span_\bbR(X,X_1,X_2),\\
&{\mathfrak g}_{-2}=\Span_\bbR(Y_1,Y_2),\\
&{\mathfrak g}_{-3}=\Span_\bbR(Z_{11},Z_{12},Z_{22}),\end{aligned}$$
where the vector fileds $Y_i,Z_{ij}$ are defined by
$$\begin{aligned}
  &Y_1=[X_1,X]=\partial_{x_1}+2\dot{x}_1\partial_{z_{11}}+2\dot{x}_2\partial_{z_{12}},\\
  &Y_2=[X_2,X]=\partial_{x_2}+2\dot{x}_1\partial_{z_{12}}+2\dot{x}_2\partial_{z_{22}},\\
  &Z_{11}=[Y_1,X_1]=-2\partial_{z_{11}},\\
  &Z_{12}=[Y_1,X_2]=-2\partial_{z_{12}},\\
  &Z_{22}=[Y_2,X_2]=-2\partial_{z_{22}}.
\end{aligned}$$
Thus $\mathfrak{m}$ is 3-step nilpotent Lie algebra, and the distribution $\mathcal D$ is Monge in the sense of Reference \cite{IZP} with co-dimension 1 commutative Lie algebra $\mathfrak{h}$ in $\mathfrak{g}_{-1}$ being
$$\mathfrak{h}=\Span_\bbR(X_1,X_2).$$
In Reference \cite{IZP} the following proposition was proven:
\begin{proposition}
  The Tanaka prolongation of the nilpotent Lie algebra $\mathfrak{m}$ as in (\ref{em1}) is $\mathfrak{g}_T=\spa(3,\bbR)$, and the Lie algebra of infinitesimal symmetries of the distribution $\mathcal D$ given by (\ref{dm1})-(\ref{dm2}) is also isomorphic to $\spa(3,\bbR)$. 
  \end{proposition}

Actually for $\mathcal D$ and $\mathfrak{m}$ as in (\ref{dm1})-(\ref{dm2}), (\ref{em1}) we have:
$$\mathfrak{g}_T=\spa(3,\bbR)=\mathfrak{m}\oplus\mathfrak{p},$$
with
$$\mathfrak{p}=\mathfrak{g}_0\oplus\mathfrak{g}_1\oplus\mathfrak{g}_2\oplus\mathfrak{g}_3,$$
and
$$\dim\mathfrak{g}_{\pm1}=\dim\mathfrak{g}_{\pm 3}=3,\quad\dim \mathfrak{g}_{\pm2}=2,\quad\dim\mathfrak{g}_0=5.$$
This gives a gradation in $\spa(3,\bbR)$:
\be
\spa(3,\bbR)=\mathfrak{m}\oplus\mathfrak{p}={\mathfrak g}_{-3}\oplus{\mathfrak g}_{-2}\oplus{\mathfrak g}_{-1}\oplus\mathfrak{g}_0\oplus\mathfrak{g}_1\oplus\mathfrak{g}_2\oplus\mathfrak{g}_3,\label{mgr}\ee
$$[\mathfrak{g}_i,\mathfrak{g}_j]\subset\mathfrak{g}_{i+1}.$$
It is related to a \emph{parabolic} subalgebra $\mathfrak{p}\subset\spa(3,\bbR)$ chosen by the following decoration of the $\spa(3,\bbR)$ Dynkin diagram: \\ 
\centerline{\begin{tikzpicture}[scale=0.7]
  \draw (1,0) -- (2,0);
    \draw (2,0.1) -- (3,0.1);
    \draw (2,-0.1) -- (3,-0.1);
    \draw (2.6,-0.2) -- (2.4,0) -- (2.6,0.2);
    \filldraw[color=black] (2,0) circle  (5pt);
    \filldraw[color=white] (2,0) circle (4pt);
    \draw (1.8,-0.2) -- (2.2,0.2);
    \draw (1.8,0.2) -- (2.2,-0.2);
  \filldraw[color=black] (3,0) circle (5pt);
  \filldraw[color=white] (3,0) circle (4pt);
  \draw (2.8,-0.2) -- (3.2,0.2);
    \draw (2.8,0.2) -- (3.2,-0.2);
    \filldraw[color=black] (1,0) circle (5pt);
    \filldraw[color=white] (1,0) circle (4pt);
  \end{tikzpicture}.}
We will call the gradation (\ref{mgr}) as the \emph{Monge gradation of} $\spa(3,\bbR)$.

This shows that the geometry we are going to consider in this paper is a (non-rigid) \emph{parabolic geometry} of type $(\spg(3,\bbR), P)$ with a parabolic subgroup $P\subset\spg(3,\bbR)$ corresponding to the decorated Dynkin diagram \begin{tikzpicture}[scale=0.7]
  \draw (1,0) -- (2,0);
    \draw (2,0.1) -- (3,0.1);
    \draw (2,-0.1) -- (3,-0.1);
    \draw (2.6,-0.2) -- (2.4,0) -- (2.6,0.2);
    \filldraw[color=black] (2,0) circle  (5pt);
    \filldraw[color=white] (2,0) circle (4pt);
    \draw (1.8,-0.2) -- (2.2,0.2);
    \draw (1.8,0.2) -- (2.2,-0.2);
  \filldraw[color=black] (3,0) circle (5pt);
  \filldraw[color=white] (3,0) circle (4pt);
  \draw (2.8,-0.2) -- (3.2,0.2);
    \draw (2.8,0.2) -- (3.2,-0.2);
    \filldraw[color=black] (1,0) circle (5pt);
    \filldraw[color=white] (1,0) circle (4pt);
\end{tikzpicture}. Our goal is to study equivalence problem for these geometries, to define their main local invariants, and provide interesting nonflat homogeneous examples of them. 

This paper should be considered as complementary to Ref. \cite{Gutt}, where an alternative method of classifying homogeneous models for the Monge geometry in dimension 8 is given.

\section{Flat model and $\spg(3,\bbR)$}\label{se2}
We first look at a Monge gradation in $\spa(3,\bbR)$, viewed in the standard 6-dimensional representation. On doing that we set the most general element $X$ of $\spa(3,\bbR)$ in the form of a $6\times 6$ real matrix $X$ given by:
\be\begin{aligned}
  X=\bma[c|c||c|c]
  {\bf c}&{\bf q}&{\bf b}&{\bf v}\\
  \cmidrule(lr){1-4}
  -{\bf p}^T&c&{\bf v}^T&q\\
  \cmidrule(lr){1-4}\morecmidrules\cmidrule(lr){1-4}
  {\bf a}&{\bf u}&-{\bf c}^T&{\bf p}\\
  \cmidrule(lr){1-4}
  {\bf u}^T&p&-{\bf q}^T&-c
  \ema
  =\bma[c|c||c|c]
  \begin{matrix}c^1{}_1&c^1{}_2\\c^2{}_1&c^2{}_2\end{matrix}&\begin{matrix}q^1\\q^2\end{matrix}&\begin{matrix}b^1{}_1&b^1{}_2\\b^1{}_2&b^2_2\end{matrix}&\begin{matrix}v^1\\v^2\end{matrix}\\
  \cmidrule(lr){1-4}
  \begin{matrix}-p^1&-p^2\end{matrix}&c&\begin{matrix}v^1&v^2\end{matrix}&q\\
  \cmidrule(lr){1-4}\morecmidrules\cmidrule(lr){1-4}
  \begin{matrix}a^1{}_1&a^1{}_2\\a^1{}_2&a^2{}_2\end{matrix}&\begin{matrix}u^1\\u^2\end{matrix}&\begin{matrix}-c^1{}_1&-c^2{}_1\\-c^1{}_2&-c^2{}_2\end{matrix}&\begin{matrix}p^1\\p^2\end{matrix}\\
  \cmidrule(lr){1-4}
  \begin{matrix}u^1&u^2\end{matrix}&p&\begin{matrix}-q^1&-q^2\end{matrix}&-c
      \ema.
\end{aligned}\label{sp3r}\ee
Here ${\bf a}={\bf a}^T$, ${\bf b}={\bf b}^T$  and ${\bf c}\in\mathrm{M}_{2\times 2}(\bbR)$ are real $2\times 2$ matrices, ${\bf p}$, ${\bf q}$, ${\bf u}$, ${\bf v}\in \bbR^2$ are real vectors, and $c,p,q\in\bbR$ are real numbers. Thus we have
$$\spa(3,\bbR)=\{X\,\mathrm{as\,in\,(\ref{sp3r})}\}.$$
The reason for this presentation of $\spa(3,\bbR)$ is that it is adapted to the Monge gradation in $\spa(3,\bbR)$. Indeed, we have 
$$\spa(3,\bbR)=\mathfrak{g}_{-3}\oplus\mathfrak{g}_{-2}\oplus\mathfrak{g}_{-1}\oplus\mathfrak{g}_{0}\oplus\mathfrak{g}_{1}\oplus\mathfrak{g}_{2}\oplus\mathfrak{g}_{3}$$
with
$$\begin{aligned}
  &\mathfrak{g}_{-3}=\{\bma[c|c||c|c]
  0&0&0&0\\
  \cmidrule(lr){1-4}
  0&0&0&0\\
  \cmidrule(lr){1-4}\morecmidrules\cmidrule(lr){1-4}
  {\bf a}&0&0&0\\
  \cmidrule(lr){1-4}
  0&0&0&0
  \ema,\,{\bf a}={\bf a}^T\in{\mathrm M}_{2\times 2}(\bbR)~\},\quad\mathfrak{g}_{-2}=\{\bma[c|c||c|c]
  0&0&0&0\\
  \cmidrule(lr){1-4}
  0&0&0&0\\
  \cmidrule(lr){1-4}\morecmidrules\cmidrule(lr){1-4}
  0&{\bf u}&0&0\\
  \cmidrule(lr){1-4}
  {\bf u}^T&0&0&0
  \ema,\,{\bf u}\in\bbR^2\}\\
 &\mathfrak{g}_{-1}=\{\bma[c|c||c|c]
  0&0&0&0\\
  \cmidrule(lr){1-4}
  -{\bf p}^T&0&0&0\\
  \cmidrule(lr){1-4}\morecmidrules\cmidrule(lr){1-4}
  0&0&0&{\bf p}\\
  \cmidrule(lr){1-4}
  0&p&0&0
  \ema,\,{\bf p}\in\bbR^2,\,p\in\bbR\},\quad \mathfrak{g}_{0}=\{\bma[c|c||c|c]
  {\bf c}&0&0&0\\
  \cmidrule(lr){1-4}
  0&c&0&0\\
  \cmidrule(lr){1-4}\morecmidrules\cmidrule(lr){1-4}
  0&0&-{\bf c}^T&0\\
  \cmidrule(lr){1-4}
  0&0&0&-c
  \ema,\,{\bf c}\in\bbR^2,\,c\in\bbR\}\\
&\mathfrak{g}_{1}=\{\bma[c|c||c|c]
  0&{\bf q}&0&0\\
  \cmidrule(lr){1-4}
  0&0&0&q\\
  \cmidrule(lr){1-4}\morecmidrules\cmidrule(lr){1-4}
  0&0&0&0\\
  \cmidrule(lr){1-4}
  0&p&-{\bf q}^T&0
  \ema,\,{\bf q}\in\bbR^2,\,q\in\bbR\},\quad\mathfrak{g}_{2}=\{\bma[c|c||c|c]
  0&0&0&{\bf v}\\
  \cmidrule(lr){1-4}
  0&0&{\bf v}^T&0\\
  \cmidrule(lr){1-4}\morecmidrules\cmidrule(lr){1-4}
  0&0&0&0\\
  \cmidrule(lr){1-4}
  0&0&0&0
  \ema,\,{\bf v}\in\bbR^2\},\\
  &\mathfrak{g}_{3}=\{\bma[c|c||c|c]
  0&0&{\bf b}&0\\
  \cmidrule(lr){1-4}
  0&0&0&0\\
  \cmidrule(lr){1-4}\morecmidrules\cmidrule(lr){1-4}
  0&0&0&0\\
  \cmidrule(lr){1-4}
  0&0&0&0
  \ema,\,{\bf b}={\bf b}^T\in\mathrm{M}_{2\times 2}(\bbR)\}.
\end{aligned}$$
Note that we have the following dimensions: $\dim\mathfrak{g}_{\pm3}=\dim\mathfrak{g}_{\pm1}=3$, $\dim\mathfrak{g}_{\pm2}=2$, $\dim\mathfrak{g}_0=5$, and thus the spaces $\mathfrak{g}_i$, $i=0,\pm1,\pm 2,\pm3$ correspond to the Monge gradation of $\spa(3,\bbR)$, as described in the previous section.

Given the group $\spg(3,\bbR)$ one has its Maurer-Cartan form
$$\omega=\bma[c|c||c|c]
  {\bf c}&{\bf q}&{\bf b}&{\bf v}\\
  \cmidrule(lr){1-4}
  -{\bf p}^T&c&{\bf v}^T&q\\
  \cmidrule(lr){1-4}\morecmidrules\cmidrule(lr){1-4}
  {\bf a}&{\bf u}&-{\bf c}^T&{\bf p}\\
  \cmidrule(lr){1-4}
  {\bf u}^T&p&-{\bf q}^T&-c
  \ema,$$
with the entry 1-forms 
$$\begin{aligned}
  {\bf a}=&\bma\theta^2&\theta^1\\\theta^1&\theta^3\ema,\quad {\bf u}=\bma\theta^5\\\theta^4\ema,\quad {\bf p}=-\bma\theta^7\\\theta^6\ema,\quad p=\theta^8\\
  &{\bf c}=\bma\Om_1&\Om_2\\\Om_3&\Om_4\ema,\quad c=\Om_5,\\
  {\bf q}=&\bma\Om_6\\\Om_7\ema,\quad q=\Om_8,\quad {\bf v}=\bma\Om_9\\\Om_{10}\ema,\quad {\bf b}=\bma\Om_{11}&\Om_{12}\\\Om_{12}&\Om_{13}\ema,
\end{aligned}$$
beeing related to a coframe 
$(\theta^1,\theta^2,\dots,\theta^8,\Om_1,\Om_2,\dots,\Om_{13})$ of the left invariant forms on $\spg(3,\bbR)$. Due to the Maurer-Cartan relations
$$\der\omega+\omega\dz\omega=0,$$
the coframe forms satisfy the following \emph{exterior differential system} (EDS):
\be
\begin{aligned}
&\der\theta^1=\Om_1\dz\theta^1+\Om_2\dz\theta^2+\Om_3\dz\theta^3+\Om_4\dz\theta^1-\theta^4\dz\theta^7-\theta^5\dz\theta^6\\
&\der\theta^2=2\Om_1\dz\theta^2+2\Om_3\dz\theta^1-2\theta^5\dz\theta^7\\
&\der\theta^3=2\Om_2\dz\theta^1+2\Om_4\dz\theta^3-2\theta^4\dz\theta^6\\
&\der\theta^4=\Om_2\dz\theta^5+\Om_4\dz\theta^4+\Om_5\dz\theta^4+\Om_6\dz\theta^1+\Om_7\dz\theta^3+\theta^6\dz\theta^8\\
&\der\theta^5=\Om_1\dz\theta^5+\Om_3\dz\theta^4+\Om_5\dz\theta^5+\Om_6\dz\theta^2+\Om_7\dz\theta^1+\theta^7\dz\theta^8\\
&\der\theta^6=\Om_2\dz\theta^7+\Om_4\dz\theta^6-\Om_5\dz\theta^6-\Om_8\dz\theta^4-\Om_9\dz\theta^1-\Om_{10}\dz\theta^3\\
&\der\theta^7=\Om_1\dz\theta^7+\Om_3\dz\theta^6-\Om_5\dz\theta^7-\Om_8\dz\theta^5-\Om_9\dz\theta^2-\Om_{10}\dz\theta^1\\
&\der\theta^8=2\Om_5\dz\theta^8+2\Om_6\dz\theta^5+2\Om_7\dz\theta^4.
\end{aligned}\label{syu}
\ee
\be
\begin{aligned}
&\der\Om_1=\theta^1\dz\Om_{12}+\theta^2\dz\Om_{11}+\theta^5\dz\Om_9+\theta^7\dz\Om_6-\Om_2\dz\Om_3\\
&\der\Om_2=\theta^1\dz\Om_{11}+\theta^3\dz\Om_{12}+\theta^4\dz\Om_9+\theta^6\dz\Om_6-\Om_1\dz\Om_2-\Om_2\dz\Om_4\\
&\der\Om_3=\theta^1\dz\Om_{13}+\theta^2\dz\Om_{12}+\theta^5\dz\Om_{10}+\theta^7\dz\Om_7+\Om_1\dz\Om_3+\Om_3\dz\Om_4\\
&\der\Om_4=\theta^1\dz\Om_{12}+\theta^3\dz\Om_{13}+\theta^4\dz\Om_{10}+\theta^6\dz\Om_7+\Om_2\dz\Om_3\\
&\der\Om_5=\theta^4\dz\Om_{10}+\theta^5\dz\Om_{9}-\theta^6\dz\Om_{7}-\theta^7\dz\Om_6+\theta^8\dz\Om_8\\
&\der\Om_6=\theta^4\dz\Om_{12}+\theta^5\dz\Om_{11}+\theta^8\dz\Om_{9}-\Om_1\dz\Om_6-\Om_2\dz\Om_7+\Om_5\dz\Om_6\\
&\der\Om_7=\theta^4\dz\Om_{13}+\theta^5\dz\Om_{12}+\theta^8\dz\Om_{10}-\Om_3\dz\Om_6-\Om_4\dz\Om_7+\Om_5\dz\Om_7\\
&\der\Om_8=-2\theta^6\dz\Om_{10}-2\theta^7\dz\Om_{9}-2\Om_5\dz\Om_8\\
&\der\Om_9=-\theta^6\dz\Om_{12}-\theta^7\dz\Om_{11}-\Om_1\dz\Om_{9}-\Om_2\dz\Om_{10}-\Om_5\dz\Om_9-\Om_6\dz\Om_8\\
&\der\Om_{10}=-\theta^6\dz\Om_{13}-\theta^7\dz\Om_{12}-\Om_3\dz\Om_{9}-\Om_4\dz\Om_{10}-\Om_5\dz\Om_{10}-\Om_7\dz\Om_8\\
&\der\Om_{11}=-2\Om_1\dz\Om_{11}-2\Om_2\dz\Om_{12}-2\Om_6\dz\Om_9\\
&\der\Om_{12}=-\Om_1\dz\Om_{12}-\Om_2\dz\Om_{13}-\Om_3\dz\Om_{11}-\Om_4\dz\Om_{12}-\Om_6\dz\Om_{10}-\Om_7\dz\Om_9\\
&\der\Om_{13}=-2\Om_3\dz\Om_{12}-2\Om_4\dz\Om_{13}-2\Om_7\dz\Om_{10}.
\end{aligned}\label{syv}
\ee
We have the following proposition.
\begin{proposition}\label{parap}
  The group $\spg(3,\bbR)$ is naturally fibered $P\to\spg(3,\bbR)\to M$ over a homogeneous space $M=\spg(3,\bbR)/P$ where $P$ is a parabolic subgroup. In the standard 6-dimensional representation as in (\ref{sp3r}), the parabolic $P$ is the subgroup preserving the flag $$N_1=\Span\Big(\bma 1\\0\\0\\0\\0\\0\ema,\bma 0\\1\\0\\0\\0\\0\ema\Big)\subset N_2= \Span\Big(\bma 1\\0\\0\\0\\0\\0\ema,\bma 0\\1\\0\\0\\0\\0\ema,\bma 0\\0\\1\\0\\0\\0\ema\Big)\subset\bbR^6.$$
  The 8-dimensional homogeneous space $M$ is naturally equipped with an $\spg(3,\bbR)$ invariant rank 3 Monge distribution $\mathcal D$. i.e. with an $\spg(3,\bbR)$ parabolic Monge geometry in dimension 8.
\end{proposition}
\begin{proof}
  Given the group $\spg(3,\bbR)$ consider its Maurer-Cartan coframe $(\theta^1,\dots,\theta^8,\Om_1,$ $\dots,\Om_{13})$ satisfying the EDS (\ref{syu})-(\ref{syv}). Let $(X_1,\dots,X_8,Y_1,\dots,Y_{13})$ be a frame dual to the coframe $(\theta^1,\dots,\theta^8,\Om_1,\dots,\Om_{13})$. We have $X_i\hook\theta^j=\delta_i^j$, $Y_A\hook\Om_B=\delta_{AB}$, $X_i\hook\Omega_A=Y_A\hook\theta^i=0$.

  Observe that the system (\ref{syu}) guarantees that
  $$\der\theta^i\dz\theta^1\dz\theta^2\dz\theta^3\dz\theta^4\dz\theta^5\dz\theta^6\dz\theta^7\dz\theta^8=0,\quad \forall i=1,2,\dots, 8.$$
  This means that the anihilator of $\Span(\theta^1,\theta^2,\dots,\theta^8)$, which is a rank 13 distribution ${\mathcal P}=\Span(Y_1,Y_2,\dots,Y_{13})$ on $\spg(3,\bbR)$ is integrable. Thus $\spg(3,\bbR)$ is foliated by the leaves of the distribution $\mathcal P$. Morever each leave is isomorphic to a Lie group $P$ whose Maurer-Cartan forms satsify the system (\ref{syv}) with all 1-forms $(\theta^1,\theta^2,\dots,\theta^8)$ being equated to zero. We define $M$ to be the leaf space of the foliation given by $\mathcal P$. This shows the fibration property
  $P\to\spg(3,\bbR)\to M$: each pont of $M$ has a leaf of $\mathcal P$ as its fiber in $\spg(3,\bbR)$. 

  Now, still on $\spg(3,\bbR)$, define $\tilde{\mathcal H}=\Span(X_6,X_7)$ and $\tilde{\mathcal D}=\Span(X_6,X_7,X_8)$. Looking at the system (\ref{syu}) we see that $[Y_A,\tilde{\mathcal H}]\subset\tilde{\mathcal H}\,\mathrm{mod}\mathcal P$ and $[Y_A,\tilde{\mathcal D}]\subset\tilde{\mathcal D}\,\mathrm{mod}\mathcal P$. Thus $\tilde{\mathcal H}$ and $\tilde{\mathcal D}$ descend to a respective well defined distributions ${\mathcal H}$ (of rank 2) and ${\mathcal D}$ (of rank 3) on the quotient 8-dimensional space $M$. Since, according to (\ref{syu}), we have $[X_6,X_7]=0\,\mathrm{mod}\mathcal P$, $[\tilde{\mathcal D},\tilde{\mathcal D}]=\Span(X_8,X_7,X_6,X_5,X_4)\,\mathrm{mod}\mathcal P$, and $[\tilde{\mathcal D},\Span(X_8,X_7,X_6,X_5,X_4)]=\Span(X_8,X_7,X_6,X_5,X_4,X_3,X_2,X_1)\,\mathrm{mod}\mathcal P$, then the distribution $\mathcal H$ is integrable, ${\mathcal H}\subset\mathcal D$ and $\mathcal D$ has symbol $\mathfrak{m}$ as in (\ref{mgr}) everywhere on $M$. By the construction the distribution $\mathcal D$ on $M$ has $\spg(3,\bbR)$ symmetry. 
\end{proof}

\section{Monge system $z_{ij}=\frac{\partial^2f(\dot{x}^k)}{\partial\dot{x}{}^i\partial\dot{x}{}^j}$ with $i,j,k=1,2$}
We will now consider a subclass of $\spg(3,\bbR)$ parabolic Monge geometries in dimension 8, corresponding to distributions associated with certain generalizations of the flat system of ODEs (\ref{flatc3}). Starting in
this section till the end of the paper, we will concentrate on ODEs displayed in the title above. For the pourpose of this paper the system of ODEs 
$$z_{ij}=f_{ij}(\dot{x}^k),\quad f_{ij}=\frac{\partial^2f(\dot{x}^k)}{\partial\dot{x}{}^i\partial\dot{x}{}^j},\quad i,j,k=1,2,$$
or more explicitly
\be\begin{aligned}
&\dot{z}_{11}=f_{11}(\dot{x}^1,\dot{x}^2)\\
&\dot{z}_{12}=f_{12}(\dot{x}^1,\dot{x}^2)\\
&\dot{z}_{22}=f_{22}(\dot{x}^1,\dot{x}^2)
\end{aligned}\label{msc3}\ee
will be calledd a \emph{Monge system}. We will associate with it a corresponding EDS that will be a `curved version' of the EDS of a `flat' $\spg(3,\bbR)$ Monge geometry discussed in the
previous section. We will construct a system of differential invariants for the ODEs $z_{ij}=f_{ij}(\dot{x}^k)$, $i,j,k=1,2$, that, in particular, will enable us to determine when our Monge system (\ref{msc3}) is really different from the flat model given by (\ref{flatc3}).  
\subsection{The corresponding exterior differential system}
First, given the Momge system (\ref{msc3}),
we associate a 3-distribution ${\mathcal D}$ in dimension eight to it. This is defined as ${\mathcal D}=\Span(X,X_1,X_2)$ via the vector fields:
$$\begin{aligned}
&X=\partial_t+f_{11}\partial_{z_{11}}+f_{12}\partial_{z_{12}}+f_{22}\partial_{z_{22}}+\dot{x}^1\partial_{x_1}+\dot{x}^2\partial_{x_2}\\
&X_1=\partial_{\dot{x}^1}\\
&X_2=\partial_{\dot{x}^2}.
\end{aligned}
$$
We also have the derived vectors $Y_i$, $Z_{ij}$, defined via $Y_i=[X_i,X]$, $Z_{ij}=[Y_i,X_j]$. Explicitly they are given by:
$$\begin{aligned}
&Y_1=f_{11,1}\partial_{z_{11}}+f_{12,1}\partial_{z_{12}}+f_{22,1}\partial_{z_{22}}+\partial_{x_1}\\
&Y_2=f_{11,2}\partial_{z_{11}}+f_{12,22}\partial_{z_{12}}+f_{22,2}\partial_{z_{22}}+\partial_{x_2}\\
&Z_{11}=-f_{11,11}\partial_{z_{11}}-f_{12,11}\partial_{z_{12}}-f_{22,11}\partial_{z_{22}}\\
&Z_{12}=-f_{11,12}\partial_{z_{11}}-f_{12,12}\partial_{z_{12}}-f_{22,12}\partial_{z_{22}}\\
&Z_{22}=-f_{11,22}\partial_{z_{11}}-f_{12,22}\partial_{z_{12}}-f_{22,22}\partial_{z_{22}}.
\end{aligned}
$$

The case $z_{ij}=f_{ij}(\dot{x}^k)$, as more general than $z_{ij}=\dot{x}^i\dot{x}^j$, has more nonvanishing commutators than just $[X_i,X]$ and $[Y_i,X_j]$. In particular the commutator $[Z_{11},X_1]$ is in general nonzero.

More specifically, modulo the antisymmetry, the nonzero commutators among the vectors $(X,X_1,X_2,Y_1,Y_2,Z_{11}, Z_{12},Z_{22})$ are given by the formulae above, and 
\be\begin{aligned}
&[Z_{11},X_1]=f_{11,111}\partial_{z_{11}}+f_{12,111}\partial_{z_{12}}+f_{22,111}\partial_{z_{22}}\\
&[Z_{11},X_2]=[Z_{12},X_1]=f_{11,112}\partial_{z_{11}}+f_{12,112}\partial_{z_{12}}+f_{22,112}\partial_{z_{22}}\\
&[Z_{22},X_1]=[Z_{12},X_2]=f_{11,122}\partial_{z_{11}}+f_{12,122}\partial_{z_{12}}+f_{22,122}\partial_{z_{22}}\\
&[Z_{22},X_2]=f_{11,222}\partial_{z_{11}}+f_{12,222}\partial_{z_{12}}+f_{22,222}\partial_{z_{22}}.
\end{aligned}\label{cr}\ee

It should be noted that the vector fields $(X,X_1,X_2,Y_1,Y_2,Z_{11}, Z_{12},Z_{22})$ are defined by the distribution $\mathcal D$ modulo the nonsingular linear transformation
\be
\bma X\\X_1\\X_2\\Y_1\\Y_2\\Z_{11}\\Z_{12}\\Z_{22}\ema\to
\bma b_{11}&b_{12}&b_{13}&b_{14}&b_{15}&0&0&0\\
 b_{21}&b_{22}&b_{23}&b_{24}&b_{25}&0&0&0\\
b_{31}&b_{32}&b_{33}&b_{34}&b_{35}&0&0&0\\
b_{41}&b_{42}&b_{43}&b_{44}&b_{45}&0&0&0\\
b_{51}&b_{52}&b_{53}&b_{54}&b_{55}&0&0&0\\
b_{61}&b_{62}&b_{63}&b_{64}&b_{65}&b_{66}&b_{67}&b_{68}\\
b_{71}&b_{72}&b_{73}&b_{74}&b_{75}&b_{76}&b_{77}&b_{78}\\
b_{81}&b_{82}&b_{83}&b_{84}&b_{85}&b_{86}&b_{87}&b_{88}
\ema
\bma X\\X_1\\X_2\\Y_1\\Y_2\\Z_{11}\\Z_{12}\\Z_{22}\ema.\label{nb}
\ee

In the case, in which 
\be\det\bma
f_{11,11}&f_{12,11}&f_{22,11}\\
f_{11,12}&f_{12,12}&f_{22,12}\\
f_{11,22}&f_{12,22}&f_{22,22}
\ema\neq 0\label{ns}\ee
at each point of our 8-dimensional manifold $M$ parametrized by $(t,x^1,x^2,\dot{x}^1,\dot{x}^2,z_{11},z_{12},z_{22})$, we can use this transformation to simplify a bit the last set of the commutation relations above. 

Indeed, assuming (\ref{ns}), we are assured that the three vector fields $(\partial_{z_{11}},\partial_{z_{12}},\partial_{z_{22}})$ are at each point a linear combination of $(Z_{11},Z_{12},Z_{22})$ only. Thus the r.h. sides of (\ref{cr}) are also linear combinations of only $(Z_{11},Z_{12},Z_{22})$ at each point. A little of cosmetics using (\ref{nb}) leads to a new basis $(E_1,E_2,\dots,E_8)$ of vector fileds on $M$, respecting the structure of $\mathcal D$, given by:\\
\centerline{\begin{tabular}{lll}
$E_8=-X$, &$E_7=X_1$, &$E_6=X_2$,\\
$E_5=Y_1$, &$E_4=Y_2$,&\\
$E_3=\tfrac12 Z_{22}$, &$E_2=\tfrac12 Z_{11}$, &$E_1=\tfrac12 Z_{12}$.\end{tabular}}

In this basis the distribution is spanned by $(E_8,E_7,E_6)$ and the nonvanishing commutation relations of the basis vectors are:
$$\begin{aligned}
&[E_8,E_7]=E_5,\quad\quad [E_8,E_6]=E_4,\\
&[E_5,E_7]=2E_2,\quad\quad [E_4,E_6]=2E_3,\quad\quad [E_5,E_6]=[E_4,E_7]=E_1,
\end{aligned}$$
and 
$$\begin{aligned}
&[E_2,E_7]=-(\tilde{u}_3E_1+\tilde{u}_7E_2+\tilde{u}_{11}E_3)\\
&[E_1,E_6]=2[E_3,E_7]=-2(\tilde{u}_1 E_1+\tilde{u}_5E_2+\tilde{u}_9E_3)\\
&[E_1,E_7]=2[E_2,E_6]=-2(\tilde{u}_2E_1+\tilde{u}_6E_2+\tilde{u}_{10}E_3)\\
&[E_3,E_6]=-(\tilde{u}_4E_1+\tilde{u}_8E_2+\tilde{u}_{12}E_3),
\end{aligned}
$$
with the functions $(\tilde{u}_1,\tilde{u}_2,\dots,\tilde{u}_{12})$ on $M$ being uniquely determined by the Monge system $z_{ij}=f_{ij}(\dot{x}^k)$. They are expressible in terms of $f_{ij}(\dot{x}^k)$ and their partial derivatives up to the order three, but the formulae are not relevant here.

Introducing the basis of 1-forms $(\om^1,\om^2,\dots,\om^8)$ dual on $M$ to $(E_1,E_2,\dots,E_8)$, $E_\mu\hook\om^\nu=\delta^\nu_\mu$, we obtain the following structural equations for the distributions $\mathcal D$ related to any Monge system of the form $z_{ij}=f_{ij}(\dot{x}^k)$, $i,j,k=1,2$:
\be
\begin{aligned}
\der\om^1&=2 \tilde{u}_1 \om^1\dz\om^6+2 \tilde{u}_2 \om^1\dz\om^7+\tilde{u}_2 \om^2\dz\om^6+\tilde{u}_3 \om^2\dz\om^7+\tilde{u}_4 \om^3\dz\om^6+\tilde{u}_1 \om^3\dz\om^7\\
&-\om^4\dz\om^7-\om^5\dz\om^6,\\
\der\om^2&=2 \tilde{u}_5 \om^1\dz\om^6+2 \tilde{u}_6 \om^1\dz\om^7+\tilde{u}_6 \om^2\dz\om^6+\tilde{u}_7 \om^2\dz\om^7+\tilde{u}_8 \om^3\dz\om^6+\tilde{u}_5 \om^3\dz\om^7\\
&-2 \om^5\dz\om^7,\\
\der\om^3&=2 \tilde{u}_9 \om^1\dz\om^6+2 \tilde{u}_{10} \om^1\dz\om^7+\tilde{u}_{10} \om^2\dz\om^6+\tilde{u}_{11} \om^2\dz\om^7+\tilde{u}_{12} \om^3\dz\om^6+\tilde{u}_9 \om^3\dz\om^7\\
&-2 \om^4\dz\om^6,\\
\der\om^4&=\om^6\dz\om^8,\\
\der\om^5&=\om^7\dz\om^8,\\
\der\om^6&=0,\\
\der\om^7&=0,\\
\der\om^8&=0.\\
\end{aligned}\label{sysc3}
\ee
We summarize in the following proposition.
\begin{proposition}
Every Monge system (\ref{msc3})
defines a coframe of 1-forms $(\om^1,\om^2,$ $\dots,\om^8)$ satisfying the EDS (\ref{sysc3}) on an 8-dimensional manifold $M$ parametrized by $(t,x^1,x^2,\dot{x}^1,$ $\dot{x}^2,z_{11},z_{12},z_{22})$. The coefficients $(\tilde{u}_1,\tilde{u}_2,\dots,\tilde{u}_{12})$ are expressible in terms of the functions $f_{11}$, $f_{12}$ and $f_{22}$ and their dervatives up to the order three.

The coframe forms $(\om^1,\om^2,\dots,\om^8)$ are defined on $M$ by the Monge system (\ref{msc3}) up to the following linear transformation: 
\be
\bma \om^1\\\om^2\\\om^3\\\om^4\\\om^5\\\om^6\\\om^7\\\om^8\ema\to
\bma a_{1}&a_{2}&a_{3}&a_{4}&a_{5}&0&0&0\\
 a_{6}&a_{7}&a_{8}&a_{9}&a_{10}&0&0&0\\
a_{11}&a_{12}&a_{13}&a_{14}&a_{15}&0&0&0\\
a_{16}&a_{17}&a_{18}&a_{19}&a_{20}&0&0&0\\
a_{21}&a_{22}&a_{23}&a_{24}&a_{25}&0&0&0\\
a_{26}&a_{27}&a_{28}&a_{29}&a_{30}&a_{31}&a_{32}&a_{33}\\
a_{34}&a_{35}&a_{36}&a_{37}&a_{38}&a_{39}&a_{40}&a_{41}\\
a_{42}&a_{43}&a_{44}&a_{45}&a_{46}&a_{47}&a_{48}&a_{49}
\ema
\bma \om^1\\\om^2\\\om^3\\\om^4\\\om^5\\\om^6\\\om^7\\\om^8\ema,\label{nbf}
\ee
with arbitrary functions $a_A$, $A=1,2,\dots, 49$, such that the transformation is nonsingular.
\end{proposition}

\subsection{Local equivalence}
The above proposition motivates the following definition.

\begin{definition}\label{defmongec3}
Two Monge systems 
$$\begin{aligned}
&\dot{z}_{11}=f_{11}(\dot{x}^1,\dot{x}^2)\\
&\dot{z}_{12}=f_{12}(\dot{x}^1,\dot{x}^2)\\
&\dot{z}_{22}=f_{22}(\dot{x}^1,\dot{x}^2)
\end{aligned}$$ and
$$\begin{aligned}
&\dot{\bar{z}}_{11}=\bar{f}_{11}(\dot{\bar{x}}^1,\dot{\bar{x}}^2)\\
&\dot{\bar{z}}_{12}=\bar{f}_{12}(\dot{\bar{x}}^1,\dot{\bar{x}}^2)\\
&\dot{\bar{z}}_{22}=\bar{f}_{22}(\dot{\bar{x}}^1,\dot{\bar{x}}^2)
\end{aligned}$$
are locally equivalent iff there exists a local diffeomeorphism 
$$\phi~:~ M\to\bar{M}$$
of the corresponding manifolds $M$, with coordinates  $(t,x^1,x^2,\dot{x}^1,\dot{x}^2,z_{11},z_{12},z_{22})$, and $\bar{M}$, with coordinates  $(\bar{t},\bar{x}^1,\bar{x}^2,\dot{\bar{x}}^1,\dot{\bar{x}}^2,\bar{z}_{11},\bar{z}_{12},\bar{z}_{22})$, such that 
$$\phi^*\bma \bar{\om}^1\\\bar{\om}^2\\\bar{\om}^3\\\bar{\om}^4\\\bar{\om}^5\\\bar{\om}^6\\\bar{\om}^7\\\bar{\om}^8\ema\to
\bma a_{1}&a_{2}&a_{3}&a_{4}&a_{5}&0&0&0\\
 a_{6}&a_{7}&a_{8}&a_{9}&a_{10}&0&0&0\\
a_{11}&a_{12}&a_{13}&a_{14}&a_{15}&0&0&0\\
a_{16}&a_{17}&a_{18}&a_{19}&a_{20}&0&0&0\\
a_{21}&a_{22}&a_{23}&a_{24}&a_{25}&0&0&0\\
a_{26}&a_{27}&a_{28}&a_{29}&a_{30}&a_{31}&a_{32}&a_{33}\\
a_{34}&a_{35}&a_{36}&a_{37}&a_{38}&a_{39}&a_{40}&a_{41}\\
a_{42}&a_{43}&a_{44}&a_{45}&a_{46}&a_{47}&a_{48}&a_{49}
\ema
\bma \om^1\\\om^2\\\om^3\\\om^4\\\om^5\\\om^6\\\om^7\\\om^8\ema
$$
with some functions $a_{A}$, $A,B=1,2,\dots,49$ on $M$.
\end{definition}

When solving an equivalence problem for the system $z_{ij}=f_{ij}(\dot{x}^k)$, $i,j,k=1,2$, i.e. when building the system of local differential invariants characterizing it, we first calculate the derived flag for it.

We recall that given a Pfaffian system $\mathcal J$ of 1-forms ${\mathcal J}=(\omega^\mu)$, $\mu=1,2,\dots,N$, on a manifold $M$, the derived flag of it is a sequence of modules of 1-forms $\{{\mathcal J}_k\}$ such that
$${\mathcal J}={\mathcal J}_1\supset{\mathcal J}_2\supset{\mathcal J}_3\supset\dots,$$
and whose corresponding sets of generators $\Theta_k$ are defined inductively by:
$$\Theta_k=\{~\om^\mu\in\Theta_{k-1}~|~\der\om^\mu= 0~{\rm mod}~\Theta_{k-1}~\}.$$
 
We also recall that given a Lie group $G\subset\glg(N,\bbR)$, and two Pfaffian systems ${\mathcal J}$ and $\bar{\mathcal J}$ on two manifolds $M$ and $\bar{M}$, we say that they are locally 
$G$-equivalent iff there exists a diffeomorphism $\phi:M\to \bar{M}$, and a $G$-valued function $a:M\to G$ such that $\phi^*(\bar{\om}^\mu)=a^\mu_{~\nu}\om^\nu$. In particular, the diffeomorphism that realizes $G$-equivalence of two Pfaffian systems have to preserve their derived flags.

In this context our equivalence problem for the Monge equations defined in the previous section, is a $G=\glg(5,\bbR)$ euivalence problem of Pfaffian systems generated by $N=5$ one-forms $\omega^\mu$: 
$$\Theta_1=(\om^1,\om^2,\om^3,\om^4,\om^5),$$ 
on 8-dimensional manifolds, whose differentials satisfy the system (\ref{sysc3}), with some functions $(\tilde{u}_1,\tilde{u}_2,\dots,\tilde{u}_{12})$, and three additional 1-forms $(\om^6,\om^7,\om^8)$, such that $\om^1\dz\om^2\dz\dots\dz\om^8\neq 0$. 

It follows that the derived flag of this system is
$${\mathcal J}_1\supset{\mathcal J}_2\supset{\mathcal J}_3\equiv\{0\},$$
with the corresponding generators
$$\Theta_1=(\om^1,\om^2,\om^3,\om^4,\om^5)$$
and
$$\Theta_2=(\om^1,\om^2,\om^3),\quad\quad \Theta_3=\{0\}.$$

This, in particular shows, that when solving an equivalence problem for such systems, or what is the same, for our Monge equations  $z_{ij}=f_{ij}(\dot{x}^k)$, $i,j,k=1,2$, we can restrict transformations (\ref{nbf}) to 
\be
\bma \om^1\\\om^2\\\om^3\\\om^4\\\om^5\\\om^6\\\om^7\\\om^8\ema\to
\bma a_{1}&a_{2}&a_{3}&0&0&0&0&0\\
 a_{6}&a_{7}&a_{8}&0&0&0&0&0\\
a_{11}&a_{12}&a_{13}&0&0&0&0&0\\
a_{16}&a_{17}&a_{18}&a_{19}&a_{20}&0&0&0\\
a_{21}&a_{22}&a_{23}&a_{24}&a_{25}&0&0&0\\
a_{26}&a_{27}&a_{28}&a_{29}&a_{30}&a_{31}&a_{32}&a_{33}\\
a_{34}&a_{35}&a_{36}&a_{37}&a_{38}&a_{39}&a_{40}&a_{41}\\
a_{42}&a_{43}&a_{44}&a_{45}&a_{46}&a_{47}&a_{48}&a_{49}
\ema
\bma \om^1\\\om^2\\\om^3\\\om^4\\\om^5\\\om^6\\\om^7\\\om^8\ema,\label{nbff}
\ee
without loss of generality. This simplifies the equivalence problem enormously.

We summarize our considerations so far in the following reformulation of Definition \ref{defmongec3}:
\begin{definition}\label{defmongec3p}
Let $G$ be a Lie \emph{subgroup} of $\glg(8,\bbR)$ consisting of $8\times 8$ real matrices $A=(A^i_{~j})$ as below:
$$G=\{~(A^i_{~j})=\bma a_{1}&a_{2}&a_{3}&0&0&0&0&0\\
 a_{6}&a_{7}&a_{8}&0&0&0&0&0\\
a_{11}&a_{12}&a_{13}&0&0&0&0&0\\
a_{16}&a_{17}&a_{18}&a_{19}&a_{20}&0&0&0\\
a_{21}&a_{22}&a_{23}&a_{24}&a_{25}&0&0&0\\
a_{26}&a_{27}&a_{28}&a_{29}&a_{30}&a_{31}&a_{32}&a_{33}\\
a_{34}&a_{35}&a_{36}&a_{37}&a_{38}&a_{39}&a_{40}&a_{41}\\
a_{42}&a_{43}&a_{44}&a_{45}&a_{46}&a_{47}&a_{48}&a_{49}
\ema,~\det( A^i_{~j})\neq 0~\}.$$
Consider two Monge systems  $$z_{ij}=f_{ij}(\dot{x}^k)=z_{ji}\quad\quad{\rm and}\quad\quad  \bar{z}_{ij}=\bar{f}_{ij}(\dot{\bar{x}}^k)=\bar{z}_{ji},\quad\quad i,j,k=1,2,$$
and the corresponding manifolds $M$, with coordinates  $(t,x^1,x^2,\dot{x}^1,\dot{x}^2,z_{11},z_{12},z_{22})$, and $\bar{M}$, with coordinates  $(\bar{t},\bar{x}^1,\bar{x}^2,\dot{\bar{x}}^1,\dot{\bar{x}}^2,\bar{z}_{11},\bar{z}_{12},\bar{z}_{22})$. Let $(\om^1,\om^2,\dots,\om^8)$ and $(\bar{\om}^1,\bar{\om}^2,\dots,\bar{\om}^8)$ be the corresponding coframes which satisfy the system as in (\ref{sysc3}) on $M$ and, respectively, on $\bar{M}$.  

We say that the two Monge systems are \emph{locally equivalent} iff there exists a local diffeomeorphism 
$$\phi~:~ M\to\bar{M},$$ 
such that 
$$\phi^*(\bar{\om}^i)=A^i_{~j}\om^j,$$
with $A=(A^i_{~j}): M\to G$ being a $G$-valued function on $M$. 
\end{definition}
\subsection{The reduced structure group}
We now pass to consider the most general of systems (\ref{sysc3}). Thus we extend the coframe 1-forms $(\om^1,\om^2,\dots,\om^8)$ satisfying (\ref{sysc3}) to the most general forms from the equivalence introduced in the Definition \ref{defmongec3p}. We denoted these extended system of 1-forms by $(\theta^1,\theta^2,\dots,\theta^8)$. We have:
\be\theta^i=A^i_{~j}\om^j,\label{theta}\ee
where $(\om^1,\om^2,\dots,\om^8)$ are linearly independent and satisfy (\ref{sysc3}) and $(A^i_{~j})$ is the most general element of $G$. Thus we have a local bundle $G\to G\times M\to M$, and the forms  $(\theta^1,\theta^2,\dots,\theta^8)$ are uniquely defined there. We also have
\be
\der\theta^i=\der A^i_{~j} A^{-1}\phantom{}^j_{~k}\dz\theta^k-\tfrac12 A^i_{~l}c^l_{~pq} A^{-1}\phantom{}^p_{~j} A^{-1}\phantom{}^q_{~k}\theta^j\dz\theta^k,\label{strc}
\ee
where the structure functions $c^i_{~jk}$ are defined by (\ref{sysc3}) and
$$\der\om^i=-\tfrac12 c^i_{~jk}\om^j\dz\om^k.$$
In the following we will be interested in those transformed coefficients
$$\gamma^i_{~jk}=A^i_{~l}c^l_{~pq} A^{-1}\phantom{}^p_{~j} A^{-1}\phantom{}^q_{~k}$$ which can not be absorbed via transformation
\be
\der A^i_{~j} A^{-1}\phantom{}^j_{~k}~\to ~\der A^i_{~j} A^{-1}\phantom{}^j_{~k}+v^i_{~k}\theta^k\label{absoc3}\ee
in the first term $\der A^i_{~j} A^{-1}\phantom{}^j_{~k}\dz\theta^k$ on the r.h.s. of (\ref{strc}).  

Because of the block structure of the matrices $A$ forming the group $H$ 
\begin{itemize}
\item[1)] for all $i=1,2,3$ we have:
$$\der\theta^i\dz\theta^1\dz\theta^2\dz\theta^3=-\tfrac12\gamma^i_{~jk}\theta^j\dz\theta^k\dz\theta^1\dz\theta^2\dz\theta^3,$$
and
\item[2)] for all $i=1,2,3,4,5$ we have:
$$\der\theta^i\dz\theta^1\dz\theta^2\dz\theta^3\dz\theta^4\dz\theta^5=-\tfrac12\gamma^i_{~jk}\theta^j\dz\theta^k\dz\theta^1\dz\theta^2\dz\theta^3\dz\theta^4\dz\theta^5.$$
\end{itemize}
Point 2) above means that the structure functions $\gamma^i_{~jk}$ with $i=1,2,3,4,5$ and $j<k=6,7,8$ can \emph{not} be absorbed by the transformation (\ref{absoc3}). Similarly, Point 1) above means that the  structure functions $\gamma^i_{~jk}$ with $i=1,2,3$ and $j<k=4,5,6,7,8$ can \emph{not} be absorbed by the transformation (\ref{absoc3}). Taking these two pieces of information together we arrive at  
\begin{proposition}
All structure functions $\gamma^i_{~jk}=A^i_{~l}c^l_{~pq} A^{-1}\phantom{}^p_{~j} A^{-1}\phantom{}^q_{~k}$ appearing in (\ref{strc}) that can not be absorbed by transformation (\ref{absoc3}) correspond to the triples $(ijk)$ with $i=1,2,3$ and $j<k=4,5,6,7,8$, and to the triples $(ijk)$ with $i=4,5$ and $j<k=6,7,8$.
\end{proposition} 

This observation enables us to significantly reduce the structure group $G$ of our equivalence problem: since these $\gamma^i_{jk}$'s can not be absorbed, and since they are \emph{linearly} related to the structurual functions $c^i_{~jk}$ of the original EDS, one can try to \emph{normalize} them in such a way that we have
$$\gamma^i_{~jk}=c^i_{~jk}$$
for all the triples $(ijk)$ mentioned in the above proposition.
These conditions are algebraic constraints on the coefficients of the matrices $A\in G$. They are obviously not contradictory, since $A=Id$ satisfies them. Remarkably, we have the following proposition.
\begin{proposition}\label{pr36}
The most general matrix $A\in G$ for which $\gamma^i_{~jk}=c^i_{~jk}$ occures for all the triples $(ijk)$ such that $i=1,2,3$ and $j<k=4,5,6,7,8$, and $i=4,5$ and $j<k=6,7,8$ is given by
$$\begin{aligned}
&(A^i_{~j})=\\
&\bma a_{29}a_{38}a_{49}(1+b_5b_6)b_7^2&a_{29}a_{38}a_{49}b_5b_7^2&a_{29}a_{38}a_{49}b_6b_7^2&0&0&0&0&0\\
 2a_{38}^2a_{49}b_6b_7^2&a_{38}^2a_{49}b_7^2&a_{38}^2a_{49}b_6^2b_7^2&0&0&0&0&0\\
2a_{29}^2a_{49}b_5b_7^2&a_{29}^2a_{49}b_5^2b_7^2&a_{29}^2a_{49}b_7^2&0&0&0&0&0\\
a_{16}&a_{17}&a_{18}&a_{29}a_{49}b_7&a_{29}a_{49}b_5b_7&0&0&0\\
a_{21}&a_{22}&a_{23}&a_{38}a_{49}b_6b_7&a_{38}a_{49}b_7&0&0&0\\
a_{26}&a_{27}&a_{28}&a_{29}&a_{29}b_5&a_{29}b_7&a_{29}b_5b_7&0\\
a_{34}&a_{35}&a_{36}&a_{38}b_6&a_{38}&a_{38}b_6b_7&a_{38}b_7&0\\
a_{42}&a_{43}&a_{44}&a_{45}&a_{46}&0&0&a_{49}
\ema,
\end{aligned}$$
with 23 coefficients $a_{16}$, $a_{17}$, $a_{18}$, $a_{21}$, $a_{22}$, $a_{23}$, $a_{26}$, $a_{27}$, $a_{28}$, $a_{29}$, $a_{34}$, $a_{35}$, $a_{36}$, $a_{38}$, $a_{42}$, $a_{43}$, $a_{44}$, $a_{45}$, $a_{46}$, $a_{49}$, $b_5$, $b_6$ and $b_7$ constrained by the open condition 
$$\det(A^i_{~j})=a_{29}^5a_{38}^5a_{49}^6(1-b_5b_6)^5b_7^{10}\neq 0.$$
\end{proposition}
\begin{proof}
The proof is a pure algebra, so we skip it.
\end{proof}

The 23-parameter subgroup $G_0$ of $G$ consisting of all the inevrtible matrices $A=(A^i_{~j})$ as in Proposition \ref{pr36}, is the reduced structure group of our equivalence problem. In other words, we have just demonstrated that two Monge systems  $z_{ij}=f_{ij}(\dot{x}^k)=z_{ji}$ and $\bar{z}_{ij}=\bar{f}_{ij}(\dot{\bar{x}}^k)=\bar{z}_{ji}$, $i,j,k=1,2,$ are locally equivalent if their corresponding coframe forms $(\om^i)$ and $(\bar{\om}^i)$ are related via diffeomorphism $\phi$ satisfying $\phi^*(\bar{\om}^i)=A^i_{~j}\om^j$, with a function $A$ having valued in $G_0$. 
\subsection{The flat case}\label{flsec}
When dealing with $G$-structures, as our Monge $G_0$ structure considered in the previous Section,  it is always constructive to analyse the simplest case first. In our situation the simplest case is undoubtly the case of the system (\ref{sysc3}) with all the structural functions $(\tilde{u}_1,\tilde{u}_2,\dots,\tilde{u}_{12})$ identically vanishing. This corresponds to the Monge system (\ref{flatc3}) which has $\spg(3,\bbR)$ as the maximal group of local symmetries. Here we will show that there is a more general class of Monge systems (\ref{sysc3}) which also has this property. For the reasons which will be clear later we will now focus on the system (\ref{sysc3}) with
\be
\tilde{u}_3=\tfrac32\tilde{u}_{10},\quad \tilde{u}_4=\tfrac32\tilde{u}_5,\quad \tilde{u}_7=3(2\tilde{u}_2-\tilde{u}_9),\quad \tilde{u}_8=\tilde{u}_{11}=0,\quad \tilde{u}_{12}=3(2\tilde{u}_1-\tilde{u}_6).\label{flco}\ee

For further convenience we rename the six free variables $(\tilde{u}_1,\tilde{u}_2,\tilde{u}_5,\tilde{u}_6,\tilde{u}_9,\tilde{u}_{10})$ to $(u_1,u_2,u_3,u_4,u_5,u_6)$ via:
$$\tilde{u}_1=u_1,\quad \tilde{u}_2=u_2,\quad \tilde{u}_5=\tfrac23u_4,\quad \tilde{u}_6=u_5,\quad \tilde{u}_9=\tfrac13(6u_2-u_6),\quad \tilde{u}_{10}=\tfrac23u_3,$$
obtaining the following system of 1-forms $(\om^1,\om^2,\dots,\om^8)$:
\be
\begin{aligned}
\der\om^1&=2 u_1 \om^1\dz\om^6+2 u_2 \om^1\dz\om^7+u_2 \om^2\dz\om^6+u_3 \om^2\dz\om^7\\&+u_4 \om^3\dz\om^6+u_1 \om^3\dz\om^7-\om^4\dz\om^7-\om^5\dz\om^6,\\
\der\om^2&=\tfrac43 u_4 \om^1\dz\om^6+2 u_5 \om^1\dz\om^7+u_5 \om^2\dz\om^6+u_6 \om^2\dz\om^7\\&+\tfrac23 u_4 \om^3\dz\om^7-2 \om^5\dz\om^7,\\
\der\om^3&=\tfrac23(6 u_2-u_6) \om^1\dz\om^6+\tfrac43u_3 \om^1\dz\om^7+\tfrac23u_3 \om^2\dz\om^6\\&+3(2u_1-u_5)\om^3\dz\om^6+\tfrac13 (6u_2-u_6) \om^3\dz\om^7-2 \om^4\dz\om^6,\\
\der\om^4&=\om^6\dz\om^8,\\
\der\om^5&=\om^7\dz\om^8,\\
\der\om^6&=0,\\
\der\om^7&=0,\\
\der\om^8&=0.\\
\end{aligned}\label{sysc3fl}
\ee
We need the closure of this EDS. A short calculation yields the following proposition.
\begin{proposition}
If the eight linearly independent 1-forms $(\om^1,\om^2,\dots,\om^8)$ satisfy the system (\ref{sysc3fl}) we have:
\be
\begin{aligned}
\der u_1&=\tfrac13(12 u_1^2-2 u_2 u_4-9 u_1 u_5+u_4 u_6+3 v_2)\om^6\\&+\tfrac19(18 u_1 u_2+4 u_3 u_4-18 u_2 u_5-3 u_1 u_6+6 v_3)\om^7,\\
\der u_2&=\tfrac19 (36 u_1 u_2+4 u_3 u_4-27 u_2 u_5-6 u_1 u_6+6 v_3)\om^6\\&+\tfrac13 (-6 u_2^2+4 u_1 u_3-3 u_3 u_5+3 u_2 u_6+3 v_1)\om^7,\\
\der u_3&=v_1 \om^6+\tfrac13u_3(6u_2+u_6)\om^7,\\
\der u_4&=u_4(4u_1-u_5)\om^6+v_2 \om^7,\\
\der u_5&=\tfrac19(18 u_1 u_5-9 u_5^2+4 u_4 u_6+6 v_2)\om^6\\&+\tfrac19(8 u_3 u_4-18 u_2 u_5+9 v_3)\om^7,\\
\der u_6&=v_3\om^6+\tfrac13(-36 u_2^2-6 u_3 u_5+12 u_2 u_6+u_6^2+12 v_1)\om^7,\\
\end{aligned}\label{dcs1}
\ee
with functions $v_1$, $v_2$ $v_3$ satisfying:
\be
\begin{aligned}
\der v_1&=\tfrac23(72 u_1 u_2 u_3 + 8 u_3^2 u_4 - 54 u_2 u_3 u_5 - 
 12 u_1 u_3 u_6 + 18 u_2 v_1 + 3 u_6 v_1 + 
 15 u_3 v_3)\om^6\\&+\tfrac19(36 u_1 u_2^2 + 4 u_2 u_3 u_4 - 36 u_2^2 u_5 + 
 3 u_1 u_3 u_5 - 3 u_3 u_5^2 - 12 u_1 u_2 u_6 + 
 6 u_2 u_5 u_6\\& + u_1 u_6^2 + 3 u_5 v_1 + u_3 v_2 + 
 6 u_2 v_3 - u_6 v_3)\om^7,\\
\der v_2&=\tfrac19 (72 u_1 u_2 u_4+8 u_3 u_4^2-54 u_2 u_4 u_5-12 u_1 u_4 u_6+36 u_1 v_2-9 u_5 v_2+15 u_4 v_3)\om^6\\&+\tfrac23 (18 u_2^2 u_4 - 12 u_1 u_3 u_4 + 27 u_2 u_5^2 - 
 9 u_2 u_4 u_6 + u_4 u_6^2 - 3 u_4 v_1 - 
 18 u_2 v_2 \\&+ 3 u_6 v_2 - 9 u_5 v_3)\om^7,
 \end{aligned}\label{dcs2}\ee
$$\begin{aligned}
\der v_3&=-\tfrac23(12 u_2^2 u_4 - 18 u_1 u_2 u_5 + 18 u_2 u_5^2 - 
 6 u_2 u_4 u_6 + 3 u_1 u_5 u_6 - 2 u_4 v_1 - 
 6 u_2 v_2\\& - 3 u_1 v_3 - 3 u_5 v_3) \om^6\\&-\tfrac29(108 u_2^2 u_5 - 18 u_1 u_3 u_5 + 27 u_3 u_5^2 - 
 4 u_3 u_4 u_6 - 18 u_2 u_5 u_6 - 27 u_5 v_1 - 
 6 u_3 v_2\\& - 18 u_2 v_3 - 3 u_6 v_3)\om^7.\end{aligned}
$$
If the above diferentials are satisfied, then for all $i=1,2,\dots 8$ $$\der^2\om^i= 0,$$  and for all $i=1,2,\dots 6$ $$\der^2u_i= 0,$$ 
and for all $i=1,2,3$ $$\der^2v_i= 0,$$ 
i.e. the system is totally closed.  
\end{proposition}

Using the relations from the above proposition we get the following result.
\begin{proposition}\label{flapr}
For all choices of the functions $(u_1, u_2,\dots,u_6,v_1,v_2,v_3)$ satisfying the Monge relations (\ref{sysc3fl}), and as a consequence (\ref{dcs1}), (\ref{dcs2}), the corresponding Monge system is locally equivalent to the flat Monge system with $u_1=u_2=\dots=u_6=v_1=v_2=v_3=0$. 
\end{proposition}

\begin{proof}
  Consider a Monge system (\ref{sysc3fl}) defined on an 8-dimensional manifold $M$. We will show that, there exists a choice of functions $\bar{A}{}^i{}_j$ on $M$, which form a matrix $(\bar{A}{}^i{}_j)$ with values in the reduced structure group $G_0$, as in Proposition \ref{pr36}), and a choice of thirteen 1-forms $(\bar{\Omega}{}_1,\bar{\Omega}{}_2,\dots,\bar{\Omega}{}_{13})$ such that the 21 forms $(\bar{\theta}{}^1,\bar{\theta}{}^2,\dots \bar{\theta}{}^8,\bar{\Omega}{}_1,\bar{\Omega}{}_2,\dots,\bar{\Omega}{}_{13})$, with $\bar{\theta}{}^i=\bar{A}{}^i{}_j\omega^j$, satisfy the \emph{flat} EDS (\ref{syu})-(\ref{syv}) on $M$. These forms will be used to define forms satisfying the flat EDS (\ref{syu})-(\ref{syv}) on $\spg(3,\bbR)$.

  The explicit realization of this program is obtained by first defining the matrix $(\bar{A}{}^i{}_j)$ on $M$. This is given by:
   $$(\bar{A}{}^i{}_j)=\bma 1&0&0&0&0&0&0&0\\
  0&1&0&0&0&0&0&0\\
  0&0&1&0&0&0&0&0\\
  \tfrac16(u_6-6u_2)&-\tfrac13u_3&\tfrac12(u_5-2u_1)&1&0&0&0&0\\
  -\tfrac12u_5&-\tfrac12u_6&-\tfrac13u_4&0&1&0&0&0\\
  \tfrac13(u_6-3u_2)&-\tfrac13u_3&u_5-u_1&1&0&1&0&0\\
  0&0&-\tfrac13u_4&0&1&0&1&0\\
  s_1&s_2&s_3&0&0&0&0&1
  \ema,
  $$
where we have introduced the following abreviations:
  $$\begin{aligned}
    &s_1=\tfrac16(6u_2u_5+u_5u_6-2v_3)\\
    &s_2=\tfrac{1}{36}(36u_2^2+12u_3u_5-12u_2u_6+u_6^2-12v_1)\\
    &s_3=\tfrac{1}{36}(24u_2u_4+9u_5^2-4u_4u_6-12v_2), 
\end{aligned}.$$
Note that the matrix $(\bar{A}{}^i{}_j)$ above has values in $G_0$. Thus if we now take a coframe $(\omega^1,\omega^2,\omega^3,\omega^4,\omega^5,\omega^6,\omega^7,\omega^8)$ satisfying
(\ref{sysc3fl}), and a new coframe $(\bar{\theta}{}^1,\bar{\theta}{}^2,\dots \bar{\theta}{}^8)$ on $M$ given by
  $$\bar{\theta}{}^i=\bar{A}{}^i{}_j\om^j, \quad \mathrm{for}\quad i=1,2,\dots 8,$$
then they both define the same Monge geometry.

We now supplement the eight 1-forms $(\bar{\theta}{}^1,\bar{\theta}{}^2,\dots \bar{\theta}{}^8)$ with thirteen 1-forms $(\bar{\Omega}{}_1,\bar{\Omega}{}_2,\dots,\bar{\Omega}{}_{13})$ given on $M$ by:
  $$\bar{\Om}{}_a=B_{ai}\om^i,\quad \mathrm{for}\quad a=1,2,\dots 13,$$
  with
$$\tiny{\begin{aligned}
  (&B_{ai})=\\&\bma
  0&0&\tfrac14u_5^2&\tfrac12u_5&\tfrac12u_6&-\tfrac12u_5&-\tfrac13u_6&0\\
  B_{21}&\tfrac19u_3u_6&\tfrac16u_5(3u_2-2u_6)&u_2&\tfrac23u_3&\tfrac16(u_6-6u_2)&-\tfrac23u_3&0\\
  0&-\tfrac16u_5u_6&\tfrac13u_4u_5&\tfrac23u_4&u_5&-\tfrac23u_4&-\tfrac12u_5&0\\
  B_{41}&\tfrac{1}{36}(6u_2-u_6)u_6&B_{43}&\tfrac12(4u_1-u_5)&\tfrac16(6u_2-u_6)&u_5-2u_1&\tfrac16(u_6-6u_2)&0\\
  0&0&0&\tfrac12u_5&\tfrac12u_6&0&0&1\\
  B_{61}&B_{62}&B_{63}&B_{64}&B_{65}&B_{66}&B_{67}&\tfrac16u_6\\
  B_{71}&B_{72}&B_{73}&B_{74}&B_{75}&B_{76}&B_{77}&\tfrac12u_5\\
  \tfrac16u_5u_6&\tfrac{1}{36}u_6^2&\tfrac14u_5^2&0&0&-u_5&-\tfrac13 u_6&-1\\
  B_{91}&B_{92}&\tfrac{1}{24}u_5^2u_6&0&B_{95}&\tfrac16u_5u_6&B_{97}&-\tfrac16u_6\\
  B_{10,1}&B_{10,2}&\tfrac{1}{24}u_5(8 u_2 u_4 + 3 u_5^2)&\tfrac23u_2u_4&B_{10,5}&\tfrac19u_4(-6 u_2 + u_6)&B_{10,7}&-\tfrac12u_5\\
  B_{11,1}&B_{11,2}&B_{11,3}&B_{11,4}&B_{11,5}&\tfrac16u_6(-4 u_2 u_5 + v_3)&B_{11,7}&-\tfrac{1}{36}u_6^2\\
  B_{12,1}&B_{12,2}&B_{12,3}&B_{12,4}&B_{12,5}&B_{12,6}&B_{12,7}&-\tfrac{1}{12}u_5u_6\\
   B_{13,1}&B_{13,2}&B_{13,3}&B_{13,4}&B_{13,5}&B_{13,6}&B_{13,7}&-\tfrac{1}{4}u_5^2\\
 \ema.
\end{aligned}}$$
Here we have used the following abreviations:
$$\begin{aligned}
  &B_{21}=\tfrac12(4u_3u_5+2u_2u_6-u_6^2),\quad B_{41}=\tfrac16(3u_2u_5+2u_1u_6-3u_5u_6)\\
  &B_{43}=\tfrac{1}{36}(36u_1u_5-27u_5^2-4u_4u_6)\\
  &B_{61}=\tfrac{1}{36}(3u_2u_5u_6-18u_2^2u_5-6u_3u_5^2-u_5u_6^2+6u_5v_1+u_6v_3)\\
  &B_{62}=\tfrac{1}{216}u_6(12u_2u_6-36u_2^2-12u_3u_5-u_6^2+12v_1), \quad B_{63}=\tfrac{1}{24}u_5(2v_3-6u_2u_5-u_5u_6)\\
  &B_{64}=\tfrac16(v_3-3u_2u_5-2u_5u_6), \quad B_{65}=\tfrac{1}{12}(4v_1-12u_2^2-4u_3u_5+4u_2u_6-u_6^2)\\
  &B_{66}=B_{77}=\tfrac{1}{12}(6u_2u_5+u_5u_6-2v_3), \quad B_{67}=\tfrac{1}{36}(36u_2^2+12u_3u_5-12u_2u_6+u_6^2-12v_1)\\
  &B_{71}=\tfrac{1}{108}(-27 u_2 u_5^2 - 12 u_2 u_4 u_6 - 9 u_5^2 u_6 + 
  2 u_4 u_6^2 + 6 u_6 v_2 + 9 u_5 v_3)\\
  &B_{72}=-\tfrac{1}{72}u_6 (6 u_2 u_5 + u_5 u_6 - 2 v_3)\\
  &B_{73}=-\tfrac{1}{2}u_5 B_{76},\quad\quad  B_{76}=\tfrac{1}{36}(24 u_2 u_4 + 9 u_5^2 - 4 u_4 u_6 - 12 v_2)\\
  &B_{74}=\tfrac16(-4 u_2 u_4 - 3 u_5^2 + 2 v_2),\quad B_{75}=\tfrac16(-3 u_2 u_5 - 2 u_5 u_6 + v_3)\\
  &B_{91}=-\tfrac{1}{36}u_5 (18 u_2^2 + 6 u_3 u_5 - 6 u_2 u_6 - u_6^2 - 6 v_1)\\
  &B_{92}=-\tfrac{1}{216}u_6 (36 u_2^2 + 12 u_3 u_5 - 12 u_2 u_6 - u_6^2 - 
  12 v_1)\\
  &B_{95}=\tfrac13(-3 u_2^2 - u_3 u_5 + u_2 u_6 + v_1),\quad B_{97}=\tfrac{1}{36} (36 u_2^2 + 12 u_3 u_5 - 12 u_2 u_6 + u_6^2 - 12 v_1)\\
  &B_{10,1}=\tfrac{1}{36}(8 u_3 u_4 u_5 - 18 u_2 u_5^2 + 4 u_2 u_4 u_6 + 
  3 u_5^2 u_6 + 6 u_5 v_3)\\
  &B_{10,2}=\tfrac{1}{216}u_6 (16 u_3 u_4 - 36 u_2 u_5 + 3 u_5 u_6 + 12 v_3),\quad B_{10,5}=\tfrac19(4 u_3 u_4 - 9 u_2 u_5 + 3 v_3)\\
  &B_{10,7}=\tfrac{1}{18}(-8 u_3 u_4 + 18 u_2 u_5 + 3 u_5 u_6 - 6 v_3)\\
  &B_{11,1}=\tfrac{1}{216}(180 u_2^2 u_5 u_6 + 72 u_3 u_5^2 u_6 - 
  36 u_2 u_5 u_6^2 - 5 u_5 u_6^3 - 60 u_5 u_6 v_1 - 
  6 u_6^2 v_3)\\
  &B_{11,2}=\tfrac{1}{648}(180 u_2^2 u_6^2 + 72 u_3 u_5 u_6^2 - 60 u_2 u_6^3 - 
  u_6^4 - 60 u_6^2 v_1)\\
  &B_{11,3}=\tfrac{1}{48}(16 u_2 u_5^2 u_6 - 3 u_5^2 u_6^2 - 4 u_5 u_6 v_3),\quad B_{11,4}=\tfrac{1}{12}(8 u_2 u_5 u_6 + u_5 u_6^2 - 2 u_6 v_3)\\
  &B_{11,5}=\tfrac{1}{36}(60 u_2^2 u_6 + 24 u_3 u_5 u_6 - 20 u_2 u_6^2 + 
  u_6^3 - 20 u_6 v_1)
   \end{aligned}$$
$$\begin{aligned}
  &B_{11,7}=\tfrac{1}{108}(-180 u_2^2 u_6 - 72 u_3 u_5 u_6 + 60 u_2 u_6^2 - 
  u_6^3 + 60 u_6 v_1)\\
  &B_{12,1}=\tfrac{1}{648}(1134 u_2^2 u_5^2 + 270 u_3 u_5^3 - 
  48 u_3 u_4 u_5 u_6 - 54 u_2 u_5^2 u_6 + 
  24 u_2 u_4 u_6^2 - 27 u_5^2 u_6^2 -\\&\quad 16 u_4 u_6^3 - 
  378 u_5^2 v_1 - 24 u_6^2 v_2 - 126 u_5 u_6 v_3)\\
 &B_{12,2}=\tfrac{1}{1296}(756 u_2^2 u_5 u_6 + 180 u_3 u_5^2 u_6 - 
 32 u_3 u_4 u_6^2 - 72 u_2 u_5 u_6^2 - 15 u_5 u_6^3 - 
 252 u_5 u_6 v_1 - 60 u_6^2 v_3)\\
 &B_{12,3}=\tfrac{1}{144}(36 u_2 u_5^3 + 16 u_2 u_4 u_5 u_6 + 3 u_5^3 u_6 - 
 12 u_4 u_5 u_6^2 - 16 u_5 u_6 v_2 - 24 u_5^2 v_3)\\
 &B_{12,4}=\tfrac{1}{54}(27 u_2 u_5^2 + 12 u_2 u_4 u_6 + 18 u_5^2 u_6 - 
 5 u_4 u_6^2 - 12 u_6 v_2 - 18 u_5 v_3)\\
 &B_{12,5}=\tfrac{1}{108} (378 u_2^2 u_5 + 90 u_3 u_5^2 - 16 u_3 u_4 u_6 - 
  36 u_2 u_5 u_6 + 9 u_5 u_6^2 - 126 u_5 v_1 - 
  30 u_6 v_3)\\
  &B_{12,6}=\tfrac{1}{18}(-9 u_2 u_5^2 - 4 u_2 u_4 u_6 - 3 u_5^2 u_6 + 
  2 u_4 u_6^2 + 4 u_6 v_2 + 6 u_5 v_3)\\
  &B_{12,7}=\tfrac{1}{216}(-756 u_2^2 u_5 - 180 u_3 u_5^2 + 32 u_3 u_4 u_6 + 
  72 u_2 u_5 u_6 - 3 u_5 u_6^2 + 252 u_5 v_1 + 
  60 u_6 v_3)\\
  &B_{13,1}=\tfrac{1}{216}(144 u_2^2 u_4 u_5 - 24 u_3 u_4 u_5^2 + 
  216 u_2 u_5^3 + 9 u_5^3 u_6 - 24 u_4 u_5 u_6^2 - 
  48 u_4 u_5 v_1 - 36 u_5 u_6 v_2 -\\&\quad\quad 72 u_5^2 v_3 - 
  4 u_4 u_6 v_3)\\
  &B_{13,2}=\tfrac{1}{1296}(288 u_2^2 u_4 u_6 - 48 u_3 u_4 u_5 u_6 + 
  432 u_2 u_5^2 u_6 - 48 u_2 u_4 u_6^2 - 
  27 u_5^2 u_6^2 - 16 u_4 u_6^3 -\\&\quad\quad 96 u_4 u_6 v_1 - 
  24 u_6^2 v_2 - 144 u_5 u_6 v_3)\\
  &B_{13,3}=\tfrac{1}{1296}(432 u_2 u_4 u_5^2 + 81 u_5^4 - 144 u_4 u_5^2 u_6 - 
  16 u_4^2 u_6^2 - 432 u_5^2 v_2 - 72 u_4 u_5 v_3)\\
  &B_{13,4}=\tfrac{1}{18}(12 u_2 u_4 u_5 + 9 u_5^3 - 12 u_5 v_2 - 2 u_4 v_3)\\
  &B_{13,5}=\tfrac{1}{54}(72 u_2^2 u_4 - 12 u_3 u_4 u_5 + 108 u_2 u_5^2 - 
  12 u_2 u_4 u_6 + 27 u_5^2 u_6 - 4 u_4 u_6^2 - 
  24 u_4 v_1 -\\&\quad\quad 6 u_6 v_2 - 36 u_5 v_3)\\
  &B_{13,6}=\tfrac{1}{36}(-24 u_2 u_4 u_5 - 9 u_5^3 + 4 u_4 u_5 u_6 + 
  24 u_5 v_2 + 4 u_4 v_3)\\
  &B_{13,7}=\tfrac{1}{108} (-144 u_2^2 u_4 + 24 u_3 u_4 u_5 - 216 u_2 u_5^2 + 
  24 u_2 u_4 u_6 - 27 u_5^2 u_6 + 8 u_4 u_6^2 +\\&\quad\quad 
  48 u_4 v_1 + 12 u_6 v_2 + 72 u_5 v_3).
     \end{aligned}$$
It is a matter of checking that the so defined system of forms $(\bar{\theta}{}^1,\bar{\theta}{}^2,\dots \bar{\theta}{}^8,\bar{\Omega}{}_1,\bar{\Omega}{}_2,$ $\dots,\bar{\Omega}{}_{13})$ satisfies (\ref{syu})-(\ref{syv}).

One can worry that the flat system (\ref{syu})-(\ref{syv}) is fullfiled by the forms $(\bar{\theta}{}^1,\bar{\theta}{}^2,\dots \bar{\theta}{}^8,$ $\bar{\Omega}{}_1,\bar{\Omega}{}_2,$ $\dots,\bar{\Omega}{}_{13})$ only on $M$.

To see that these forms define also the flat Monge system (\ref{syu})-(\ref{syv}) on $\spg(3,\bbR)$, consider the parabolic subgroup $P$ in $\spg(3,\bbR)$ as defined in Proposition \ref{parap}. This has the Lie agebra $\mathfrak{p}$ consisting of matrices $X$ as in (\ref{sp3r}) with ${\bf a}={\bf p}=p={\bf u}=0$. Let $b$ be a general element from $P$, $b\in P$, in the representation corresponding to the considered representation of $\mathfrak{p}$. Define:
$$
  \bar{\omega}=\bma[c|c||c|c]
  \begin{matrix}\bar{\Om}{}_1&\bar{\Om}{}_2\\\bar{\Om}{}_3&\bar{\Om}{}_4\end{matrix}&\begin{matrix}\bar{\Om}{}_6\\\bar{\Om}{}_7\end{matrix}&\begin{matrix}\bar{\Om}{}_{11}&\bar{\Om}{}_{12}\\\bar{\Om}{}_{12}&\bar{\Om}{}_{13}\end{matrix}&\begin{matrix}\bar{\Om}{}_9\\\bar{\Om}{}_{10}\end{matrix}\\
  \cmidrule(lr){1-4}
  \begin{matrix}\bar{\theta}{}^7&\bar{\theta}{}^6\end{matrix}&\bar{\Om}{}_5&\begin{matrix}\bar{\Om}{}_9&\bar{\Om}{}_{10}\end{matrix}&\bar{\Om}{}_8\\
  \cmidrule(lr){1-4}\morecmidrules\cmidrule(lr){1-4}
  \begin{matrix}\bar{\theta}{}^2&\bar{\theta}{}^1\\\bar{\theta}{}^1&\bar{\theta}{}^3\end{matrix}&\begin{matrix}\bar{\theta}{}^5\\\bar{\theta}{}^4\end{matrix}&\begin{matrix}-\bar{\Om}{}_1&-\bar{\Om}{}_3\\-\bar{\Om}{}_2&-\bar{\Om}{}_4\end{matrix}&\begin{matrix}-\bar{\theta}{}^7\\-\bar{\theta}{}^6\end{matrix}\\
  \cmidrule(lr){1-4}
  \begin{matrix}\bar{\theta}{}^5&\bar{\theta}{}^4\end{matrix}&\bar{\theta}{}^8&\begin{matrix}-\bar{\Om}{}_6&-\bar{\Om}{}_7\end{matrix}&-\bar{\Om}{}_5
      \ema.
      $$
      Then the $\spa(3,\bbR)$-valued form
      $$\omega=b\bar{\omega}b^{-1}-(\der b) b^{-1},$$
      is defined on $P\times M$, satisfies $\der\omega+\omega\dz\omega=0$, and defines a \emph{coframe} of 1-forms $(\theta^1,\theta^2,\dots,\theta^8,\Om_1,$ $\Om_2,\dots,\Om_{13})$ on $P\times M$ via:
$$\omega=\bma[c|c||c|c]
  \begin{matrix}\Om_1&\Om_2\\\Om_3&\Om_4\end{matrix}&\begin{matrix}\Om_6\\\Om_7\end{matrix}&\begin{matrix}\Om_{11}&\Om_{12}\\\Om_{12}&\Om_{13}\end{matrix}&\begin{matrix}\Om_9\\\Om_{10}\end{matrix}\\
  \cmidrule(lr){1-4}
  \begin{matrix}\theta^7&\theta^6\end{matrix}&\Om_5&\begin{matrix}\Om_9&\Om_{10}\end{matrix}&\Om_8\\
  \cmidrule(lr){1-4}\morecmidrules\cmidrule(lr){1-4}
  \begin{matrix}\theta^2&\theta^1\\\theta^1&\theta^3\end{matrix}&\begin{matrix}\theta^5\\\theta^4\end{matrix}&\begin{matrix}-\Om_1&-\Om_3\\-\Om_2&-\Om_4\end{matrix}&\begin{matrix}-\theta^7\\-\theta^6\end{matrix}\\
  \cmidrule(lr){1-4}
  \begin{matrix}\theta^5&\theta^4\end{matrix}&\theta^8&\begin{matrix}-\Om_6&-\Om_7\end{matrix}&-\Om_5
      \ema.
$$
Hence, locally, $P\times M$ is isomorphic to $\spg(3,\bbR)$. Moreover, the 21-linearly indeendent forms $(\theta^1,\theta^2,\dots,\theta^8,\Om_1,\Om_2,\dots,\Om_{13})$ satisfy the flat Monge system on it. 
\end{proof}

\subsection{The main invariant}
We now pass to the general case of the system (\ref{sysc3}) and we will find its first invariant w.r.t. the action of the reduced group $G_0$. 

The procedure here consists in trying to make absorptions and normalizations for the invariant forms $\theta^i=A^i_{~j}\om^j$, with $\om^j$ as in (\ref{sysc3}) and $(A^i_{~j})=A\in G_0$, in such a way that the resulting system is as close to the system (\ref{syu}) for the flat model, as possible. This means, that we will only modify the system (\ref{syu}) by adding to its right hand sides only specific `horizontal' terms of the form $t^i{}_{jk}\theta^j\dz\theta^k$. 

This, in modern language, can be phrased as follows:
\begin{itemize}
\item we try to find a Cartan bundle $P\to{\mathcal G}^{21}\to M$ assocaiated with a Cartan geometry of type $\big(\spg(3,\bbR),P\big)$ proper to our Monge systems
\item we do it by trying to lift the invariant coframe forms $(\theta^i)$ to the searched Cartan bundle $P\to{\mathcal G}^{21}\to M$ in such a way that the lifts are a part of the basis $(\theta^1,\theta^2,\dots,\theta^8,\Om_1,\Om_2,\dots,\Om_{13})$ for an \emph{apropriately normalized} $\spa(3,\bbR)$ Cartan connection $\omega$
\item this connection, in particular, should be such that in the case when $\kappa=\der\omega+\omega\dz\omega=0$ it coincides with the Maurer-Cartan form $\omega$ considered in the previous section
\item the phrases `specific horizontal terms' and `apropriate normalization' above mean, that we want to force the invariant forms $\theta^i=A^i_{~j}\om^j$, $i=1,2,3$ to satisfy the equations below:
\end{itemize}  
\be
\begin{aligned}
\der\theta^1&=\Om_1\dz\theta^1+\Om_2\dz\theta^2+\Om_3\dz\theta^3+\Om_4\dz\theta^1-\theta^4\dz\theta^7-\theta^5\dz\theta^6\\&+\sum_{i=1}^3\sum_{j=i+1}^8t^1_{~ij}\theta^i\dz\theta^j\\
\der\theta^2&=2\Om_1\dz\theta^2+2\Om_3\dz\theta^1-2\theta^5\dz\theta^7+\sum_{i=1}^3\sum_{j=i+1}^8t^2_{~ij}\theta^i\dz\theta^j\\
\der\theta^3&=2\Om_2\dz\theta^1+2\Om_4\dz\theta^3-2\theta^4\dz\theta^6+\sum_{i=1}^3\sum_{j=i+1}^8t^3_{~ij}\theta^i\dz\theta^j,
\end{aligned}\label{syuf3}
\ee
with some functional coefficients $t^i_{~jk}$ - the torsions.

Note that in these equations we allowed only such torsions, which do not destroy terms quadratic in $\theta$s which are preserved byt the group $G_0$. It is why the first sums in 
these expressions have the upper limit equal to 3.

Now we have the following result, which is a pure calculation:

\begin{proposition}\label{proe}
Let
$$
\begin{aligned}
&E_2=\der\theta^2-\Big(2\Om_1\dz\theta^2+2\Om_3\dz\theta^1-2\theta^5\dz\theta^7+\sum_{i=1}^3\sum_{j=i+1}^8t^2_{~ij}\theta^i\dz\theta^j\Big)\\
&E_3=\der\theta^3-\Big(2\Om_2\dz\theta^1+2\Om_4\dz\theta^3-2\theta^4\dz\theta^6+\sum_{i=1}^3\sum_{j=i+1}^8t^3_{~ij}\theta^i\dz\theta^j\Big).
\end{aligned}
$$
Then 
$$E_2\dz\theta^1\dz\theta^2\dz\theta^4\dz\theta^5\dz\theta^7\dz\theta^8= 0\quad
{\rm iff}\quad 
t^2_{~36}=\frac{a_{38}^2T}{a_{29}^3b_7(b_5b_6-1)^3},$$
and
$$E_3\dz\theta^1\dz\theta^3\dz\theta^4\dz\theta^5\dz\theta^6\dz\theta^8= 0\quad 
{\rm iff}\quad t^2_{~27}=\frac{a_{29}^2}{a_{38}^3b_7(b_5b_6-1)^3}T.$$
Here 
$$\begin{aligned}
T=&-\tilde{u}_8+(3\tilde{u}_5-2\tilde{u}_4)b_6+(6\tilde{u}_1-3\tilde{u}_6-\tilde{u}_{12})b_6^2\\&+
(\tilde{u}_7-6\tilde{u}_2+3\tilde{u}_9)b_6^3+(2\tilde{u}_3-3\tilde{u}_{10})b_6^4+\tilde{u}_{11}b_6^5\end{aligned}.$$
\end{proposition}
\begin{proof}
By inspection.
\end{proof}

We call $T$ the \emph{main invariant of the system} (\ref{sysc3}). Properties of $T$, e.g. properties of its roots, 
when it is considered as a polynomial in variable $b_6$, 
$$T(b_6)=\al_0(\tilde{u}_i)+\al_1(\tilde{u}_i)b_6+\al_2(\tilde{u}_i)b_6^2+\al_3(\tilde{u}_i)b_6^3+\al_4(\tilde{u}_i)b_6^4+\al_5(\tilde{u}_i)b_6^5,$$
yield $G_0$-invariant information about the system 
(\ref{sysc3}). 

Being more specific, note that the identical vanishing of the \emph{quintic} polynomial $T=T(b_6)$, i.e. $T(b_6)= 0$, or what is the same $\al_0(\tilde{u}_i) =\al_1(\tilde{u}_i)=\dots=\al_5(\tilde{u}_i)=0$, is equivalent to the equations
(\ref{flco}) considered in Section \ref{flsec}. This means that when 
performing the equivalence method for systems (\ref{sysc3}) we have two, very different, cases:
\begin{itemize}
\item[A)] either $T$ is identically zero, $T= 0$,
\item[B)] or $T\neq 0$.
\end{itemize}
In the first case, by virtue of Proposition \ref{flapr} modulo the equivalence there is only one such system - the flat system (\ref{sysc3fl}) with all $u_1=u_2=\dots=u_6=0$. In the second case we have more possibilities. They can be distinguished by enumerating roots, and the multiplicities of the roots, of the polynomial $T=T(b_6)$. 

At this stage the following comment is in order. In view of the Tanaka theory \cite{Tan}, it is clear that we can continue our the search of the Cartan connection for the most general Monge system (\ref{sysc3}), and that with the normalizations taken into account so far we will eventually reduce the structure group from $G_0$ to the arabolic $P\spg(3,\bbR)$ obtaining a Cartan connection $\omega$ for a Cartan geometry of type $(\spg(3,\bbR),P)$. But it is also clear that from now on we can split our analysys of Monge systems into the cases that A) the system is flat, B) it is nonflat, with this second case still splitting into subcases corresponding to various root multiplicities of the main invariant $T=T(b_6)$. If we only to enumerate noninvariant classes of the Monge systems, it may more appropraite to find a connection description for each of these subcases separately.

Since the aim of this paper is to describe the Monge systems for which the main invariant $T=T(b_6)$ \emph{has a root with multiplicity equal to five}, we stop the procedure of constructing the $\spa(3,\bbR)$ Cartan connection for all Monge systems (\ref{sysc3}) here, and will only use the fact that such connection exists.

\section{Systems with the main invariant having fivefold multiple root}
In this Section we consider systems (\ref{sysc3}) whose main invariant $T=T(b_6)$ has a \emph{single} root. For such systems, without loss of generality, we
can assume that all the structural functions 
$\tilde{u}_i$ in (\ref{sysc3}) identically vanish, except the structural function $\tilde{u}_{11}\neq 0$. For simplicity of notation we denote this structural function by $u$,
$$\tilde{u}_{11}=u\neq 0.$$
\subsection{Integration of the EDS in this case}
We have the following proposition.
\begin{proposition}
The system (\ref{sysc3}) in which $u_{11}=u\neq 0$ and all other $\tilde{u}_i=0$, satisfies:
\be
\begin{aligned}
\der\om^1&=-\om^4\dz\om^7-\om^5\dz\om^6,\\
\der\om^2&=-2 \om^5\dz\om^7,\\
\der\om^3&=u \om^2\dz\om^7-2 \om^4\dz\om^6,\\
\der\om^4&=\om^6\dz\om^8,\\
\der\om^5&=\om^7\dz\om^8,\\
\der\om^6&=0,\\
\der\om^7&=0,\\
\der\om^8&=0.\\
\end{aligned}\label{sysc3u}
\ee
with
\be\begin{aligned}
&\der u=v_1\om^2+v_2\om^7,\\
&\der v_1=w_1\om^2+w_2\om^7,\\
&\der v_2=w_2\om^2+2v_1\om^5+w_3\om^7,\\
&\der w_1=z_1\om^2+z_2\om^7,\\
&\der w_2=z_2\om^2+2w_1\om^5+z_3\om^7,\\
&\der w_3=z_3\om^2+4w_2\om^5+z_4\om^7+2v_1\om^8,\\
&\der z_1=z_5\om^2+z_6\om^7,\\
&\der z_2=z_6\om^2+2z_1\om^5+z_7\om^7,\\
&\der z_3=z_7\om^2+4z_2\om^5+z_8\om^7+2w_1\om^8,\\
&\der z_4=z_8\om^2+6z_3\om^5+z_9\om^7+6w_2\om^8,
\end{aligned}\label{sysc3u1}\ee
and some functions $v_1,v_2,w_1,w_2,w_3,z_1,z_2,z_3,z_4,z_5,z_6,z_7,z_8,z_9.$
\end{proposition} 
\begin{proof}
The proof consists on writing the differential of $u$, $\der u$, in the basis of forms $\om^1,\om^2,\dots,\om^8$, then using conditions $\der^2u= 0$, and then by decomposing the differentials of the remaining two functions $v_1$ and $v_2$ onto the basis of $\omega^i$s, and requiring $\der^2v_1=\der^2v_2= 0$. The system from the proposition is implied by, and guarantees that $\der^2u=\der^2v_1=\der^2v_2= 0$.   
\end{proof}

Let us now define the basic invariants $I_1,I_2,I_3$ and $R$ of the system (\ref{sysc3u}). These are
\be
\begin{aligned}
&I_1=5v_2^2-4uw_3\\
&I_2=v_1\\
&I_3=3v_1v_2-2uw_2\\
&R=4v_1^2-3uw_1.
\end{aligned}\label{in1}
\ee
The reason for the term `basic invariants' for these quantities will be evident later.

Using $I_1,I_2,I_3$ and $R$ we also define 
\be
\begin{aligned}
&t^4_{~12}=-\frac{a_{38}^3}{8u^4a_{29}^4}I_3\\
&t^4_{~15}=\frac{a_{38}}{6u^2a_{29}^2}I_2+\frac{a_{38}^4}{16u^4a_{29}^4}I_1\\
&t^4_{~17}=\frac{a_{38}^4}{16u^4a_{29}^4}I_1\\
&r_{212}=-\frac{a_{38}^2}{9u^4a_{29}^4}R.
\end{aligned}\label{in2}
\ee
It is instructive to \emph{integrate} the EDS (\ref{sysc3u}).
\begin{theorem}\label{exa}
  Every Monge structure with the main invariant $T(b_6)$ having root of multiplicity five admits local coordinates $(y_1,y_2,\dots,y_8)$ in which the forms $(\om^1,\om^2,\dots,\om^8)$ satisfying (\ref{sysc3u}) read:
  \be\begin{aligned}
   \om^1=&\der y_8-y_5\der y_2+y_3\der y_4,\\
   \om^2=&\der y_7-2y_4\der y_2,\\
   \om^3=&\der y_6+2y_3\der y_5-\tfrac12 h_{22}\der y_7,\\
   \om^4=&\der y_5-y_1\der y_3,\\
   \om^5=&\der y_4-y_1\der y_2,\\
   \om^6=&\der y_3,\\
   \om^7=&\der y_2,\\
   \om^8=&\der y_1.
    \end{aligned}
  \label{systemu}
  \ee
  Here the function $h=h(y_2,y_7)$ is a differentiable function of its variables $y_2,y_7$. It is related to the variable $u$ via:
  $$u=\tfrac12 h_{222}\neq 0,$$
  where as usual we have used $h_{22}=\frac{\partial^2 h}{\partial y_2\partial y_2}$ and $h_{222}=\frac{\partial^3 h}{\partial y_2\partial y_2\partial y_2}$.

  In this coordinate system the derived quantities $v_1,v_2,w_1,w_2,w_3,z_1,z_2,z_3,z_4,z_5,z_6,z_7,z_8,z_9$ read:
  \be\begin{aligned}
  v_1=&u_7=\tfrac12 h_{2227},\\
  v_2=&Du=\tfrac12(h_{2222}+2y_4h_{2227}),\\
  w_1=&u_{77}=\tfrac12h_{22277},\\
  w_2=&(Du)_7=\tfrac12(h_{22227}+2y_4h_{22277}),\\
  w_3=&D^2u+2y_1u_7=\tfrac12(h_{22222}+4y_4h_{22227}+4y_4^2h_{22277}+2y_1h_{2227}),\\
  z_1=&u_{777}=\tfrac12h_{222777},\\
  z_2=&(Du)_{77}=\tfrac12(h_{222277}+2y_4h_{222777}),\\
  z_3=&\big(D^2u+2y_1u_7\big)_7=\tfrac12(h_{222227}+4y_4h_{222277}+4y_4^2h_{222777}+2y_1h_{22277}),\\
  z_4=&D^3u+2y_1(2Du+u)_7=\tfrac12(h_{222222}+6y_4h_{222227}+12y_4^2h_{222277}+6y_1h_{22227}+\\&8y_4^3h_{222777}+12y_1y_4h_{22277}),\\
  z_5=&u_{7777}=\tfrac12h_{2227777},\\
  z_6=&(Du)_{777}=\tfrac12(h_{2222777}+2y_4h_{2227777}),\\
  z_7=&\big(D^2u+2y_1u_7\big)_{77}=\tfrac12(h_{2222277}+4y_4h_{2222777}+4y_4^2h_{2227777}+2y_1h_{222777}),\\
  z_8=&\big(D^3u+6y_1(Du)_7\big)_7=\tfrac12(h_{2222227}+6y_4h_{2222277}+12y_4^2h_{2222777}+6y_1h_{222277}+\\&8y_4^3h_{2227777}+12y_1y_4h_{222777}),\\
  z_9=&D^4u+\big(10y_2D^2u+2y_1Du+12y_1^2u_7\big)_7=\tfrac12(h_{2222222}+8y_4h_{2222227}+\\&24y_4^2h_{2222277}+12y_1h_{222227}+32y_4^3h_{2222777}+48y_1y_4h_{222277}+16y_4^4h_{2227777}+\\&48y_1y_4^2h_{222777}+12y_1^2h_{22277}).
  \end{aligned}\ee
  The basic invariants $I_1,I_2,I_3$ and $R$ are
  $$\begin{aligned}
  I_1=&\tfrac14\Big(5h_{2222}^2-4h_{222}h_{22222}-2y_1h_{222}h_{2227}+4y_4(5h_{2227}h_{2222}-4h_{222}h_{22227})+\\&4y_4^2(5h_{2227}^2-4h_{222}h_{22277})\Big),\\
  I_2=&\tfrac12h_{2227},\\
  I_3=&\tfrac14\Big(3h_{2227}h_{2222}-2h_{222}h_{22227}+2y_4(3h_{2227}^2-2h_{222}h_{22277})\Big),\\
  R=&\tfrac14(4h_{2227}^2-3h_{222}h_{22277}).
  \end{aligned}$$
\end{theorem}
\begin{proof}
  The last three equations (\ref{sysc3u}) ensure that there exists functions $(y_1,y_2,y_3)$ on $M$ such that
  $$\om^8=\der y_1,\quad\om^7=\der y_2,\quad\om^6=\der y_3.$$
  Because of the independence condition
  \be \om^1\dz\om^2\dz\dots\om^8\neq 0,\label{ic}\ee
  we have $$\der y_1\dz\der y_2\dz\der y_3\neq 0.$$
  Using this we see that the 5th equation (\ref{sysc3u}) becomes:
  $$\der\om^5+\der y_1\dz\der y_2=0.$$
  This can be written as:
  $$\der\Big( \om^5+y_1\der y_2\Big)=0,$$
  so that one can find a function $y_4$ on $M$ such that
  $$\om^5+y_1\der y_2=\der y_4.$$
  This means that
  $$\om^5=\der y_4-y_1\der y_2,$$
  and
  $$\der y_1\dz\der y_2\dz\der y_3\dz\der y_4\neq 0,$$
  due to the independence condition (\ref{ic}).

  Similarly the 4th equation (\ref{sysc3u}) becomes:
  $$\der\om^4+\der y_1\dz\der y_3=0.$$
  Writing this as
  $$\der\Big( \om^4+y_1\der y_3\Big)=0,$$
 we get another function $y_5$ on $M$ such that
 $$\om^4=\der y_5-y_1\der y_3,$$
  and
  $$\der y_1\dz\der y_2\dz\der y_3\dz\der y_4\dz\der y_5\neq 0.$$
  In a similar way we integrate the second and the first equations (\ref{sysc3u}) which become:
  $$0=\der\om^2-2\der y_2\dz\der y_4=\der\Big(\om^2+2y_4\der y_2\Big),$$
  and
  $$0=\der\om^1-\der y_2\dz\der y_5-\der y_3\dz\der y_4=\der\Big(\om^1+y_5\der y_2-y_3\der y_4\Big).$$
  We thus obtain
  $$\begin{aligned}
    \om^2=&\der y_7-2y_4\der y_2,\\
    \om^1=&\der y_8-y_5\der y_2+y_3\der y_4,
    \end{aligned}$$
  with
  $$\der y_1\dz\der y_2\dz\der y_3\dz\der y_4\dz\der y_5\dz\der y_7\dz\der y_8\neq 0.$$
  The last equation to be solved is the third equation (\ref{sysc3u}). In view of the above this read:
  \be\der\om^3-2\der y_3\dz\der y_5+u\der y_2\dz\der y_7=0.\label{o3}\ee
  This equation has a differential consequence. Indeed, applying $\der$ on both sides of it, we get that:
  $$\der u\dz\der y_2\dz\der y_7=0.$$
  This means that the function $u$ defining the Monge structure is functionally dependent on the functions $y_2$ and $y_7$, 
  $$u=u(y_2,y_7).$$
  For further convenience, without loss of generality, we take
  $$u=\frac12h_{222},\quad\mathrm{with}\quad h=h(y_2,y_7).$$ 
  With this notation equation (\ref{o3}), becomes
  $$0=\der\Big(\om^3-2y_3\der y_5+\tfrac12 h_{22}\der y_7\Big).$$
  Thus we have
  $$\om^3=\der y_6+2y_3\der y_5-\tfrac12 h_{22}\der y_7,$$
  with $$\der y_1\dz\der y_2\dz\der y_3\dz\der y_4\dz\der y_5\dz\der y_6\dz\der y_7\dz\der y_8\neq 0.$$
  This solves the equations (\ref{sysc3u}) introducing, as a byproduct, a coordinate system $(y_1,y_2,\dots,y_8)$ on $M$.

  The rest of the proof consists in the use of the definitions (\ref{sysc3u1})-(\ref{in1}) and the rules of differentiation.
\end{proof}

Although in Theorem \ref{exa} we integrated the Monge EDS (\ref{sysc3u}) and found local representation of the Monge system coframe $(\om^1,\om^2,\dots,\om^8)$ it is not particularly easy to find explicit expressions for the corresponding Monge equations (\ref{msc3}). There is however a subclass of Monge systems (\ref{systemu}), given by a particular choice of a class of functions $h=h(y_2,y_7)$, for which the corresponding Monge equations can be written explicitly. This is described by the following proposition.
\begin{proposition}\label{exu1}
Monge system (\ref{systemu}) in which the function $h=h(y_2,y_7)$ does not depend on the coordinate $y_7$, $h_7=0$, is equivalent to the system defined in terms of Monge ODEs
 $$\dot{z}_{11}=\dot{x}_1^2,\quad \dot{z}_{12}=\dot{x}_1\dot{x}_2,\quad z_{22}=\dot{x}_2^2+h(\dot{x}_1),\quad\quad h_{\dot{x}_1\dot{x}_1\dot{x}_1}\neq 0.$$
  This system has $u\neq 0$ iff $h^{(3)}\neq 0$. Its basic invariants are:
  $$I_1=\tfrac14\Big(5h^{(4)}{}^2-4h^{(3)}h^{(5)}\Big),\quad I_2=I_3=R=0.$$
  \end{proposition}
\begin{proof}
  We first perform a change of coordinates
  $$(y_1,y_2,\dots,y_8)\mapsto (z_{11},z_{12},z_{22},x_1,x_2,\dot{x}_1,\dot{x}_2,t)$$
  given by:
  $$\begin{aligned}
    &y_1=t,\quad y_2=\dot{x}_1,\quad y_3=\dot{x_2},\quad y_4=-x_1+t\dot{x}_1,\\&y_5=-x_2+t\dot{x}_2,\quad y_6=z_{22}+t\dot{x}_1h'-t\dot{x}_2^2- th-x_1h',\\& y_7=z_{11}-2x_1\dot{x}_1+t\dot{x}_1^2,\quad y_8=z_{12}-x_2\dot{x}_1.
  \end{aligned}$$
  This brings the coframe 1-forms (\ref{systemu}) into a rather ugly form. However, it turns out that, once rewritten in this way, coframe $(\om^1,\om^2,\dots,\om^8)$ can be easilly transformed by a Monge equivalence of the form
  $$
\bma \om^1\\\om^2\\\om^3\\\om^4\\\om^5\\\om^6\\\om^7\\\om^8\ema\to\bma \tilde{\om}{}^1\\\tilde{\om}{}^2\\\tilde{\om}{}^3\\\tilde{\om}{}^4\\\tilde{\om}{}^5\\\tilde{\om}{}^6\\\tilde{\om}{}^7\\\tilde{\om}{}^8\ema=
\bma a_{1}&a_{2}&a_{3}&a_4&a_5&0&0&0\\
 a_{6}&a_{7}&a_{8}&a_9&a_{10}&0&0&0\\
a_{11}&a_{12}&a_{13}&a_{14}&a_{15}&0&0&0\\
a_{16}&a_{17}&a_{18}&a_{19}&a_{20}&0&0&0\\
a_{21}&a_{22}&a_{23}&a_{24}&a_{25}&0&0&0\\
a_{26}&a_{27}&a_{28}&a_{29}&a_{30}&a_{31}&a_{32}&a_{33}\\
a_{34}&a_{35}&a_{36}&a_{37}&a_{38}&a_{39}&a_{40}&a_{41}\\
a_{42}&a_{43}&a_{44}&a_{45}&a_{46}&a_{47}&a_{48}&a_{49}
\ema
\bma \om^1\\\om^2\\\om^3\\\om^4\\\om^5\\\om^6\\\om^7\\\om^8\ema,
$$
into a coframe in the canonical Monge ODE form. Indded we have:
$$
\begin{aligned}
  \tilde{\om}^1 &=\om^1-\dot{x}_1\om^4-\dot{x}_2\om^5=\der z_{12}-\dot{x}_1\dot{x}_2\der t,\\
\tilde{\om}^2& =\om^2-2\dot{x}_1\om^5=\der z_{11}-\dot{x}_1^2\der t,\\
\tilde{\om}^3& =\om^3+\tfrac12h''(\dot{x}_1)\om^2-2\dot{x}_2\om^4-h'(\dot{x}_1)\om^5=\der z_{22}-\big(\dot{x}_2^2+h(\dot{x}_1)\big)\der t,\\
\tilde{\om}^4&=-\om^4=
\der x_2-\dot{x}_2\der t,\\
\tilde{\om}^5&=-\om^5=
\der x_1-\dot{x}_1\der t,\\
\tilde{\om}^6&=\om^6=\der \dot{x}_2,\\
\tilde{\om}^7&=\om^7=\der \dot{x}_1,\\
\tilde{\om}^8&=\om^8=\der t.
  \end{aligned}
$$
Thus the new coordinates $(z_{11},z_{12},z_{22},x_1,x_2,\dot{x}_1,\dot{x}_2,t)$ on $M$ are the canonical `jet' coordinates in which the coframe $(\tilde{\om}{}^2,\tilde{\om}{}^2,\dots,\tilde{\om}{}^8)$, Monge equivalent to the initial $(\om^1,\om^2,\dots,\om^8)$, defines the Monge ODEs:
$$z_{11}=\dot{x}_1^2,\quad z_{12}=\dot{x}_1\dot{x}_2,\quad z_{22}=\dot{x}_2^2+h(\dot{x}_1).$$
Note that we need $h^{(3)}(\dot{x}_1)\neq 0$ to have $u\neq 0$.
\end{proof}

\subsection{Reduction and the principal connection. Invariants $I_1$, $I_2$, $I_3$ and $R$}
The idea now is to describe the considered geometries in terms of a certain subset ${\mathcal G}^{11}$ of the Cartan bundle $P\to{\mathcal G}^{21}\to M$ assocaiated with a Cartan geometry of type $\big(\spg(3,\bbR),P\big)$ proper to our Monge systems. The subset ${\mathcal G}^{11}$ will be defined by taking apropriate well defined \emph{normalizations} of the corresponding $\spa(3,\bbR)$ Cartan connection. Contrary to the flat model, these normalizations are possible, since we will now know that certain curvature functions for this connection are \emph{not} zero, due to $u\neq 0$.

We have the fundamental theorem.
\begin{theorem}\label{ft}
  Consider a Monge system on an 8-dimensional manifold $M$ with the main invariant $T=T(b_6)$ having a single root. Then such a system is given in terms of a coframe $(\omega^1,\omega^2,\dots,\omega^8)$ as in (\ref{sysc3u})-(\ref{sysc3u1}) and its all local differential invariants are given in terms of a curvature of \emph{principal} connection $\gamma$ with \emph{torsion} on an 11-dimensional principal fiber bundle $H\to{\mathcal G}^{11}\to M$.

  The bundle $H\to{\mathcal G}^{11}\to M$ is locally a subbundle of the Cartan bundle $P\to{\mathcal G}^{21}\to M$ with an $\spa(3,\bbR)$ Cartan connection
  $$\omega=\bma[c|c||c|c]
  \begin{matrix}\Om_1&\Om_2\\\Om_3&\Om_4\end{matrix}&\begin{matrix}\Om_6\\\Om_7\end{matrix}&\begin{matrix}\Om_{11}&\Om_{12}\\\Om_{12}&\Om_{13}\end{matrix}&\begin{matrix}\Om_9\\\Om_{10}\end{matrix}\\
  \cmidrule(lr){1-4}
  \begin{matrix}\theta^7&\theta^6\end{matrix}&\Om_5&\begin{matrix}\Om_9&\Om_{10}\end{matrix}&\Om_8\\
  \cmidrule(lr){1-4}\morecmidrules\cmidrule(lr){1-4}
  \begin{matrix}\theta^2&\theta^1\\\theta^1&\theta^3\end{matrix}&\begin{matrix}\theta^5\\\theta^4\end{matrix}&\begin{matrix}-\Om_1&-\Om_3\\-\Om_2&-\Om_4\end{matrix}&\begin{matrix}-\theta^7\\-\theta^6\end{matrix}\\
  \cmidrule(lr){1-4}
  \begin{matrix}\theta^5&\theta^4\end{matrix}&\theta^8&\begin{matrix}-\Om_6&-\Om_7\end{matrix}&-\Om_5
      \ema,
      $$
      in which the forms $(\theta^1,\theta^2,\dots,\theta^8)$ are
      \be\theta^i=A^i{}_j\omega^j,\label{fo0}\ee with $(A^i{}_j)\in G_0$ as in Proposition \ref{pr36}. The subset ${\mathcal G}^{11}\subset{\mathcal G}^{21}$ is defined by the following condition on the curvature $$\kappa=\der\omega+\omega\dz\omega$$ of the Cartan connection:
      \be\kappa=\bma[l|c||c|c]
  \begin{matrix}0&\quad\quad\,\,\, 0\\0&\quad\quad\,\,\, 0\end{matrix}&\begin{matrix}0\\0\end{matrix}&\begin{matrix}0&0\\0&0\end{matrix}&\begin{matrix}0\\0\end{matrix}\\
  \cmidrule(lr){1-4}
  \begin{matrix}0&\quad\quad\,\,\, 0\end{matrix}&0&\begin{matrix}0&0\end{matrix}&0\\
  \cmidrule(lr){1-4}\morecmidrules\cmidrule(lr){1-4}
  \begin{matrix}0&0\\0&\theta^2\dz(\theta^5+\theta^7)\end{matrix}&\begin{matrix}0\\0\end{matrix}&\begin{matrix}0&0\\0&0\end{matrix}&\begin{matrix}0\\0\end{matrix}\\
  \cmidrule(lr){1-4}
  \begin{matrix}0&\quad\quad\,\,\, 0\end{matrix}&0&\begin{matrix}0&0\end{matrix}&0
      \ema.\label{kap}\ee
      The subgroup $H\subset G_0$ consists of real $8\times 8$ matrices 
 \be({\mathcal A}^i{}_j)=\bma
  a^5b&a^5bc&0&0&0&0&0&0\\
  0&a^4b&0&0&0&0&0&0\\
  2a^6bc&a^6bc^2&a^6b&0&0&0&0&0\\
  0&0&0&a^3b&a^3bc&0&0&0\\
  0&0&0&0&a^2b&0&0&0\\
  0&0&0&a^3(1-b)&a^3(1-b)c&a^3&a^3c&0\\
  0&0&0&0&a^2(1-b)&0&a^2&0\\
   0&0&0&0&0&0&0&b
   \ema,\quad ab\neq 0.\label{mah}\ee
\end{theorem}
Before the proof, for the completness of this result, we also give a corollary, which will be proven, after the proof of the theorem.
\begin{corollary}\label{coco}
  The normalization condition (\ref{kap}) implies that the basis forms $(\theta^1,\theta^2,\dots,\theta^8,\Om_1,\Om_2,\dots,\Om_{13})$ of the Cartan connection $\omega$, when pullbacked from ${\mathcal G}^{21}$ to ${\mathcal G}^{11}$, satisfy
\be\begin{aligned}
&\Om_3=\Om_7=\Om_{10}=\Om_{12}=\Om_{13}=0\\
    &\Om_5=3\Om_1-2\Om_4-\theta^8,\quad\quad\quad\quad\Om_8=6\Om_1-4\Om_4-\theta^8\\
&\Om_6=-t^4_{~12}\theta^2-t^4_{~15}\theta^5-t^4_{~17}\theta^7\\
&\Om_9=-t^4_{~12}\theta^2+(2t^4_{~17}-3t^4_{~15})\theta^5+(t^4_{~17}-2t^4_{~15})\theta^7\\
&\Om_{11}=r_{212}\theta^2-\tfrac43t^4_{~12}\theta^5-\tfrac43t^4_{~12}\theta^7\\
\end{aligned}\label{coco1}
\ee
  on ${\mathcal G}^{11}$. In terms of the pullbacks the principal connection $\gamma$ on ${\mathcal G}^{11}$ is given by
{\tiny\be
(\gamma^i{}_j)=\bma
-\Om_1-\Om_4&-\Om_2&0&0&0&0&0&0\\0&-2\Om_1&0&0&0&0&0&0\\-2\Om_2&0&-2\Om_4&0&0&0&0&0\\0&0&0&-3\Om_1+\Om_4&-\Om_2&0&0&0\\0&0&0&0&-4\Om_1+2\Om_4&0&0&0\\0&0&0&6\Om_1-4\Om_4&0&3\Om_1-3\Om_4&-\Om_2&0\\0&0&0&0&6\Om_1-4\Om_4&0&2\Om_1-2\Om_4&0\\
  0&0&0&0&0&0&0&-6\Om_1+4\Om_4
  \ema.\label{coco2}\ee}
\end{corollary}
\begin{proof}(of the Theorem and the Corollary)
  We take the most general coframe $(\theta^i)=(A^i{}_j\omega^j)$, with $(A^i{}_j)\in G_0$, corresponding to the coframe $(\omega^j)$ satisfying (\ref{sysc3u})-(\ref{sysc3u1}), and try to solve the normalization conditions (\ref{kap}) for $A^i{}_j$ and $(\Omega_1,\Omega_2,\dots,\Om_{13})$. Explicitly, we are looking for  $A^i{}_j$ and $(\Omega_1,\Omega_2,\dots,\Om_{13})$ such that:
  {\tiny $$\begin{aligned}
       0=&\bma[l|c||c|c]
  \begin{matrix}0&\quad\quad\,\,\, 0\\0&\quad\quad\,\,\, 0\end{matrix}&\begin{matrix}0\\0\end{matrix}&\begin{matrix}0&0\\0&0\end{matrix}&\begin{matrix}0\\0\end{matrix}\\
  \cmidrule(lr){1-4}
  \begin{matrix}0&\quad\quad\,\,\, 0\end{matrix}&0&\begin{matrix}0&0\end{matrix}&0\\
  \cmidrule(lr){1-4}\morecmidrules\cmidrule(lr){1-4}
  \begin{matrix}{\color{blue}0}&0\\0&{\color{red} \theta^2\dz(\theta^5+\theta^7)}\end{matrix}&\begin{matrix}0\\0\end{matrix}&\begin{matrix}0&0\\0&0\end{matrix}&\begin{matrix}0\\0\end{matrix}\\
  \cmidrule(lr){1-4}
  \begin{matrix}0&\quad\quad\,\,\, 0\end{matrix}&0&\begin{matrix}0&0\end{matrix}&0
      \ema-\der\bma[c|c||c|c]
  \begin{matrix}\Om_1&\Om_2\\\Om_3&\Om_4\end{matrix}&\begin{matrix}\Om_6\\\Om_7\end{matrix}&\begin{matrix}\Om_{11}&\Om_{12}\\\Om_{12}&\Om_{13}\end{matrix}&\begin{matrix}\Om_9\\\Om_{10}\end{matrix}\\
  \cmidrule(lr){1-4}
  \begin{matrix}\theta^7&\theta^6\end{matrix}&\Om_5&\begin{matrix}\Om_9&\Om_{10}\end{matrix}&\Om_8\\
  \cmidrule(lr){1-4}\morecmidrules\cmidrule(lr){1-4}
  \begin{matrix}{\color{blue}\theta^2}&\theta^1\\\theta^1&{\color{red} \theta^3}\end{matrix}&\begin{matrix}\theta^5\\\theta^4\end{matrix}&\begin{matrix}-\Om_1&-\Om_3\\-\Om_2&-\Om_4\end{matrix}&\begin{matrix}-\theta^7\\-\theta^6\end{matrix}\\
  \cmidrule(lr){1-4}
  \begin{matrix}\theta^5&\theta^4\end{matrix}&\theta^8&\begin{matrix}-\Om_6&-\Om_7\end{matrix}&-\Om_5
      \ema-\\&\bma[c|c||c|c]
  \begin{matrix}\Om_1&\Om_2\\\Om_3&\Om_4\end{matrix}&\begin{matrix}\Om_6\\\Om_7\end{matrix}&\begin{matrix}\Om_{11}&\Om_{12}\\\Om_{12}&\Om_{13}\end{matrix}&\begin{matrix}\Om_9\\\Om_{10}\end{matrix}\\
  \cmidrule(lr){1-4}
  \begin{matrix}\theta^7&\theta^6\end{matrix}&\Om_5&\begin{matrix}\Om_9&\Om_{10}\end{matrix}&\Om_8\\
  \cmidrule(lr){1-4}\morecmidrules\cmidrule(lr){1-4}
  \begin{matrix}{\color{blue}\theta^2}&{\color{blue}\theta^1}\\{\color{red}\theta^1}&{\color{red}\theta^3}\end{matrix}&\begin{matrix}{\color{blue}\theta^5}\\{\color{red}\theta^4}\end{matrix}&\begin{matrix}{\color{blue}-\Om_1}&{\color{blue}-\Om_3}\\{\color{red}-\Om_2}&{\color{red}-\Om_4}\end{matrix}&\begin{matrix}{\color{blue}-\theta^7}\\{\color{red}-\theta^6}\end{matrix}\\
  \cmidrule(lr){1-4}
  \begin{matrix}\theta^5&\theta^4\end{matrix}&\theta^8&\begin{matrix}-\Om_6&-\Om_7\end{matrix}&-\Om_5
      \ema\dz\bma[c|c||c|c]
  \begin{matrix}{\color{blue}\Om_1}&{\color{red}\Om_2}\\{\color{blue}\Om_3}&{\color{red}\Om_4}\end{matrix}&\begin{matrix}\Om_6\\\Om_7\end{matrix}&\begin{matrix}\Om_{11}&\Om_{12}\\\Om_{12}&\Om_{13}\end{matrix}&\begin{matrix}\Om_9\\\Om_{10}\end{matrix}\\
  \cmidrule(lr){1-4}
  \begin{matrix}{\color{blue}\theta^7}&{\color{red}\theta^6}\end{matrix}&\Om_5&\begin{matrix}\Om_9&\Om_{10}\end{matrix}&\Om_8\\
  \cmidrule(lr){1-4}\morecmidrules\cmidrule(lr){1-4}
  \begin{matrix}{\color{blue}\theta^2}&{\color{red}\theta^1}\\{\color{blue}\theta^1}&{\color{red}\theta^3}\end{matrix}&\begin{matrix}\theta^5\\\theta^4\end{matrix}&\begin{matrix}-\Om_1&-\Om_3\\-\Om_2&-\Om_4\end{matrix}&\begin{matrix}-\theta^7\\-\theta^6\end{matrix}\\
  \cmidrule(lr){1-4}
  \begin{matrix}{\color{blue}\theta^5}&{\color{red}\theta^4}\end{matrix}&\theta^8&\begin{matrix}-\Om_6&-\Om_7\end{matrix}&-\Om_5
      \ema
     \end{aligned}$$}
for $A^i{}_j$ and $(\Omega_1,\Omega_2,\dots,\Om_{13})$. 
The first thing we note is that these equations differ from the flat Maurer-Cartan equations only in the equation for $\der\theta^3$ (distinguished in red color above). And the difference is very tiny: even this equation differes from the equation for $\der\theta^3$ in the flat case by one term only, the term equal to ${\color{red}\theta^2\dz(\theta^5+\theta^7)}$. Writing this (red colored) equation explicitly we have ${\color{red}E_3=0}$ with  
  \be{\color{red}E_3=\der\theta^3-\Big(2\Om_2\dz\theta^1+2\Om_4\dz\theta^3-2\theta^4\dz\theta^6+\theta^2\dz(\theta^5+\theta^7)\Big)}\label{e3}\ee
as one of the equations to solve. On the other hand, from Proposition \ref{proe} we know that in the case of \emph{any} Monge system (\ref{sysc3}) the equation for $\der\theta^2$, which now reads ${\color{blue}E_2=0}$ with
\be {\color{blue}E_2= \der\theta^2-\Big(2\Om_1\dz\theta^2+2\Om_3\dz\theta^1-2\theta^5\dz\theta^7\Big)},\label{e2}\ee
satisfies
$$\Big({\color{blue}\der\theta^2}-\frac{a_{38}^2T}{a_{29}^3b_7(b_5b_6-1)^3}\theta^3\dz\theta^6\Big)\dz\theta^1\dz\theta^2\dz\theta^4\dz\theta^5\dz\theta^7\dz\theta^8=0.$$
This last equation is compatible with the equation (\ref{e2}) which we want to solve, if and only if
$$\frac{a_{38}^2T}{a_{29}^3b_7(b_5b_6-1)^3}=0.$$
In the flat case, when $T= 0$, this is satisfied automatically, but in our case when $T=ub_6^5$ we are forced to have
$$\frac{a_{38}^2ub_6^5}{a_{29}^3b_7(b_5b_6-1)^3}=0.$$
Since $u\neq 0$ this requires a \emph{normalization of one of the parameters} of the group $G_0$. Moreover, since the determinant of the most general element $(A^i{}_j)$ of the group $G_0$ is
\be\det(A^i{}_j)=-a_{29}^5a_{38}^5a_{49}^6b_7^{10}(1-b_5b_6)^^5\neq 0,\label{det}\ee
we can \emph{not} choose $a_{38}=0$. Thus, to satisfy equation (\ref{e2}), we must restrict the set of group parameters to a subset in which
\be
b_6=0.\label{no1}\ee
Note that such a normalization would not be justified in the flat case!

Now we accept the normalization, $b_6=0$, assuming it from now on,
and try to solve the rest of the equations (\ref{kap}). In particular we look at the equation $${\color{red}E_3}\dz\theta^1\dz\theta^3\dz\theta^4{\color{red}=0}.$$
A short calculation shows that this is equivalent to:
$$\begin{aligned}&{\color{red} 0=}-a_{38}^3a_{49}b_7^5{\color{red}E_3}\dz\theta^1\dz\theta^3\dz\theta^4=\\
&\Big((ua_{29}^2+a_{38}^3a_{49}b_7^2)\theta^5-2a_{38}(a_{17}-a_{16}b_5+a_{18}b_5^2)\theta^6-a_{49}b_7(ua_{29}^2-a_{38}^3b_7)\theta^7\Big)\dz\theta^{1234},\end{aligned}$$
where we have introduced $\theta^{1234}:=\theta^1\dz\theta^2\dz\theta^3\dz\theta^4$. Since the indices in the basis forms $\theta^i$ run from 1 to 8, i.e. they are one-digit, we will use similar abreviations in the following: for example $\theta^{ij}$ and $\theta^{ijk}$ will respectively mean: $\theta^{ij}=\theta^i\dz\theta^j$ and $\theta^{ijk}=\theta^i\dz\theta^j\dz\theta^k$. 

Because of $a_{38}a_{49}b_7\neq 0$, due to the determinant condition (\ref{det}), and the linear independence of $\theta^5$, $\theta^6$ and $\theta^7$, we have to make another restriction on the group $G_0$ parameters to satisfy this equation. Solving (\ref{no2}) with respect to $a_{49}$, $b_7$ and $a_{17}$ we get the following further restrictions:
\be
a_{49}=-\frac{a_{38}^3}{ua_{29}^2},\quad b_7=\frac{ua_{29}^2}{a_{38}^3},\quad a_{17}=b_5(a_{16}-a_{18}b_5).\label{no2}\ee
Having imposed restrictions (\ref{no1}), (\ref{no2}) we now look at
$${\color{red} 0=E_3}\dz\theta^1\dz\theta^3,\quad \mathrm{and}\quad {\color{blue} 0=E_2}\dz\theta^1\dz\theta^2.$$
These, respectively, are:
$$\begin{aligned}
  {\color{red} 0=u a_{29}^2E_3}\dz\theta^1\dz\theta^3=-2a_{38}(a_{27}-a_{26}b_5+a_{28}b_5^2)\theta^{1234}\\
{\color{blue} 0=ua_{29}^4E_2}\dz\theta^1\dz\theta^2=-2a_{38}^3(a_{36}\theta^5-a_{23}\theta^7)\dz\theta^{123}.\end{aligned}$$ 
This brings next \emph{three} new normalizations:
\be a_{23}=a_{36}=0,\quad a_{27}=b_5(a_{26}-a_{28}b_5).\label{no3}\ee
After imposing these normalizations we now look at ${\color{blue}0=E_2}\dz\theta^2$. This reads:
$${\color{blue}0=E_2}\dz\theta^2=-2(\Om_3+\frac{a_{38}^2a_{34}}{ua_{29}^3}\theta^5-\frac{a_{21}a_{38}^2}{ua_{29}^3}\theta^7)\dz\theta^{12}.$$
  This, in particular shows that the 1-form $\Om_3$ must be a `lift' of a 1-form from $M$, as to satisfy this equation we have to have
  \be\Om_3=-\frac{a_{38}^2a_{34}}{ua_{29}^3}\theta^5+\frac{a_{38}^2a_{21}}{ua_{29}^3}\theta^7+u_1 \theta^1+u_2\theta^2,\label{fo1}\ee
  with some \emph{new} parameters $$u_1\quad\&\quad u_2.$$
  It is now time to analyze consequences of equation $E_1$ for the form $\der\theta^1$. This reads:
  $$E_1=\der\theta^1-(\Om_1\dz\theta^1+\Om_2\dz\theta^2+\Om_3\dz\theta^3+\Om_4\dz\theta^1-\theta^4\dz\theta^7-\theta^5\dz\theta^6).$$
  Using the normalizations (\ref{no1}), (\ref{no2}), (\ref{no3}), and the definition (\ref{fo1}) we easilly find that:
  $$0=u a_{29}^4E_1\dz\theta^1\dz\theta^2=a_{38}^2\Big((a_{29}a_{34}-a_{28}a_{38})\theta^5+(a_{18}a_{38}-a_{21}a_{29})\theta^7\Big)\dz\theta^{123}.$$
  This gives next normalizations, which require tu put:
  \be a_{18}=\lambda_2 a_{29},\quad a_{21}=\lambda_2a_{38},\quad a_{28}=\lambda_1 a_{29},\quad a_{34}=\lambda_1 a_{38},\label{no4}\ee
  with new unknowns
  $$\lambda_1\quad\&\quad \lambda_2.$$
  Now the linear algebraic equations $E_1={\color{blue} E_2}={\color{red} E_3}=0$ for the forms $\Om_1$, $\Om_2$ $\Om_3$ can be totally solved. The result is:
  \be\begin{aligned}
  \Om_1=&\frac{\der a_{29}}{a_{29}} - \frac{\der a_{38}}{2 a_{38}} +\\& \frac{2 u^3 u_2 a_{29}^6 + 2 \lambda_1 u a_{16} a_{38}^5 - 2 \lambda_2 u a_{26} a_{38}^5 + \lambda_1 v_2 a_{29} a_{38}^5 + \lambda_2 v_2 a_{29} a_{38}^5}{2 u^3 a_{29}^6} \theta^1+ u_3 \theta^2 - \\&\frac{a_{38}^2 (2 u a_{26} - v_2 a_{29} - 4 \lambda_1 u a_{29} b_{5})}{2 u^2 a_{29}^3} \theta^5 + \frac{a_{38}^2 (2 u a_{16} + v_2 a_{29} - 4 \lambda_2 u a_{29} b_{5})}{2 u^2 a_{29}^3} \theta^7\\
  \Om_2=&\frac{a_{29} \der b_{5}}{a_{38}} +\\& \frac{2 u^3 u_3 a_{29}^5 + u v_1 a_{29}^3 a_{38} - v_2 a_{16} a_{38}^4 - v_2 a_{26} a_{38}^4 + 2 \lambda_1 v_2 a_{29} a_{38}^4 b_{5} + 2 \lambda_2 v_2 a_{29} a_{38}^4 b_{5}}{2 u^3 a_{29}^5} \theta^1 +\\& \frac{(\lambda_1 + \lambda_2) a_{38}^3}{2 u a_{29}^3} \theta^2 + \frac{u^2 u_2 a_{29}^6 + \lambda_1 a_{16} a_{38}^5 - \lambda_2 a_{26} a_{38}^5}{u^2 a_{29}^6} \theta^3 - \frac{a_{38}^2 (a_{26} - 2 \lambda_1 a_{29} b_{5})}{u a_{29}^3}\theta^4 + \\&\frac{a_{38}^2 (a_{16} - 2 \lambda_2 a_{29} b_{5})}{u a_{29}^3} \theta^6\\
  \Om_4=&\frac{2 \der a_{29}}{a_{29}} - \frac{3 \der a_{38}}{2 a_{38}} + \frac{2 u^3 u_2 a_{29}^5 + \lambda_1 v_2 a_{38}^5 + \lambda_2 v_2 a_{38}^5}{2 u^3 a_{29}^5} \theta^1 - \\&\frac{a_{38} (u v_1 a_{29}^3 - v_2 a_{16} a_{38}^3 - v_2 a_{26} a_{38}^3 + 
    2 \lambda_1 v_2 a_{29} a_{38}^3 b_{5} + 2 \lambda_2 v_2 a_{29} a_{38}^3 b_{5})}{2 u^3 a_{29}^5} \theta^2 + u_1 \theta^3 -\\& \frac{\lambda_1 a_{38}^3}{u a_{29}^3} \theta^4 + \frac{v_2 a_{38}^2}{2 u^2 a_{29}^2} \theta^5 + \frac{\lambda_2 a_{38}^3}{u a_{29}^3} \theta^6 + \frac{v_2 a_{38}^2}{2 u^2 a_{29}^2} \theta^7,
  \end{aligned}\label{fo2}\ee
  with the following new normalizations:
  \be a_{22}=\frac{a_{38}(a_{16}-\lambda_2a_{29}b_5)}{a_{29}},\quad a_{35}=\frac{a_{38}(a_{26}-\lambda_1a_{29}b_5)}{a_{29}}.\label{no5}\ee
  Note that in the above formulas for $\Omega_1,\Omega_2,\Omega_4$ a new unknown $$u_3$$
 was introduced. Note also the appearence of the derivatives $v_1$ and $v_2$ of the function $u$ in these expressions.

 Let us summarize, this what we have achieved so far:

 We solved the equations $E_1=E_2=E_3=0$, which forced us to normalize \emph{thirteen} group $G_0$ parameters, namely $b_6, b_7, a_{17},a_{18}, a_{21}, a_{22}, a_{23}, a_{27},  a_{28}, a_{34}, a_{35}, a_{36}, a_{49}$. We also determined \emph{four} 1-forms $\Om_1,\Om_2,\Om_3,\Om_4$, which were part of our unknowns. The price paid for this was the introduction of \emph{five new} unknowns, namely the variables $u_1, u_2,u_3,\lambda_1,\lambda_2$. Recalling that $G_0$ has 23 arameters, we are now rested with $23-13+5=15$ unknown scalar variables, and still undetermined 1-forms $\Om_5,\Om_6,\dots, \Om_{13}$.

 Now in the same way we solve the normalization equations (\ref{kap}) with $\der\theta^4$, $\der\theta^5$, $\der\theta^6$, $\der\theta^7$, $\der\theta^8$. These are:
 $$\begin{aligned}
 E_4=&\der\theta^4-(\Om_2\dz\theta^5+\Om_4\dz\theta^4+\Om_5\dz\theta^4+\Om_6\dz\theta^1+\Om_7\dz\theta^3+\theta^6\dz\theta^8)=0\\
E_5=&\der\theta^5-(\Om_1\dz\theta^5+\Om_3\dz\theta^4+\Om_5\dz\theta^5+\Om_6\dz\theta^2+\Om_7\dz\theta^1+\theta^7\dz\theta^8)=0\\
E_6=&\der\theta^6-(\Om_2\dz\theta^7+\Om_4\dz\theta^6-\Om_5\dz\theta^6-\Om_8\dz\theta^4-\Om_9\dz\theta^1-\Om_{10}\dz\theta^3)=0\\
E_7=&\der\theta^7-(\Om_1\dz\theta^7+\Om_3\dz\theta^6-\Om_5\dz\theta^7-\Om_8\dz\theta^5-\Om_9\dz\theta^2-\Om_{10}\dz\theta^1)=0\\
E_8=&\der\theta^8-(2\Om_5\dz\theta^8+2\Om_6\dz\theta^5+2\Om_7\dz\theta^4)=0.
 \end{aligned}$$ 
 Here we only indicate in which order one should do it, and what is the result at each step:
 \begin{itemize}
 \item Equations $E_4\dz\theta^{134}=0$ and $E_5\dz\theta^{125}=0$ are equivalent to the following normalizations
   \be\lambda_1=\lambda_2=a_{44}=a_{45}=0, \quad a_{16}=\frac{ua_{29}^3a_{46}}{2a_{38}^3},\quad a_{43}=\frac{4a_{38}^3a_{42}b_5-ua_{29}^2a_{46}^2}{4a_{38}^3}\label{no6}\ee
 \item Equations $E_4\dz\theta^{14}=0$, $E_5\dz\theta^{25}=0$ and $E_8\dz\theta^{58}=0$ are equivalent to
   \be \Om_7=u_1\theta^4+u_2\theta^5,\label{fo3}\ee
   and a normalization
   \be a_{42}=0.\label{no7}\ee
 \item Equations $E_6\dz\theta^{146}=0$ and $E_7\dz\theta^{257}=0$ are equivalent to
   \be\Om_{10}=u_4\theta^1-u_1\theta^6-u_2\theta^7.\label{fo4}\ee
   Here we had to introduce a \emph{new unknown} $u_4$.
   \item Now solving $E_4\dz\theta^4=0$, $E_4=0$, $E_5=0$, $E_7\dz\theta^5=0$, $E_7=0$, $E_6=0$ for $\Om_6$, $\Om_5$, $\Om_9$ $\Om_8$, respectively, and using all the normalizations made so far, we get the ten 1-forms $\Om_1,\Om_2,\dots,\Om_{10}$ as below:
 \end{itemize}
 \be\begin{aligned}
 \Om_1=&\frac{\der a_{29}}{a_{29}} - \frac{\der a_{38}}{2 a_{38}} +u_2 \theta^1+ u_3 \theta^2 -\frac{a_{38}^2 (2 u a_{26} - v_2 a_{29})}{2 u^2 a_{29}^3} \theta^5 + \frac{u^2a_{29}^2a_{46} + v_2 a_{38}^3}{2 u^2 a_{29}^2a_{38}} \theta^7\\
 \Om_2=&\frac{a_{29} \der b_{5}}{a_{38}} +\frac{4 u^3 u_3 a_{29}^5 +2 u v_1 a_{29}^3 a_{38} - 2v_2 a_{26} a_{38}^4 - u v_2 a_{29}^3 a_{38}a_{46}}{4 u^3 a_{29}^5} \theta^1 +\\& + u_2\theta^3 - \frac{a_{38}^2a_{26}}{u a_{29}^3}\theta^4 + \frac{a_{46}}{2 a_{38}} \theta^6\\
 \Om_3=&u_1\theta^1+u_2\theta^2\\
 \Om_4=&\frac{2 \der a_{29}}{a_{29}} - \frac{3 \der a_{38}}{2 a_{38}} + u_2\theta^1 -\frac{a_{38} (2u v_1 a_{29}^3 - 2v_2 a_{26} a_{38}^3 - uv_2 a_{29}^3 a_{46})}{4 u^3 a_{29}^5} \theta^2 + u_1 \theta^3 +\\ &\frac{v_2 a_{38}^2}{2 u^2 a_{29}^2} \theta^5  + \frac{v_2 a_{38}^2}{2 u^2 a_{29}^2} \theta^7\\
 \Om_5=&-\frac{\der a_{29}}{a_{29}} + \frac{3\der a_{38}}{2 a_{38}} -\\& \frac{-2 u v_1 a_{29}^3 a_{38}^3 + 2 v_2 a_{26} a_{38}^6 + 2 u^2 a_{26} a_{29}^2 a_{38}^3 a_{46} + u v_2 a_{29}^3 a_{38}^3 a_{46} + u^3 a_{29}^5 a_{46}^2}{4 u^3 a_{29}^5 a_{38}^2} \theta^2 -\\& \frac{2 u a_{26} a_{38}^3 + v_2 a_{29} a_{38}^3 + 2 u^2 a_{29}^3 a_{46}}{2 u^2 a_{29}^3 a_{38}} \theta^5 - \frac{v_2 a_{38}^3 + u^2 a_{29}^2 a_{46}}{2 u^2 a_{29}^2 a_{38}} \theta^7 - \theta^8\\
 \Om_6=&-\frac{a_{46} \der a_{29}}{a_{29} a_{38}} + \frac{3 a_{46} \der a_{38}}{2 a_{38}^2} - \frac{\der a_{46}}{2 a_{38}} - \\&\frac{a_{46} (-4 u v_1 a_{29}^3 a_{38}^3 + 4 v_2 a_{26} a_{38}^6 + 2 u^2 a_{26} a_{29}^2 a_{38}^3 a_{46} + 2 u v_2 a_{29}^3 a_{38}^3 a_{46} +u^3 a_{29}^5 a_{46}^2}{8 u^3 a_{29}^5 a_{38}^3} \theta^2+ u_2 \theta^4 + \\&\frac{4 u^3 u_3 a_{29}^5 a_{38}^2 + 2 u v_1 a_{29}^3 a_{38}^3 - 2 v_2 a_{26} a_{38}^6 - 2 u^2 a_{26} a_{29}^2 a_{38}^3 a_{46} - 3 u v_2 a_{29}^3 a_{38}^3 a_{46} - 2 u^3 a_{29}^5 a_{46}^2}{4 u^3 a_{29}^5 a_{38}^2} \theta^5 -\\& \frac{a_{46} (2 v_2 a_{38}^3 + u^2 a_{29}^2 a_{46})}{4 u^2 a_{29}^2 a_{38}^2} \theta^7 - \frac{2 a_{26} a_{38}^3 + u a_{29}^3 a_{46}}{2 u a_{29}^3 a_{38}} \theta^8\\
 \Om_7=&u_1\theta^4+u_2\theta^5\\
 \Om_8=&-\frac{2 \der a_{29}}{a_{29}} + \frac{3 \der a_{38}}{a_{38}} - \\&\frac{-4 u v_1 a_{29}^4 a_{38}^3 - 4 u a_{26}^2 a_{38}^6 + 4 v_2 a_{26} a_{29} a_{38}^6 + 2 u v_2 a_{29}^4 a_{38}^3 a_{46} + u^3 a_{29}^6 a_{46}^2}{4 u^3 a_{29}^6 a_{38}^2} \theta^2 -\\&
 \frac{v_2 a_{38}^3 + u^2 a_{29}^2 a_{46}}{u^2 a_{29}^2 a_{38}} \theta^5 + \frac{(2 u a_{26} - v_2 a_{29}) a_{38}^2}{u^2 a_{29}^3} \theta^7 - \theta^8\\
 \Om_9=&\frac{a_{38}^2 \der a_{26}}{u a_{29}^3} - \frac{(3 a_{26} a_{38}^3 + u a_{29}^3 a_{46}) \der a_{29}}{u a_{29}^4 a_{38}} + \frac{3 (2 a_{26} a_{38}^3 + u a_{29}^3 a_{46}) \der a_{38}}{2 u a_{29}^3 a_{38}^2} - \\&\frac{(2 a_{26} a_{38}^3 + u a_{29}^3 a_{46})(-4 u v_1 a_{29}^3 a_{38}^3 + 4 v_2 a_{26} a_{38}^6 + 2 u v_2 a_{29}^3 a_{38}^3 a_{46} + u^3 a_{29}^5 a_{46}^2)}{8 u^4 a_{29}^8 a_{38}^3} \theta^2- \\&\frac{2 u a_{26}^2 a_{38}^6 + 2 v_2 a_{26} a_{29} a_{38}^6 + 2 u^2 a_{26} a_{29}^3 a_{38}^3 a_{46} + u v_2 a_{29}^4 a_{38}^3 a_{46} + u^3 a_{29}^6 a_{46}^2}{2 u^3 a_{29}^6 a_{38}^2} \theta^5 - u_2 \theta^6 - \\&\frac{4 u^3 u_3 a_{29}^5 + 2 u v_1 a_{29}^3 a_{38} + 2 v_2 a_{26} a_{38}^4 - 2 u^2 a_{26} a_{29}^2 a_{38} a_{46} + u v_2 a_{29}^3 a_{38} a_{46}}{ 4 u^3 a_{29}^5} \theta^7 - \\&\frac{2 a_{26} a_{38}^3 + u a_{29}^3 a_{46}}{2 u a_{29}^3 a_{38}} \theta^8\\
 \Om_{10}=&-u_1\theta^6-u_2\theta^7.
 \end{aligned}\label{fo5}\ee
 We have to mention that to satisfy equations $E_4\dz\theta^4=0$, $E_4=0$, $E_5=0$, $E_7\dz\theta^5=0$, $E_7=0$, $E_6=0$ we had to make the normalization
 \be u_4=0,\label{no8}\ee
 as it is vissible in the expression for the form $\Om_{10}$ above. Also, we note that solving the equations $E_4\dz\theta^4=0$, $E_4=0$, $E_5=0$, $E_7\dz\theta^5=0$, $E_7=0$, $E_6=0$ did \emph{not} bring any new additional variables. So, at this stage we normalized \emph{eighteen} of the group parameters: $b_6, b_7, a_{16},a_{17},a_{18}, a_{21}, a_{22}, a_{23}, a_{27},a_{28}, a_{34}, a_{35}, a_{36},a_{42}, a_{43},$
 $a_{44},$ $a_{45},a_{49}$, and were forced to kill three from the additional unknowns $\lambda_1,\lambda_2,u_1,$ $u_2,u_3, u_4$, namely $\lambda_1,\lambda_2, u_4$. So now, the number of scalar unknowns reduced to $23+4+2-18-1-2=8$. These are $b_5, a_{26}, a_{29},a_{38}, a_{46},u_1,u_2,u_3$.  Moreover it follows that not only we now solved the equations $E_1=E_2=E_3=E_4=E_5=E_6=E_7=0$, but we also have $E_8=0$, which is a consequence of our normalizations and the other 7 equations $E_i=0$, $i=1,2,\dots,7$. So to fully solve our normalization conditions (\ref{kap}) we are now left with equations involving $\der\Omega_a$, $a=1,2,\dots,13$. The equations with $\der\theta^i$, $i=1,2\dots,8$ were solved!

We now pass to solve the equations involving $\der\Omega_a$, $a=1,2,\dots,13$. Explicitly, these are: 
$$
\begin{aligned}
E_9=&\der\Om_1-(\theta^1\dz\Om_{12}+\theta^2\dz\Om_{11}+\theta^5\dz\Om_9+\theta^7\dz\Om_6-\Om_2\dz\Om_3)=0\\
E_{10}=&\der\Om_2-(\theta^1\dz\Om_{11}+\theta^3\dz\Om_{12}+\theta^4\dz\Om_9+\theta^6\dz\Om_6-\Om_1\dz\Om_2-\Om_2\dz\Om_4)=0\\
E_{11}=&\der\Om_3-(\theta^1\dz\Om_{13}+\theta^2\dz\Om_{12}+\theta^5\dz\Om_{10}+\theta^7\dz\Om_7+\Om_1\dz\Om_3+\Om_3\dz\Om_4)=0\\
E_{12}=&\der\Om_4-(\theta^1\dz\Om_{12}+\theta^3\dz\Om_{13}+\theta^4\dz\Om_{10}+\theta^6\dz\Om_7+\Om_2\dz\Om_3)=0\\
E_{13}=&\der\Om_5-(\theta^4\dz\Om_{10}+\theta^5\dz\Om_{9}-\theta^6\dz\Om_{7}-\theta^7\dz\Om_6+\theta^8\dz\Om_8)=0\\
E_{14}=&\der\Om_6-(\theta^4\dz\Om_{12}+\theta^5\dz\Om_{11}+\theta^8\dz\Om_{9}-\Om_1\dz\Om_6-\Om_2\dz\Om_7+\Om_5\dz\Om_6)=0\\
E_{15}=&\der\Om_7-(\theta^4\dz\Om_{13}+\theta^5\dz\Om_{12}+\theta^8\dz\Om_{10}-\Om_3\dz\Om_6-\Om_4\dz\Om_7+\Om_5\dz\Om_7)=0\\
E_{16}=&\der\Om_8-(-2\theta^6\dz\Om_{10}-2\theta^7\dz\Om_{9}-2\Om_5\dz\Om_8)=0\\
E_{17}=&\der\Om_9-(-\theta^6\dz\Om_{12}-\theta^7\dz\Om_{11}-\Om_1\dz\Om_{9}-\Om_2\dz\Om_{10}-\Om_5\dz\Om_9-\Om_6\dz\Om_8)=0\\
E_{18}=&\der\Om_{10}-(-\theta^6\dz\Om_{13}-\theta^7\dz\Om_{12}-\Om_3\dz\Om_{9}-\Om_4\dz\Om_{10}-\Om_5\dz\Om_{10}-\Om_7\dz\Om_8)=0\\
E_{19}=&\der\Om_{11}-(-2\Om_1\dz\Om_{11}-2\Om_2\dz\Om_{12}-2\Om_6\dz\Om_9)=0\\
E_{20}=&\der\Om_{12}-(-\Om_1\dz\Om_{12}-\Om_2\dz\Om_{13}-\Om_3\dz\Om_{11}-\Om_4\dz\Om_{12}-\Om_6\dz\Om_{10}-\Om_7\dz\Om_9)=0\\
E_{21}=&\der\Om_{13}-(-2\Om_3\dz\Om_{12}-2\Om_4\dz\Om_{13}-2\Om_7\dz\Om_{10})=0.
\end{aligned}
$$
As before we only indicate the order of solving the equations and the results:
\begin{itemize}
\item It follows that with all our normalizations made so far $E_{13}=0$ is automatically satisfied
\item Equation $E_{10}\dz\theta^{13}=0$ gives a normalization
  \be u_2=0\label{no9}\ee
\item Equation $E_9\dz\theta^2=0$ gives
  \be \Om_{12}=-\frac{u_1a_{29}}{a_{38}}\der b_5+\frac{u_1 a_{26}a_{38}^2}{ua_{29}^3}\theta^4-\frac{u_1a_{46}}{2a_{38}}\theta^6+u_5\theta^1+u_6\theta^2,\label{fo6}\ee
\item Now the $E_{12}\dz\theta^3=0$ gives new normalizations
  \be u_1=u_6=0,\label{no10},\ee
  whereas the equations $E_{11}=E_{12}=0$ tell that
  \be \Om_{13}=0,\label{fo7}\ee
  and that we have to normalize 
  \be u_5=0.\label{no11}\ee
\item Finally the equations $E_9=E_{10}=0$ solve for $\Om_{11}$, resulting in:
  \be\begin{aligned}
 \Om_{11}=& -\der u_3 + \frac{a_{38} (v_2 a_{38}^3 + u^2 a_{29}^2 a_{46})}{2 u^3 a_{29}^5} \der a_{26} -\\& \frac{4 u^3 u_3 a_{29}^5 a_{38}^2 + 3 v_2 a_{26} a_{38}^6 + 5 u^2 a_{26} a_{29}^2 a_{38}^3 a_{46} + u^3 a_{29}^5 a_{46}^2}{2 u^3 a_{29}^6 a_{38}^2} \der a_{29} + \\&\frac{4 u^3 u_3 a_{29}^5 a_{38}^2 + 6 v_2 a_{26} a_{38}^6 + 12 u^2 a_{26} a_{29}^2 a_{38}^3 a_{46} + 3 u^3 a_{29}^5 a_{46}^2}{4 u^3 a_{29}^5 a_{38}^3} \der a_{38}
 - \\&\frac{(2 u a_{26} - v_2 a_{29}) a_{38}}{4 u^2 a_{29}^3} \der a_{46}
 -u_{32}\theta^2-u_{35}\theta^5-u_{37}\theta^7-\frac{(2 a_{26} a_{38}^3 + u a_{29}^3 a_{46})^2}{4 u^2 a_{29}^6 a_{38}^2} \theta^8,
  \end{aligned}\label{fo8}\ee
  where
  $$\begin{aligned}
    u_{32}=&\Big(16 u^6 u_3^2 a_{29}^{10} a_{38}^4 + 12 u^2 v_1^2 a_{29}^6 a_{38}^6 - 
  8 u^3 w_1 a_{29}^6 a_{38}^6 - 32 u v_1 v_2 a_{26} a_{29}^3 a_{38}^9 +\\& 
  16 u^2 w_2 a_{26} a_{29}^3 a_{38}^9 + 20 v_2^2 a_{26}^2 a_{38}^{12} - 
  8 u w_3 a_{26}^2 a_{38}^{12} - 8 u^3 v_1 a_{26} a_{29}^5 a_{38}^6 a_{46} - \\&
  12 u^2 v_1 v_2 a_{29}^6 a_{38}^6 a_{46} + 
  8 u^3 w_2 a_{29}^6 a_{38}^6 a_{46} + 
  16 u^2 v_2 a_{26}^2 a_{29}^2 a_{38}^9 a_{46} + 
  16 u v_2^2 a_{26} a_{29}^3 a_{38}^9 a_{46} - \\&
  8 u^2 w_3 a_{26} a_{29}^3 a_{38}^9 a_{46} + 
  4 u^4 a_{26}^2 a_{29}^4 a_{38}^6 a_{46}^2 + 
  12 u^3 v_2 a_{26} a_{29}^5 a_{38}^6 a_{46}^2 + 
  3 u^2 v_2^2 a_{29}^6 a_{38}^6 a_{46}^2 - \\&
  2 u^3 w_3 a_{29}^6 a_{38}^6 a_{46}^2 + 
  4 u^5 a_{26} a_{29}^7 a_{38}^3 a_{46}^3 + 
  2 u^4 v_2 a_{29}^8 a_{38}^3 a_{46}^3 + u^6 a_{29}^{10} a_{46}^4\Big)\times \Big(16 u^6 a_{29}^{10} a_{38}^4\Big)^{-1}\\
  u_{35}=&\Big(-8 u^4 u_3 a_{26} a_{29}^5 a_{38}^5 + 4 u^3 u_3 v_2 a_{29}^6 a_{38}^5 - 
  2 u v_1 v_2 a_{29}^4 a_{38}^6 + 2 u^2 w_2 a_{29}^4 a_{38}^6 + 
  4 u v_2 a_{26}^2 a_{38}^9 +\\& 4 v_2^2 a_{26} a_{29} a_{38}^9 - 
  2 u w_3 a_{26} a_{29} a_{38}^9 + 2 u^3 v_1 a_{29}^6 a_{38}^3 a_{46} + 
  4 u^3 a_{26}^2 a_{29}^2 a_{38}^6 a_{46} +\\& 
  6 u^2 v_2 a_{26} a_{29}^3 a_{38}^6 a_{46} + 
  u v_2^2 a_{29}^4 a_{38}^6 a_{46} - u^2 w_3 a_{29}^4 a_{38}^6 a_{46} + 
  4 u^4 a_{26} a_{29}^5 a_{38}^3 a_{46}^2 +\\& 
  u^3 v_2 a_{29}^6 a_{38}^3 a_{46}^2 + u^5 a_{29}^8 a_{46}^3\Big)\times \Big(4 u^5 a_{29}^8 a_{38}^3 \Big)^{-1}\\
u_{37}=&\Big(4 u^3 u_3 v_2 a_{29}^5 a_{38}^3 - 2 u v_1 v_2 a_{29}^3 a_{38}^4 + 
  2 u^2 w_2 a_{29}^3 a_{38}^4 + 4 v_2^2 a_{26} a_{38}^7 - 
  2 u w_3 a_{26} a_{38}^7 + \\&4 u^5 u_3 a_{29}^7 a_{46} + 
  2 u^3 v_1 a_{29}^5 a_{38} a_{46} + 
  2 u^2 v_2 a_{26} a_{29}^2 a_{38}^4 a_{46} + 
  u v_2^2 a_{29}^3 a_{38}^4 a_{46} -\\& u^2 w_3 a_{29}^3 a_{38}^4 a_{46}\Big)\times \Big(4 u^5 a_{29}^7 a_{38} \Big)^{-1}
  \end{aligned}$$
  \item A remarkable fact now is that, it follows that the set of forms $(\Om_1,\Om_2,\dots,$ $\Om_{12},\Om_{13})$ given by (\ref{fo5}), (\ref{fo6}), (\ref{fo7}), (\ref{fo8}), together with normalizations (\ref{no1}), (\ref{no2}), (\ref{no3}), (\ref{no4}), (\ref{no5}), (\ref{no6}), (\ref{no7}), (\ref{no8}), (\ref{no9}), (\ref{no10}), (\ref{no11}) solves \emph{all} the equations $E_1=E_2=\dots=E_{21}=0$!
  \end{itemize}
Thus, taking into account all our normalizations (\ref{no1}), (\ref{no2}), (\ref{no3}), (\ref{no4}), (\ref{no5}), (\ref{no6}), (\ref{no7}), (\ref{no8}), (\ref{no9}), (\ref{no10}), (\ref{no11}), we are now left with only \emph{six} parameters $b_5, a_{26}, a_{29}, a_{38}, a_{46}, u_3$. Using them we found explicit expressions (\ref{fo0}), (\ref{fo5}), (\ref{fo6}), (\ref{fo7}), (\ref{fo8}) for all 21
forms $(\theta^1,\theta^2,\dots,\theta^8,\Om_1,\Om_2,$ $\dots,\Om_{13})$ which satisfy the normalization condition (\ref{kap}). These forms live on a $8+6=14$-dimensional manifold $M\times G_1$, which is parameterized by the points of $M$ and \emph{six} parameters $b_5, a_{26}, a_{29}, a_{38}, a_{46}, u_3$ localy describing $G_1$.

Note that with our current normalizations we have in particular:
$$\Om_3=\Om_7=\Om_{10}=\Om_{12}=\Om_{13}=0.$$
Since $21-5=16$ and the dimension of $M\times G_1$ is 14, then
only six out of the eight 1-forms $\Om_1,\Om_2,\Om_4,\Om_5,\Om_6,\Om_8,\Om_9,\Om_{11}$ may be independent. We now choose $(\theta^1,\theta^2,\dots,\theta^8,\Om_1,\Om_2,\Om_4,\Om_6,\Om_9,\Om_{11})$ as a basis of 1-forms on $M\times G_1$. This is an \emph{invariant} coframe on $M\times G_1$, so it is reasonable to ask if this defines a Cartan connection on $M\times G_1$.

The answer to this question is \emph{no}, since one can check, using our definitions (\ref{fo0}), (\ref{fo5}), (\ref{fo6}), (\ref{fo7}), (\ref{fo8}), with  (\ref{no1}), (\ref{no2}), (\ref{no3}), (\ref{no4}), (\ref{no5}), (\ref{no6}), (\ref{no7}), (\ref{no8}), (\ref{no9}), (\ref{no10}), (\ref{no11}), that:
$$\begin{aligned}
\der\Om_6=&2\Om_1\dz\Om_6-2\Om_4\dz\Om_6+\Om_6\dz\theta^8-\Om_9\dz\theta^8-\Om_{11}\dz\theta^5+\\&+F_1\Om_6\dz\theta^2+F_2\Om_6\dz\theta^5+F_3\Om_6\dz\theta^7,
\end{aligned}
$$
with
$$\begin{aligned}
  F_1=&\frac{12 u^3 u_3 a_{29}^5 a_{38}^2 + 2 u v_1 a_{29}^3 a_{38}^3 - 
   2 v_2 a_{26} a_{38}^6 + 2 u^2 a_{26} a_{29}^2 a_{38}^3 a_{46} - 
   u v_2 a_{29}^3 a_{38}^3 a_{46} + u^3 a_{29}^5 a_{46}^2}{4 u^3 a_{29}^5 a_{38}^2},\\
  F_2=&\frac{-2 u a_{26} a_{38}^3 + v_2 a_{29} a_{38}^3 + u^2 a_{29}^3 a_{46}}{u^2 a_{29}^3 a_{38}},\quad\quad\quad\quad\quad\quad\quad\quad
  F_3=\frac{v_2 a_{38}^3 + 2 u^2 a_{29}^2 a_{46}}{u^2 a_{29}^2 a_{38}}.
\end{aligned}$$
These are the last three terms in $\der\Om_6$ that prevent the coframe $(\theta^1,\theta^2,\dots,\theta^8,\Om_1,$ $\Om_2,\Om_4,\Om_6,\Om_9,\Om_{11})$ to be a basis for a Cartan connection on $M\times G_1$. This is because the forms $(\theta^1,\theta^2,\dots,\theta^8)$ are `horizontal', the forms $(\Om_1,$ $\Om_2,\Om_4,\Om_6,\Om_9,\Om_{11})$ are `vertical', and in the curvature of any Cartan connection the mixed `horizontal-vertical' terms,  $\Om_a\dz\theta^i$, as our $\Om_6\dz\theta^2$, $\Om_6\dz\theta^5$, and $\Om_6\dz\theta^7$ terms in the diffferential $\der\Om_6$, can only have \emph{constant} coefficients. This forces us to introduce \emph{three new normalizations}. These are made in such a way that $F_1$, $F_2$ and $F_3$ are \emph{constant}. Since the normalizations
\be F_1=F_2=F_3=0\label{ultn}\ee
are the simplest, and they do not contradict any group condition, such as e.g. $\det(A^i{}_j)\neq 0$, we chose them from now on.

Solving the normalization conditions (\ref{ultn}) gives
\be
u_3=-\frac{v_1a_{38}}{6u^2a_{29}^2},\quad\quad a_{26}=\frac{v_2a_{29}}{4u},\quad\quad a_{46}=-\frac{v_2a_{38}^3}{2u^2a_{29}^2}.\label{no12}\ee

So now, we are on 11-dimensional manifold, locally $M\times H$, where $H$ is parameterized by $b_5, a_{29}, a_{38}$. Taking into account all the normalizations 
(\ref{no1}), (\ref{no2}), (\ref{no3}), (\ref{no4}), (\ref{no5}), (\ref{no6}), (\ref{no7}), (\ref{no8}), (\ref{no9}), (\ref{no10}), (\ref{no11}) and (\ref{no12}) the forms $(\Om_1,\Om_2,\dots,\Om_{13})$ solving (\ref{kap}) read now:
\be\begin{aligned}
 \Om_1=&\frac{\der a_{29}}{a_{29}} - \frac{\der a_{38}}{2 a_{38}} -\frac{v_1a_{38}}{6u^2a_{29}^2} \theta^2 +\frac{v_2a_{38}^2}{4 u^2 a_{29}^2} \theta^5 + \frac{v_2a_{38}^2}{4 u^2 a_{29}^2}\theta^7\\
 \Om_2=&\frac{a_{29} \der b_{5}}{a_{38}} + \frac{v_1a_{38}}{3u^2a_{29}^2}\theta^1  -  \frac{v_2a_{38}^2}{4 u^2 a_{29}^2} \theta^4 -\frac{v_2a_{38}^2}{4 u^2 a_{29}^2} \theta^6\\
 \Om_3=&0\\
 \Om_4=&\frac{2 \der a_{29}}{a_{29}} - \frac{3 \der a_{38}}{2 a_{38}}- \frac{v_1a_{38}}{2u^2a_{29}^2} \theta^2 + \frac{v_2a_{38}^2}{2 u^2 a_{29}^2} \theta^5  + \frac{v_2a_{38}^2}{2 u^2 a_{29}^2}\theta^7\\
 \Om_5=&-\frac{\der a_{29}}{a_{29}} + \frac{3\der a_{38}}{2 a_{38}} +\frac{v_1a_{38}}{2u^2a_{29}^2}  \theta^2 -\frac{v_2a_{38}^2}{4 u^2 a_{29}^2} \theta^5 - \frac{v_2a_{38}^2}{4 u^2 a_{29}^2}\theta^7 - \theta^8\\
 \Om_6=&-t^4{}_{12}\theta^2-t^4{}_{15}\theta^5-t^4{}_{17} \theta^7\\
 \Om_7=&0\\
 \Om_8=&-\frac{2\der a_{29}}{a_{29}} + \frac{3\der a_{38}}{a_{38}} +\frac{v_1a_{38}}{u^2a_{29}^2}  \theta^2 -\frac{v_2a_{38}^2}{2 u^2 a_{29}^2} \theta^5 - \frac{v_2a_{38}^2}{2 u^2 a_{29}^2}\theta^7 - \theta^8\\
 \Om_9=&-t^4{}_{12}\theta^2-(3t^4{}_{15}-2t^4{}_{17})\theta^5-(2t^4{}_{15}-t^4{}_{17})\theta^7\\
 \Om_{10}=&0\\
 \Om_{11}=&r_{212}\theta^2-\tfrac43 t^4{}_{12}(\theta^5+\theta^7)\\
 \Om_{12}=&0\\
 \Om_{13}=&0.
\end{aligned}\label{fifo}\ee
In these expressions we used the definitions (\ref{in2}). We also have
$$\theta^i=A^i{}_j\om^j,$$
with $(\om^1,\om^2,\dots,\om^8)$ as in (\ref{sysc3u})-(\ref{sysc3u1}), and the matrix $(A^i{}_j)$, reduced by our normalizations from this in Proposition \ref{pr36} to:
$$A^i{}_j=\bma
  -\frac{ua_{29}^3}{a_{38}^2}&-\frac{ua_{29}^3b_5}{a_{38}^2}&0&0&0&0&0&0\\
  0&-\frac{ua_{29}^3}{a_{38}^2}&0&0&0&0&0&0\\
  -\frac{2ua_{29}^4b_5}{a_{38}^3}&-\frac{ua_{29}^4b_5^2}{a_{38}^3}&-\frac{ua_{29}^4}{a_{38}^3}&0&0&0&0&0\\
  -\frac{v_2a_{29}}{4u}&-\frac{v_2a_{29}b_5}{4u}&0&-a_{29}&-a_{29}b_5&0&0&0\\
  0&-\frac{v_2a_{38}}{4u}&0&0&-a_{38}&0&0&0\\
  \frac{v_2a_{29}}{4u}&\frac{v_2a_{29}b_5}{4u}&0&a_{29}&a_{29}b_5&\frac{ua_{29}^3}{a_{38}^3}&\frac{ua_{29}^3b_5}{a_{38}^3}&0\\
  0& \frac{v_2a_{38}}{4u}&0&0&a_{38}&0&\frac{ua_{29}^2}{a_{38}^2}&0\\
   0&-\frac{v_2^2a_{38}^3}{16u^3a_{29}^2}&0&0&-\frac{v_2a_{38}^3}{2u^2a_{29}^2}&0&0&-\frac{a_{38}^3}{ua_{29}^2}
   \ema.$$

One can now easilly check that $\Om_5$ and $\Om_8$ satisfy
$$\Om_5=3\Om_1-2\Om_4-\theta^8,\quad\quad\Om_8=6\Om_1-4\Om_4-\theta^8,$$
so that all the relations (\ref{coco1}) from Corollary \ref{coco} are proven.

To see that the system of forms $(\theta^1,\theta^2,\dots,\theta^8,\Om_1,\Om_2,\Om_4)$ defines a princial connection on $M\times H$ we first note that these forms constitute a coframe on  $M\times H$, and then calculate their differentials. These are:
\be
\begin{aligned}
\der\theta^1&=\Om_1\dz\theta^1+\Om_2\dz\theta^2+\Om_4\dz\theta^1-\theta^4\dz\theta^7-\theta^5\dz\theta^6,\\
\der\theta^2&=2\Om_1\dz\theta^2-2\theta^5\dz\theta^7,\\
\der\theta^3&=2\Om_2\dz\theta^1+2\Om_4\dz\theta^3-2\theta^4\dz\theta^6+\theta^2\dz\theta^5+\theta^2\dz\theta^7,\\
\der\theta^4&=3\Om_1\dz\theta^4+\Om_2\dz\theta^5-\Om_4\dz\theta^4+t^4_{~12}\theta^1\dz\theta^2+t^4_{~15}\theta^1\dz\theta^5+t^4_{~17}\theta^1\dz\theta^7\\&+\theta^4\dz\theta^8+\theta^6\dz\theta^8,\\
\der\theta^5&=4\Om_1\dz\theta^5-2\Om_4\dz\theta^5+t^4_{~15}\theta^2\dz\theta^5+t^4_{~17}\theta^2\dz\theta^7+\theta^5\dz\theta^8+\theta^7\dz\theta^8,\\
\der\theta^6&=-6\Om_1\dz\theta^4-3\Om_1\dz\theta^6+\Om_2\dz\theta^7+4\Om_4\dz\theta^4+3\Om_4\dz\theta^6\\&-t^4_{~12}\theta^1\dz\theta^2+(2t^4_{~17}-3t^4_{~15})\theta^1\dz\theta^5+(t^4_{~17}-2t^4_{~15})\theta^1\dz\theta^7
-\theta^4\dz\theta^8-\theta^6\dz\theta^8,\\
\der\theta^7&=-6\Om_1\dz\theta^5-2\Om_1\dz\theta^7+4\Om_4\dz\theta^5+2\Om_4\dz\theta^7+(2t^4_{~17}-3t^4_{~15})\theta^2\dz\theta^5\\&
+(t^4_{~17}-2t^4_{~15})\theta^2\dz\theta^7
-\theta^5\dz\theta^8-\theta^7\dz\theta^8,\\
\der\theta^8&=6\Om_1\dz\theta^8-4\Om_4\dz\theta^8-2t^4_{~12}\theta^2\dz\theta^5+2t^4_{~17}\theta^5\dz\theta^7,\\
\der\Om_1&=-\tfrac13t^4_{~12}\theta^2\dz\theta^5-\tfrac13t^4_{~12}\theta^2\dz\theta^7+(t^4_{~17}-t^4_{~15})\theta^5\dz\theta^7,\\
\der\Om_2&=-\Om_1\dz\Om_2-\Om_2\dz\Om_4+r_{212}\theta^1\dz\theta^2-\tfrac43t^4_{~12}\theta^1\dz\theta^5-\tfrac43t^4_{~12}\theta^1\dz\theta^7\\&+t^4_{~12}\theta^2\dz\theta^4+t^4_{~12}\theta^2\dz\theta^6+(2t^4_{~17}-3t^4_{~15})\theta^4\dz\theta^5\\&+(t^4_{~17}-2t^4_{~15})\theta^4\dz\theta^7+t^4_{~15}\theta^5\dz\theta^6-t^4_{~17}\theta^6\dz\theta^7,\\
\der\Om_4&=0.
\end{aligned}\label{syuu}
\ee
Now, one can interprete these equations as follows:

We have a 3-dimensional freedeom - related to the residual group parameters $a_{29}, a_{38}, b_5$ - in defining the forms $(\theta^1,\dots,\theta^8,\Om_1,\Om_2,\Om_4)$. For example, the simplest choice $a_{29}=a_{38}=1,b_5=0$,
defines a section $$\sigma:M\to M\times H,\quad\mathrm{with}\quad\sigma(x)=(x,a_{29}=1,a_{38}=1,b_5=0)$$ of the bundle $M\times H\to M$, and 
a coframe 
$$\bar{\omega}{}^i:=\sigma^*(\theta^i),\quad i=1,2,\dots, 8,$$
on $M$. Having chosen the section $\sigma$ we also have the corresponding 1-forms $$\bar{\Om}{}_1=\sigma^*(\Om_1),\quad\bar{\Om}{}_2=\sigma^*(\Om_2),\quad\bar{\Om}{}_4=\sigma^*(\Om_4)$$ on $M$. Explicitly we have:
\be\begin{aligned}
&\bar{\om}{}^1=-u\om^1\\
&\bar{\om}{}^2=-u\om^2\\
&\bar{\om}{}^3=-u\om^3\\
&\bar{\om}{}^4=-\om^4-\frac{v_2}{4u}\om^1\\
&\bar{\om}{}^5=-\om^5-\frac{v_2}{4u}\om^2\\
&\bar{\om}{}^6=u\om^6+\om^4+\frac{v_2}{4u}\om^1\\
&\bar{\om}{}^7=u\om^7+\om^5+\frac{v_2}{4u}\om^2\\
&\bar{\om}{}^8=-\frac{1}{u}\om^8-\frac{v_2}{2u^2}\om^5-\frac{v_2^2}{16u^3}\om^2\\
\end{aligned}\ee
and
\be\begin{aligned}
&\bar{\Om}_1=\frac{v_1}{6u^2}\bar{\om}^2-\frac{v_2}{4u^2}\bar{\om}^5-\frac{v_2}{4u^2}\bar{\om}^7,\\
&\bar{\Om}_2=-\frac{v_1}{3u^2}\bar{\om}^1+\frac{v_2}{4u^2}\bar{\om}^4+\frac{v_2}{4u^2}\bar{\om}^6,\\
&\bar{\Om}_4=\frac{v_1}{2u^2}\bar{\om}^2-\frac{v_2}{2u^2}\bar{\om}^5-\frac{v_2}{2u^2}\bar{\om}^7,\\
\end{aligned}\label{omegabar}\ee
We now collect the last three 1-forms to an $8\times 8$  matrix
{\tiny\be
(\bar{\gamma}{}^i{}_j)=\bma
-\bar{\Om}_1-\bar{\Om}_4&-\bar{\Om}_2&0&0&0&0&0&0\\0&-2\bar{\Om}_1&0&0&0&0&0&0\\-2\bar{\Om}_2&0&-2\bar{\Om}_4&0&0&0&0&0\\0&0&0&-3\bar{\Om}_1+\bar{\Om}_4&-\bar{\Om}_2&0&0&0\\0&0&0&0&-4\bar{\Om}_1+2\bar{\Om}_4&0&0&0\\0&0&0&6\bar{\Om}_1-4\bar{\Om}_4&0&3\bar{\Om}_1-3\bar{\Om}_4&-\bar{\Om}_2&0\\0&0&0&0&6\bar{\Om}_1-4\bar{\Om}_4&0&2\bar{\Om}_1-2\bar{\Om}_4&0\\
  0&0&0&0&0&0&0&-6\bar{\Om}_1+4\bar{\Om}_4
  \ema,\label{s2}\ee}
on $M$. 
Due to the first eight equations (\ref{syuu}), evaluated at $a_{29}=a_{38}=1,b_5=0$, this matrix of 1-forms satisfies: 
{\tiny \be
(\der\bar{\omega}^i+\bar{\gamma}^i{}_j\dz\bar{\omega}^j)=\bma
-\bar{\om}^4\dz\bar{\om}^7-\bar{\om}^5\dz\bar{\om}^6\\
-2\bar{\om}^5\dz\bar{\om}^7\\
\bar{\om}^2\dz\bar{\om}^5+\bar{\om}^2\dz\bar{\om}^7-2\bar{\om}^4\dz\bar{\om}^6\\
s_1\bar{\om}^1\dz\bar{\om}^2+s_2\bar{\om}^1\dz\bar{\om}^5-s_3\bar{\om}^1\dz\bar{\om}^7+\bar{\om}^4\dz\bar{\om}^8+\bar{\om}^6\dz\bar{\om}^8\\
s_4\bar{\om}^2\dz\bar{\om}^5-s_5\bar{\om}^2\dz\bar{\om}^7+\bar{\om}^5\dz\bar{\om}^8+\bar{\om}^7\dz\bar{\om}^8\\
-s_6\bar{\om}^1\dz\bar{\om}^2-s_7\bar{\om}^1\dz\bar{\om}^5+s_8\bar{\om}^1\dz\bar{\om}^7-\bar{\om}^4\dz\bar{\om}^8-\bar{\om}^6\dz\bar{\om}^8\\
-s_9\bar{\om}^2\dz\bar{\om}^5-s_{10}\bar{\om}^2\dz\bar{\om}^7-\bar{\om}^5\dz\bar{\om}^8-\bar{\om}^7\dz\bar{\om}^8\\
-s_{11}\bar{\om}^2\dz\bar{\om}^5-s_{12}\bar{\om}^5\dz\bar{\om}^7
\ema:=\bar{\mathcal T}{}^i,\label{s3}\ee}
with functions $s_1,s_2,\dots,s_{12}$ on $M$ given by:
$$\begin{aligned}
  &s_1=s_6=\tfrac12s_{11}=-\frac{I_3}{8u^4},\quad s_2=s_4=\frac{I_2}{6u^2}+\frac{I_1}{16u^4},\quad s_3=s_5=\tfrac12s_{12}=-\frac{I_1}{16u^4},\\
  &\quad\quad s_7=s_2+\frac{I_2}{3u^2},\quad s_8=s_3-\frac{I_2}{3u^2},\\
  &\quad\quad\quad\quad s_9=3s_2+2s_3,\quad s_{10}=2s_2+s_3,
\end{aligned}$$
where $I_1,I_2,I_3$ are as in (\ref{in1}).

The last three equations (\ref{syuu}) mean that:
\be
\begin{aligned}
  \bar{\mathcal K}{}^i{}_j=&\der\bar{\gamma}{}^i{}_j+\bar{\gamma}{}^i{}_k\dz\bar{\gamma}{}^k{}_j=\\
&{\tiny \bma
\bar{K}_1+\bar{K}_4&\bar{K}_2&0&0&0&0&0&0\\0&2\bar{K}_1&0&0&0&0&0&0\\2\bar{K}_2&0&2\bar{K}_4&0&0&0&0&0\\0&0&0&3\bar{K}_1-\bar{K}_4&\bar{K}_2&0&0&0\\0&0&0&0&4\bar{K}_1-2\bar{K}_4&0&0&0\\0&0&0&-6\bar{K}_1+4\bar{K}_4&0&-3\bar{K}_1+3\bar{K}_4&\bar{K}_2&0\\0&0&0&0&-6\bar{K}_1+4\bar{K}_4&0&-2\bar{K}_1+2\bar{K}_4&0\\
  0&0&0&0&0&0&0&6\bar{K}_1-4\bar{K}_4
  \ema}
  \end{aligned}\label{krzho}\ee
with
\be\begin{aligned}
  \bar{K}_1=&-\frac{I_3}{24u^4}\bar{\om}{}^2\dz\bar{\om}{}^5-\frac{I_3}{24u^4}\bar{\om}{}^2\dz\bar{\om}{}^7+\frac{I_2}{6u^2}\bar{\om}{}^5\dz\bar{\om}{}^7,\\
  \bar{K}_2=&\frac{R}{9u^4}\bar{\om}{}^1\dz\bar{\om}{}^2-\frac{I_3}{6u^4}\bar{\om}{}^1\dz\bar{\om}{}^5-\frac{I_3}{6u^4}\bar{\om}{}^1\dz\bar{\om}{}^7+\frac{I_3}{8u^4}\bar{\om}{}^2\dz\bar{\om}{}^4+\frac{I_3}{8u^4}\bar{\om}{}^2\dz\bar{\om}{}^6+\\
  &\frac{I_1+8I_2u^2}{16u^4}\bar{\om}{}^4\dz\bar{\om}{}^5+\frac{3I_1+16I_2u^2}{48u^4}\bar{\om}{}^4\dz\bar{\om}{}^7-\frac{3I_1+8I_2u^2}{48u^4}\bar{\om}{}^5\dz\bar{\om}{}^6+\\&\frac{I_1}{16u^4}\bar{\om}{}^6\dz\bar{\om}{}^7,\\
  \bar{K}_4=&0.
  \end{aligned}\label{krzHo1}
\ee

As a final step in the proof we reparameterize $H$ with:
\be a_{29}=a^3b,\quad a_{38}=a^2b,\quad b_5=c,\quad\quad bc\neq 0.\label{defc}\ee
With these definitions, one can then check that the forms $(\theta^1,\theta^2,\dots,\theta^8,\Om_1,\Om_2,\Om_4)$, which appear in (\ref{fifo}), and which satisfy the system (\ref{syuu}), are given by
\be
\begin{aligned}
  &\theta^i={\mathcal A}^i{}_j\bar{\om}^j,\\
  &\gamma^i{}_j={\mathcal A}^i{}_k\bar{\gamma}{}^k{}_l({\mathcal A}^{-1})^l{}_j-\der {\mathcal A}^i{}_k({\mathcal A}^{-1})^k{}_j.\end{aligned}\label{conh}\ee
Here the matrix $({\mathcal A}^i{}_j)$ is as in (\ref{mah}),
and the forms $(\Om_1,\Om_2,\Om_4)$ can be red off from the matrix $\gamma^i{}_j$ given in (\ref{conh}) via the definition (\ref{coco2}).
   
In view of this we have now placed the forms $(\theta^1,\theta^2,\dots,\theta^8,\Om_1,\Om_2,\Om_4)$ on an 11-dimensional (trivial) fiber bundle $H\to H\times M\to M$, with the fiber $H$ beeing identified as a 3-dimensional Lie group of $8\times 8$ real matrices:
$$H=\{~\glg(8,\bbR)\ni ({\mathcal A}^i{}_j)~|~ ({\mathcal A}^i{}_j)\,\mathrm{as\,in}\, {\mathrm (\ref{mah})}~\}.$$
This bundle becomes a $H$-principal fiber bundle over $M$. In this interpretation, the $\mathfrak{h}$-valued \emph{matrix} of 1-forms $(\gamma^i{}_j)$, as defined in (\ref{conh}), (\ref{s2}), (\ref{mah}) and (\ref{defc}), \emph{becomes a principal connection} on $H\to H\times M\to M$. The connection $\gamma$ in general has \emph{nonzero torsion}
$${\mathcal T}^i=\der\theta^i+\gamma^i{}_{j}\dz\theta^j={\mathcal A}^i{}_k\bar{\mathcal T}{}^k,$$
since $\bar{\mathcal T}{}^k$, defined in (\ref{s3}), is not zero whenever the basic invariants $I_1$, $I_2$ and $I_3$ are not zero,
$$|I_1|+|I_2|+|I_3|\neq 0.$$
The curvature ${\mathcal K}=({\mathcal K}^i{}_{j})$ of $\gamma$ is $${\mathcal K}^i{}_j=\der\gamma^i{}_j+\gamma^i{}_k\dz\gamma^k{}_j={\mathcal A}^i{}_k\bar{\mathcal K}{}^k{}_l({\mathcal A}^{-1})^l{}_j,$$
and again is nonzero, whenever
$$|I_1|+|I_2|+|I_3|+|R|\neq 0,$$
as can be seen from (\ref{krzho})-(\ref{krzHo1}). These equations show also that the curvature ${\mathcal K}$ of $\gamma$ is identically zero, or to say it differently: that the principal connection $\gamma$ is flat, if and only if the basic invariants $I_1= I_2= I_3= R= 0$. This justifies the name `basic invariants' atributed in (\ref{in1}) to $I_1,I_2,I_3$ and $R$. 

This finishes the proof of the Theorem and the Corollary.
\end{proof}

\subsection{The flat model here: $I_1= 0$. Eleven dimensional symmetry group}
Looking at the formulae (\ref{in2}), we see that, for example $I_1$, is a relative invariant of the system (\ref{sysc3u}). Thus in enumarting nonequivalent structures among the ones described by (\ref{sysc3u}), we have to distinguish between these with $I_1= 0$, and these with $I_1\neq 0$. Any structure with the $I_1=0$ branch is definitely nonequivalent to any structure from the $I_1\neq 0$ branch. Concentrating on the $I_1= 0$ for a while, we have the following proposition.
\begin{proposition}
Every structure (\ref{sysc3u}) with the relative invariant 
$$I_1= 0$$
is locally equivalent to the structure (\ref{sysc3u}) with
$$u= 1.$$
The Monge system corresponding to this unique structure is given by
\be\dot{z}_{11}=\dot{x}_1^2,\quad \dot{z}_{12}=\dot{x}_1\dot{x}_2,\quad z_{22}=\dot{x}_2^2+\tfrac13 \dot{x}_1^3.\label{di111}\ee
The corresponding distribution
\be
\begin{aligned}
{\mathcal D}=\Span\Big(\partial_t&+\dot{x}_1^2\partial_{z_{11}}+\dot{x}_1\dot{x}_2\partial_{z_{12}}+(\dot{x}_2^2+\tfrac13\dot{x}_1^3)\partial_{z_{22}}+\dot{x}_1\partial_{x_1}+\dot{x}_2\partial_{x_2},\\&\partial_{\dot{x}_1},\quad\partial_{\dot{x}_2}\Big)\end{aligned}\label{di11}\ee
has 11-dimensional group of symmetries.
\end{proposition}
\begin{proof}
Assuming that $I_1=0$ and $u\neq 0$ means, via formula (\ref{in1}), that $$w_3=\frac{5v_2^2}{4u}.$$
  Then the first, the third and the sixth from equations (\ref{sysc3u1}) show
  that
  \be\begin{aligned}
  -2&v_1\om^8+\\&\frac{10u v_2 w_2-4u^2 z_3-5v_1 v_2^2}{4u^2}\om^2+
  \frac{5v_1v_2-4uw_2}{u}\om^5+\frac{15v_2^3-8u^2z_4}{8u^2}\om^7=0.\end{aligned}\label{eqa}\ee
  Since the $\om^2,\om^5,\om^7\om^8$ are parts of the coframe on $M$ than, we in particular, have
  $$v_1=0.$$
  This, when compared with the second, the fourth and the fifth of equations (\ref{sysc3u1}) gives
  $$w_1=w_2=z_1=z_2=z_3=0.$$
  Finally equation (\ref{eqa}) gives
  $$z_4=\frac{15v_2^3}{8u^2}.$$
  Summarizing, we see that the system with $I_1=0$ is given on $M$ by means of a coframe $(\om^1,\om^2,\dots,\om^8)$ satisfying (\ref{sysc3u}) and
  \be\der u= v_2\om_7,\quad \&\quad \der v_2=\frac{5v_2^2}{4u}\om^7.\label{eqb}\ee
  It is a matter of checking that the so defined system (\ref{sysc3u}), (\ref{eqb}) is closed: $\der^2\om^i= 0$ for all $i=1,2,\dots,8$, $\der^2u=\der^2v_2=0$, and \emph{more imporatntly}, that in addition to $I_1= 0$, it also satisfies $I_2= I_3=R= 0$.

  So what we have now proven is that
  $$(~I_1= 0~)\Longrightarrow (~I_2\equiv I_3= R= 0~),$$
  so in turn we have proven that if $I_1= 0$, all the relative invariants of the Monge structure with $u\neq 0$ identicaly vanish. Thus all such systems are equivalent to the invariant system
\be
\begin{aligned}
&\der\theta^1=\Om_1\dz\theta^1+\Om_2\dz\theta^2+\Om_4\dz\theta^1-\theta^4\dz\theta^7-\theta^5\dz\theta^6,\\
&\der\theta^2=2\Om_1\dz\theta^2-2\theta^5\dz\theta^7,\\
&\der\theta^3=2\Om_2\dz\theta^1+2\Om_4\dz\theta^3-2\theta^4\dz\theta^6+\theta^2\dz\theta^5+\theta^2\dz\theta^7,\\
&\der\theta^4=3\Om_1\dz\theta^4+\Om_2\dz\theta^5-\Om_4\dz\theta^4+\theta^4\dz\theta^8+\theta^6\dz\theta^8,\\
&\der\theta^5=4\Om_1\dz\theta^5-2\Om_4\dz\theta^5+\theta^5\dz\theta^8+\theta^7\dz\theta^8,\\
&\der\theta^6=-6\Om_1\dz\theta^4-3\Om_1\dz\theta^6+\Om_2\dz\theta^7+4\Om_4\dz\theta^4+3\Om_4\dz\theta^6
-\theta^4\dz\theta^8-\theta^6\dz\theta^8,\\
&\der\theta^7=-6\Om_1\dz\theta^5-2\Om_1\dz\theta^7+4\Om_4\dz\theta^5+2\Om_4\dz\theta^7
-\theta^5\dz\theta^8-\theta^7\dz\theta^8,\\
&\der\theta^8=6\Om_1\dz\theta^8-4\Om_4\dz\theta^8,\\
&\der\Om_1=0,\\
&\der\Om_2=-\Om_1\dz\Om_2-\Om_2\dz\Om_4,\\
&\der\Om_4=0,
\end{aligned}
\label{sys11}\ee
obtained from (\ref{syuu}), by putting $I_1=I_2=I_3=R=0$.

This system is a system of Maurer-Cartan equations for the Maurer-Cartan forms $(\theta^1,\theta^2,\dots,\theta^8,\Om_1,\Om_2,\Om_4)$ on a 11-dimensional \emph{group} $G^{11}$. In particular, the structure constants of this group
in the Maurer-Cartan basis $(\theta^1,\theta^2,\dots,\theta^8,\Om_1,\Om_2,\Om_4)$ can be easily read off from (\ref{sys11}). The construction given in Theorem \ref{ft} and Corollary \ref{coco} shows that $G^{11}$ is a fiber bundle $H\to G^{11}\to M=G^{11}/H$, in which the base manifold $M$ is a homogeneous space $M=G^{11}/H$. This is equipped with a Monge structure having $u\neq 0$ and $I_1=0$. The structure has 11-dimensional symmetry group equal to $G^{11}$. Since in equations (\ref{sys11}) the variable $u$ is not present, it means that \emph{all} systems (\ref{sysc3u}), (\ref{eqb}) with $u\neq 0$ are equivalent. In particular they are equivalent to the structure with $u=1$.

In Theorem \ref{exa} we have solved the structural equations (\ref{sysc3u}) for the Monge systems with $u\neq 0$ in full generality. Then in Proposition \ref{exu1} we have found a subclass of Monge systems with $I_2=I_3=R=0$ and $u=\tfrac12 h_{\dot{x}_1\dot{x}_1\dot{x}_1}$ parametrized by one function $h$ of one variable $\dot{x}_1$. In particular, this class contains an example of a Monge system with $u=1$. This corresponds to $h=\tfrac13\dot{x}_1^3$. Quick look at the invariant $I_1$ in Proposition \ref{exu1} calculated for $h=\tfrac13\dot{x}_1^3$, shows that $I_1=0$. Thus, an example from Proposition \ref{exu1} with $h=\tfrac13\dot{x}_1^3$ has $I_1=0$ and $u=1$. It therefore gives a unique (up to local equivalence) Monge structure with $I_1=0$ and $u\neq 0$. Proposition \ref{exu1} shows that the corresponding Monge ODEs are as claimed. This finishes the proof. 
\end{proof}
\subsection{Structures with $I_1\neq 0$}
If $I_1\neq 0$ then the coeeficient
$$t^4{}_{17}=\frac{I_1}{16u^4a^4}$$
in (\ref{syuu}) is not equal to zero. Here $a$ is a group parameter in ${\mathcal A}\in H$ as in (\ref{mah}). 
\bigbreak

This enables us to make a \emph{new normalization}:
$$t^4{}_{17}=1.$$
This choses a hypersurface 
$$Q=\{H\times M\ni (a,b,c,x)\quad |\quad a=\frac{(\epsilon I_1)^{1/4}}{2u},\quad\epsilon=\mathrm{sign}(I_1)\},$$
in $H\to H\times M\to M$ transversal to the fibers. This reduces the system (\ref{syuu}) to \emph{ten} dimensions, where only \emph{ten} out of the eleven forms $(\theta^1,\theta^2,\dots,\theta^8,\Om_1,\Om_2,\Om_4)$ are linearly independent.

In particular on this 10-dimensional manifold $Q$ we have:
$$\begin{aligned}
  \Om_4=&\Om_1+\frac{2u^2(10I_1I_2+42 I_3v_2-63I_2v_2^2+24u^2z_3)}{3b\epsilon I_1^2}\theta^2-\\
  &\frac{16I_3u^2+9bI_1v_2-15bv_2^3+8bu^2z_4}{2b\sqrt{\epsilon I_1^3}}\theta^5-\frac{9I_1v_2-15v_2^3+8u^2z_4}{2\sqrt{\epsilon I_1^3}}\theta^7+\frac{2I_2u^2}{bI_1}\theta^8.\end{aligned}
$$
Furthermore, on this 10-dimensional hypersurface $H\times M\subset Q\to M$ the first equation in the system (\ref{syuu}) becomes:
\be
\begin{aligned}
  \der\theta^1&=2\Om_1\dz\theta^1+\Om_2\dz\theta^2-\frac{2u^2(10I_1I_2+42 I_3v_2-63I_2v_2^2+24u^2z_3)}{3b\epsilon I_1^2}\theta^1\dz\theta^2+\\&\frac{16I_3u^2+9bI_1v_2-15bv_2^3+8bu^2z_4}{2b\sqrt{\epsilon I_1^3}}\theta^1\dz\theta^5+\frac{9I_1v_2-15v_2^3+8u^2z_4}{2\sqrt{\epsilon I_1^3}}\theta^1\dz\theta^7-\\&\frac{2I_2u^2}{bI_1}\theta^1\dz\theta^8 -\theta^4\dz\theta^7-\theta^5\dz\theta^6.
\end{aligned}\label{syun}
\ee
It is now convenient to introduce abreviations
\be \begin{aligned}S_1=&\frac{2I_2u^2}{bI_1},\quad S_2=\frac{8I_3u^2}{b\sqrt{\epsilon I_1^3}},\quad S_3=\frac{4u^2(10I_1I_2+42 I_3v_2-63I_2v_2^2+24u^2z_3)}{3b\epsilon I_1^2},\\& S_4=\frac{9I_1v_2-15v_2^3+8u^2z_4}{2\sqrt{\epsilon I_1^3}}\end{aligned},\label{fsa}\ee
and a 1-form $\Theta$ related to the 1-form $\Om_2$ via:
$$\Theta=\Om_2+\tfrac12 S_3\theta^1-(S_2+S_4)\theta^4-S_4\theta^6.$$
Then the equations (\ref{syuu}) pulbackked to $Q\subset(H\times M)$ are:
 \be
\begin{aligned}
\der\theta^1=&2\Om_1\dz\theta^1+\Theta\dz\theta^2-S_3\theta^1\dz\theta^2+(S_2+S_4)\theta^1\dz\theta^5+\\&S_4\theta^1\dz\theta^7-S_1\theta^1\dz\theta^8-(S_2+S_4)\theta^2\dz\theta^4-S_4\theta^2\dz\theta^6-\\&\theta^4\dz\theta^7-\theta^5\dz\theta^6,\\
\der\theta^2=&2\Om_1\dz\theta^2-2\theta^5\dz\theta^7,\\
\der\theta^3=&2\Om_1\dz\theta^3+2\Theta\dz\theta^1-2(S_2+S_4)\theta^1\dz\theta^4-2S_4\theta^1\dz\theta^6+\\&S_3\theta^2\dz\theta^3+\theta^2\dz\theta^5+\theta^2\dz\theta^7+2(S_2+S_4)\theta^3\dz\theta^5+\\&2S_4\theta^3\dz\theta^7-2S_1\theta^3\dz\theta^8-2\theta^4\dz\theta^6,\\
\der\theta^4=&2\Om_1\dz\theta^4+\Theta\dz\theta^5-\epsilon S_2\theta^1\dz\theta^2+(\epsilon+\tfrac43\epsilon S_1-\tfrac12S_3)\theta^1\dz\theta^5+\\&\epsilon \theta^1\dz\theta^7-\tfrac12S_3\theta^2\dz\theta^4-S_4\theta^4\dz\theta^7+(1+S_1)\theta^4\dz\theta^8-\\&S_4\theta^5\dz\theta^6+\theta^6\dz\theta^8,\\
\der\theta^5=&2\Om_1\dz\theta^5+(\epsilon+\tfrac43\epsilon S_1-S_3)\theta^2\dz\theta^5+\epsilon \theta^2\dz\theta^7-2S_4\theta^5\dz\theta^7+\\&(1+2S_1)\theta^5\dz\theta^8+\theta^7\dz\theta^8,\\
\der\theta^6=&-2\Om_1\dz\theta^4+\Theta\dz\theta^7+\epsilon S_2\theta^1\dz\theta^2-\epsilon(1+4S_1) \theta^1\dz\theta^5-\\&(\epsilon+\tfrac83\epsilon S_1+\tfrac12S_3) \theta^1\dz\theta^7+2S_3\theta^2\dz\theta^4+\tfrac32S_3\theta^2\dz\theta^6+\\&4(S_2+S_4)\theta^4\dz\theta^5+(S_2+5S_4)\theta^4\dz\theta^7-(1+4S_1)\theta^4\dz\theta^8-\\&3(S_2+S_4)\theta^5\dz\theta^6
+4S_4\theta^6\dz\theta^7-(1+3S_1)\theta^6\dz\theta^8,\\
\der\theta^7=&-2\Om_1\dz\theta^5-(\epsilon+4\epsilon S_1-2S_3)\theta^2\dz\theta^5-(\epsilon+\tfrac83\epsilon S_1-S_3) \theta^2\dz\theta^7-\\&2(S_2-S_4)\theta^5\dz\theta^7-(1+4S_1)\theta^5\dz\theta^8-(1+2S_1)\theta^7\dz\theta^8,\\
\der\theta^8=&2\Om_1\dz\theta^8+2\epsilon S_2\theta^2\dz\theta^5-2S_3\theta^2\dz\theta^8+2\epsilon \theta^5\dz\theta^7+\\&4(S_2+S_4)\theta^5\dz\theta^8+4S_4\theta^7\dz\theta^8,\\
\der\Om_1=&\tfrac13\epsilon S_2\theta^2\dz\theta^5+\tfrac13\epsilon S_2\theta^2\dz\theta^7-\tfrac43\epsilon S_1\theta^5\dz\theta^7,\end{aligned}
\label{ss2}\ee
and
\be
\der\Theta=-S_1\Theta\dz\theta^8+\sum s _{ij}\theta^i\dz\theta^j,
\label{ss3}\ee
where the terms involving $s_{ij}$, $i<j=1,2,\dots, 8$ are not so important.

What is important is that the coefficient at the $\Theta\dz\theta^8$ term in the last equation, as well as the coefficients at the $\theta^1\dz\theta^8$ term in the equation for $\der\theta^1$ is
\be -S_1=-\frac{2I_2u^2}{bI_1}.\label{ss17}\ee
The importance of this observation is that it shows that in the currently considered situation, when $I_1\neq0$, we have to distinguish two cases. Either
\begin{itemize}
\item[a)] $I_2=0$, or
\item[b)] $I_2\neq 0$.
\end{itemize}
Obviously the Monge structures belonging to one of this cases are nonequivalent from the structures form the other case.

We now introduce an $(8\times 8)$-matrix-valued 1 form
$$
(\Gamma{}^i{}_j)=\bma
-2\Om_1&-\Theta&0&0&0&0&0&0\\0&-2\Om_1&0&0&0&0&0&0\\-2\Theta&0&-2\Om_1&0&0&0&0&0\\0&0&0&-2\Om_1&-\Theta&0&0&0\\0&0&0&0&-2\Om_1&0&0&0\\0&0&0&2\Om_1&0&0&-\Theta&0\\0&0&0&0&2\Om_1&0&0&0\\
  0&0&0&0&0&0&0&-2\Om_1
  \ema.$$
  With this notation, the first eight equations (\ref{ss2}) become:
  $$\der\theta^i+\Gamma^i{}_j\dz\theta^j=T^i,$$
  where the `torsion' $(T^i)$ has coefficients:
  $$\begin{aligned}
    T^1=& -S_3\theta^1\dz\theta^2+(S_2+S_4)(\theta^1\dz\theta^5-\theta^2\dz\theta^4)+S_4(\theta^1\dz\theta^7-\theta^2\dz\theta^6)-S_1\theta^1\dz\theta^8-\\&\theta^4\dz\theta^7-\theta^5\dz\theta^6,\\
 T^2=& -2\theta^5\dz\theta^7,\\
T^3=&2(S_2+S_4)(\theta^3\dz\theta^5-\theta^1\dz\theta^4)+2S_4(\theta^3\dz\theta^7-\theta^1\dz\theta^6)+S_3\theta^2\dz\theta^3-2S_1\theta^3\dz\theta^8+\\&\theta^2\dz\theta^5+\theta^2\dz\theta^7-2\theta^4\dz\theta^6,\\
T^4=&-S_2\epsilon\theta^1\dz\theta^2+(\epsilon+\tfrac43S_1\epsilon-\tfrac12S_3)\theta^1\dz\theta^5-S_4(\theta^4\dz\theta^7+\theta^5\dz\theta^6)+\epsilon\theta^1\dz\theta^7-\\&\tfrac12S_3\theta^2\dz\theta^4+(1+S_1)\theta^4\dz\theta^8+\theta^6\dz\theta^8,\\
T^5=&-2S_4\theta^5\dz\theta^7+(\epsilon+\tfrac43S_1\epsilon-S_3)\theta^2\dz\theta^5+\epsilon \theta^2\dz\theta^7+
(1+2S_1)\theta^5\dz\theta^8+\theta^7\dz\theta^8,\\
T^6=&S_2\epsilon\theta^1\dz\theta^2-(1+4S_1)\epsilon\theta^1\dz\theta^5-(\epsilon+\tfrac83\epsilon S_1+\tfrac12S_3)\theta^1\dz\theta^7+\\&2S_3(\theta^2\dz\theta^4+\tfrac34\theta^2\dz\theta^6)+(S_2+S_4)(4\theta^4\dz\theta^5-3\theta^5\dz\theta^6)+
  4S_4\theta^6\dz\theta^7+\\&(S_2+5S_4)\theta^4\dz\theta^7-(1+4S_1)\theta^4\dz\theta^8-(1+3S_1)\theta^6\dz\theta^8,\\
 T^7=& 2(S_4-S_2)\theta^5\dz\theta^7-(\epsilon+4S_2\epsilon-2S_3)\theta^2\dz\theta^5-(\epsilon+\tfrac83S_1\epsilon-S_3)\theta^2\dz\theta^7-\\&(1+4S_1)\theta^5\dz\theta^8-(1+2S_1)\theta^7\dz\theta^8,\\
 T^8=& 2S_2\epsilon\theta^2\dz\theta^5-2S_3\theta^2\dz\theta^8+4(S_2+S_4)\theta^5\dz\theta^8+4S_4\theta^7\dz\theta^8+2\epsilon\theta^5\dz\theta^7.
\end{aligned}$$
  On the other hand, the last equation (\ref{ss2}) and equation (\ref{ss3}) can be colectively written as:
  $$\der\Gamma^i{}_j+\Gamma^i{}_k\dz\Gamma^k{}_j=K^i{}_j,$$
  with
  $$\begin{aligned}
(K{}^i{}_j)=&S_1\tiny{\bma
0&1&0&0&0&0&0&0\\0&0&0&0&0&0&0&0\\2&0&0&0&0&0&0&0\\0&0&0&0&1&0&0&0\\0&0&0&0&0&0&0&0\\0&0&0&0&0&0&1&0\\0&0&0&0&0&0&0&0\\
  0&0&0&0&0&0&0&0
  \ema}\Theta\dz\theta^8+\\&\quad\quad\quad\quad\quad\quad\mathrm{terms~of ~the~form}~\tiny{\bma
*&**&0&0&0&0&0&0\\0&*&0&0&0&0&0&0\\2**&0&*&0&0&0&0&0\\0&0&0&*&**&0&0&0\\0&0&0&0&*&0&0&0\\0&0&0&-*&0&0&**&0\\0&0&0&0&-*&0&0&0\\
  0&0&0&0&0&0&0&*
  \ema}\theta^l\dz\theta^k.\end{aligned}$$
  \subsection{Case $I_1\neq 0$, $I_2=0$. Reduction to ten dimensions and new connection}
The expression for $(K^i{}_j)$ above shows that if
  $$I_2=0\Longleftrightarrow S_1=0,$$
  and we have only `horizontal terms' $\theta^k\dz\theta^l$ in $(K^i{}_j)$, then $(\Gamma^i{}_j)$ may be interpreted as a \emph{principal connection} on $Q$ with \emph{curvature} $(K^i{}_j)$. In particular, in such a case $Q$ locally becomes a principal fiber bundle $H_0\to Q\to M$, with $H_0$ a Lie group with the Lie algebra
  $$\mathfrak{h}_0=\{\gla(8,\bbR)\ni {\bf a}~|~ ({\bf a}^i{}_j)=\tiny{\bma
B&C&0&0&0&0&0&0\\0&B&0&0&0&0&0&0\\2C&0&B&0&0&0&0&0\\0&0&0&B&C&0&0&0\\0&0&0&0&B&0&0&0\\0&0&0&-B&0&0&C&0\\0&0&0&0&-B&0&0&0\\
  0&0&0&0&0&0&0&B
  \ema},\,B,C\in\bbR\}.$$
  The fibers of the bundle $H_0\to H_0\times M\to M$ are tangent to the integrable distribution $$Ann=\{X,Y\in \Gamma(\mathrm{T}Q)~|~X\hook\theta^i=Y\hook\theta^i=0,~i=1,2,\dots, 8\}.$$ Thus, if $I_2=0$ the collective quantity $(T^i, K^i{}_j)$ gets interpreted as respective \emph{torsion}, $(T^i)$, and \emph{curvature}, $(K^i{}_j)$, of the principal connection $(\Gamma^i{}_j)$.

  We have the following Theorem.

\begin{theorem}\label{dzie}
  Every Monge structure (\ref{sysc3u}) with $I_1\neq 0$ and $I_2=0$ defines a \emph{unique} local reduction of its 11-dimensional structural bundle $H\to{\mathcal G}^{11}\to M$ to a 10-dimensional principal fiber bundle $H_0\to Q\to M$ with uniquely defined principal $\mathfrak{h}_0$-valued connection
  $$
(\Gamma{}^i{}_j)=\bma
-2\Om_1&-\Theta&0&0&0&0&0&0\\0&-2\Om_1&0&0&0&0&0&0\\-2\Theta&0&-2\Om_1&0&0&0&0&0\\0&0&0&-2\Om_1&-\Theta&0&0&0\\0&0&0&0&-2\Om_1&0&0&0\\0&0&0&2\Om_1&0&0&-\Theta&0\\0&0&0&0&2\Om_1&0&0&0\\
  0&0&0&0&0&0&0&-2\Om_1
  \ema.$$
The principal connection $(\Gamma^i{}_j)$ has torsion
  $$T^i=\der\theta^i+\Gamma^i{}_j\dz\theta^j,$$
  such that
  $$(T^i)=\bma S_4(\theta^1\dz\theta^5+\theta^1\dz\theta^7-\theta^2\dz\theta^4-
  \theta^2\dz\theta^6)-\theta^4\dz\theta^7-\theta^5\dz\theta^6\\
  -2\theta^5\dz\theta^7\\
2S_4(-\theta^1\dz\theta^4-\theta^1\dz\theta^6+\theta^3\dz\theta^5+
  \theta^3\dz\theta^7)+\theta^2\dz\theta^5+\theta^2\dz\theta^7-2\theta^4\dz\theta^6\\
S_4(-\theta^4\dz\theta^7-\theta^5\dz\theta^6)+\epsilon(\theta^1\dz\theta^5+
\theta^1\dz\theta^7)+\theta^4\dz\theta^8+\theta^6\dz\theta^8\\
-2S_4\theta^5\dz\theta^7+\epsilon(\theta^2\dz\theta^5+\theta^2\dz\theta^7)+
\theta^5\dz\theta^8+\theta^7\dz\theta^8\\
S_4(4\theta^4\dz\theta^5+5\theta^4\dz\theta^7-3\theta^5\dz\theta^6+
  4\theta^6\dz\theta^7)-\epsilon(\theta^1\dz\theta^5+\theta^1\dz\theta^7)-\theta^4\dz\theta^8-\theta^6\dz\theta^8\\
  2S_4\theta^5\dz\theta^7-\epsilon(\theta^2\dz\theta^5+\theta^2\dz\theta^7)-\theta^5\dz\theta^8-\theta^7\dz\theta^8\\
  4S_4(\theta^5\dz\theta^8+\theta^7\dz\theta^8)+2\epsilon\theta^5\dz\theta^7
  \ema,$$
  and the curvature
  $$K^i{}_j=\der\Gamma^i{}_j+\Gamma^i{}_k\dz\Gamma^k{}_j,$$
  such that
  $$(K^i{}_j)=-S_5\tiny{\bma
0&1&0&0&0&0&0&0\\0&0&0&0&0&0&0&0\\2&0&0&0&0&0&0&0\\0&0&0&0&1&0&0&0\\0&0&0&0&0&0&0&0\\0&0&0&0&0&0&1&0\\0&0&0&0&0&0&0&0\\
  0&0&0&0&0&0&0&0
  \ema}(\theta^4\dz\theta^5+\theta^4\dz\theta^7-\theta^5\dz\theta^6+\theta^6\dz\theta^7).$$

  In these equations the invariant 
$$\epsilon=\mathrm{sign}(I_1)=\pm1,$$
and the functions $S_4$ and $S_5$ are
  \be\begin{aligned}
  &S_4=\frac{9I_1v_2-15v_2^3+8u^2z_4}{2\sqrt{\epsilon I_1^3}},\\
  &S_5=\frac{64I_1u^3z_9-22I_1^3-9I_1^2v_2^2-180I_1v_2^4+675v_2^6-16I_1u^2v_2z_4-720u^2v_2^3z_4+192u^4z_4^2}{4\epsilon I_1^3},\end{aligned}
\label{ss6}\ee
where the quantities $u$, $v_2$, $z_4$, $I_1$ and $z_9$ defined in (\ref{sysc3u})-(\ref{sysc3u1}).

  The torsion $(T^i)$ and the curvature $(K^i{}_j)$ of the principal connection $(\Gamma^i{}_j)$ define the system of all fundamental invariants of the structures with $I_1\neq 0$ and $I_2=0$. In particular, we have
  \be\der S_4=(\epsilon+3S_4^2+S_5)(\theta^5+\theta^7).\label{ss5}\ee
Thus, all the invariants are obtained in terms of the differentiation of $S_4$. 
\end{theorem}
\begin{proof}
  If $I_2=0$ we have that $v_1=0$ in equations (\ref{sysc3u1}). Then it is easy to see that in such a situation, these equations, when written in terms of the basic invarinat $I_1$ and $u$, read:
\be\der u=v_2\om^7,\quad\der v_2=\frac{5v_2^2-I_1}{4u}\om^7,\quad\der I_1=\frac{15v_2^3-3I_1v_2-8u^2z_4}{2u}\om^7,\quad\der z_4=z_9\om^7,\label{sysc3t1t2}\ee
with $(\om^1,\om^2,\dots,\om^8)$ and $u\neq 0$ as in (\ref{sysc3u}). In particular we also have $$z_1=z_2=z_3=z_5=z_6=z_7=z_8=0,$$ and
$$I_3=R=0.$$
Thus the vanishing of the invariant $I_2$ and nonvanishing of $I_1$, implies the vanishing of the invariants $I_3$ and $R$: 
$$(I_1\neq 0\,\,\&\,\,I_2=0)\Longrightarrow (I_3=0\,\,\&\,\, R=0).$$
This in turn means that:
\be S_1=S_2=S_3=0,\quad S_4=\frac{9I_1v_2-15v_2^3+8u^2z_4}{2\sqrt{\epsilon I_1^3}}.\label{ss4}\ee
Therefore the most general Monge system with $I_1\neq 0$ and $I_2=0$ is given by the coframe $(\om^1,\om^2,\dots,\om^8)$ on $M$ satsifying (\ref{sysc3u}) and (\ref{sysc3t1t2}). It is now a matter of checking that for such a system (\ref{sysc3u}) the EDS (\ref{ss2})-(\ref{ss3}), when reduced by the conditions (\ref{ss4}), becomes
\be
\begin{aligned}
\der\theta^1=&2\Om_1\dz\theta^1+\Theta\dz\theta^2+S_4\theta^1\dz\theta^5+S_4\theta^1\dz\theta^7-S_4\theta^2\dz\theta^4-S_4\theta^2\dz\theta^6-\\&\theta^4\dz\theta^7-\theta^5\dz\theta^6,\\
\der\theta^2=&2\Om_1\dz\theta^2-2\theta^5\dz\theta^7,\\
\der\theta^3=&2\Om_1\dz\theta^3+2\Theta\dz\theta^1-2S_4\theta^1\dz\theta^4-2S_4\theta^1\dz\theta^6+\theta^2\dz\theta^5+\theta^2\dz\theta^7+2S_4\theta^3\dz\theta^5+\\&2S_4\theta^3\dz\theta^7-2\theta^4\dz\theta^6,\\
\der\theta^4=&2\Om_1\dz\theta^4+\Theta\dz\theta^5+\epsilon\theta^1\dz\theta^5+\epsilon \theta^1\dz\theta^7-S_4\theta^4\dz\theta^7+\theta^4\dz\theta^8-\\&S_4\theta^5\dz\theta^6+\theta^6\dz\theta^8,\\
\der\theta^5=&2\Om_1\dz\theta^5+\epsilon\theta^2\dz\theta^5+\epsilon \theta^2\dz\theta^7-2S_4\theta^5\dz\theta^7+\theta^5\dz\theta^8+\theta^7\dz\theta^8,\\
\der\theta^6=&-2\Om_1\dz\theta^4+\Theta\dz\theta^7-\epsilon \theta^1\dz\theta^5-\epsilon \theta^1\dz\theta^7+4S_4\theta^4\dz\theta^5+5S_4\theta^4\dz\theta^7-\theta^4\dz\theta^8-\\&3S_4\theta^5\dz\theta^6
+4S_4\theta^6\dz\theta^7-\theta^6\dz\theta^8,\\
\der\theta^7=&-2\Om_1\dz\theta^5-\epsilon\theta^2\dz\theta^5-\epsilon\theta^2\dz\theta^7+2S_4\theta^5\dz\theta^7-\theta^5\dz\theta^8-\theta^7\dz\theta^8,\\
\der\theta^8=&2\Om_1\dz\theta^8+2\epsilon \theta^5\dz\theta^7+4S_4\theta^5\dz\theta^8+4S_4\theta^7\dz\theta^8,\\
\der\Om_1=&0,\\
\der\Theta=&S_5(\theta^4\dz\theta^5+\theta^4\dz\theta^7-\theta^5\dz\theta^6+\theta^6\dz\theta^7).\end{aligned}
\label{ssc}\ee
A quick check shows that this is egquivalent to $T^i=\der\theta^i+\Gamma^i{}_j\dz\theta^j$ and $K^i{}_j=\der\Gamma^i{}_j+\Gamma^i{}_k\dz\Gamma^k{}_j$, with the connection $(\Gamma^i{}_j)$, torsion $(T^i)$ and the curvature $(K^i{}_j)$ as claimed in the theorem.
The equation (\ref{ss5}) is a consequences of Bianchi identities $\der^2=0$ applied to the system (\ref{ssc}). It can also be checked directly using definitions (\ref{ss6}), (\ref{in1}), and the relations (\ref{sysc3u1}).
\end{proof}
It is worthwhile to note that system (\ref{ssc}), (\ref{ss5}) can not be reduced to lower dimensions. This is because the group $H_0$ parameters $b$ and $c$ that \emph{a'priori} could appear in the definitions of the invariants $\epsilon$, $S_4$ and $S_5$ are not present there.

We close this section with the fololowing proposition.
\begin{proposition}\label{I1I2}
All possible Monge systems with $I_1\neq 0$ and $I_2=0$ are given in Proposition \ref{exu1}. They correspond to Monge ODEs:
\be\dot{z}_{11}=\dot{x}_1^2,\quad \dot{z}_{12}=\dot{x}_1\dot{x}_2,\quad z_{22}=\dot{x}_2^2+h(\dot{x}_1).\label{dj111}\ee
The invariant $S_4$ for them is:
\be
S_4=\frac{15h^{(4)}{}^3-18h^{(3)}h^{(4)}h^{(5)}+4h^{(3)}{}^2h^{(6)}}{\sqrt{\epsilon(5h^{(4)}{}^2-4h^{(3)}h^{(5)})^3}},\label{dd}\ee
and we have to have
\be
I_1=\tfrac14(5h^{(4)}{}^2-4h^{(3)}h^{(5)})\neq 0.\label{dd1}\ee
\end{proposition}
\begin{proof}
  Proof follows from Theorem \ref{exa}. Looking at the explicit expression for $I_2$ given in Theorem \ref{exa} for the most general Monge system with $T$ having quintic root, we see that $I_2=0$ correspond to functions $h$ such that $h_{2227}=0$. Thus the most general $h$ having $I_2=0$ satisfies $h_{22}=f(y_2)+2g'(y_7)$ with some differentiable functions $f=f(y_2)$ and $g=g(y_7)$. But $h$ appears in the coframe (\ref{systemu}) defining these Monge systems in the 1-form $\om^3=\der y_6+2y_3\der y_5-\tfrac12h_{22}\der y_7$ only. Also, the coordinate $y_6$ does not appear in any other $\om^i$ than $\om^3$. Since we have:
  $$\om^3=\der y_6+2y_3\der y_5-\tfrac12f(y_2)\der y_7-\der g(y_7)=\der \Big(y_6-g(y_7)\Big)+2y_3\der y_5-\tfrac12f(y_2)\der y_7,$$
  then changing the coordinate $y_6$ into $y_6\to y_6-g(y_7)$, we see that the function $g(y_7)$ totally disapears from the system (\ref{systemu}). Thus all the Monge structures with with $h=f(y_2)+g(y_7)$ are locally equivalent to the Monge structures with the corresponding $h$ function given by $h=f(y_2)$. Then the proof of Proposition \ref{exu1} shows that such structures are equivalent to the structures associated with the Monge ODEs (\ref{dj111}). 
\end{proof}
\subsection{Ten dimensional symmetry group}
As we have noticed in the proof of Theorem \ref{dzie} the Monge systems with $I_1\neq 0$ and $I_2=0$ are genuinly defined on 10-dimensional manifold and are characterized by the structural function $S_4$. If among these systems there are such which have 10-dimensional transitive group of symmetries, then this structural function must be constant: 
$$\der S_4=0.$$
We have the following theorem.
\begin{theorem}\label{10} All Monge structures with the main invariant having the root with quintic multiplicity that have 10-dimensional transitive group of local symmetries correspond to the case $I_1\neq 0$, $I_2=0$ and $S_4=\mathrm{const}$. A class of Monge ODEs corresponding to these structures is given by
  $$\dot{z}_{11}=\dot{x}_1^2,\quad \dot{z}_{12}=\dot{x}_1\dot{x}_2,\quad z_{22}=\dot{x}_2^2+\dot{x}_1^s,\quad\quad s=\mathrm{const}\in\bbR.$$
  The constant invariant $S_4$ for these structures is
  $$S_4=\frac{s-1}{\sqrt{\epsilon(s+1)(s-3)}},$$
  where
  $$\epsilon=\mathrm{sign}\big((s+1)(s-3)\big).$$
  These structures have $I_1\neq 0$ iff $$s\neq -1,0,1,2,3.$$
\end{theorem}
\begin{proof}
  If the Monge structure has $T\neq 0$ and $T(b_6)=0$ has quintic root, then the corresponding EDS reduces to 11 dimensions. So the symmetry of such a structure has maximal symmetry dimension to 11. This happens precisely if $I_1=0$. If $I_1\neq 0$ then the EDS reduces to 10-dimensional EDS (\ref{ss2})-(\ref{ss3}). The maximum symmetry of such systems, which may have maximum dimension 10, can occur only if further reduction is not possible. This happens if and only if $I_2=0$, as if $I_2\neq 0$ we may normalize $-S_1=-\frac{2I_2u^2}{bI_1}$ to $-1$, with $b=\frac{2I_2u^2}{I_1}$, and reduce to \emph{nine} dimensions, where the dimension of the symmetry group of the Monge system can not be greater than \emph{nine}. Thus the Monge structures with $T(b_6)=0$ having quintic root with the symmetry group of dimension ten, may only be in the case $I_1\neq =0$ and $I_2=0$. For the transitivity of symmetries we need $S_4=\mathrm{const}$, hence $\der S_4=0$. Looking at the differential of $S_4$ given by (\ref{ss5}), we see that this imples that
  $$S_5=-(\epsilon+3S_4^2).$$
  This brings the Monge EDS (\ref{ssc}) into:
\be
\begin{aligned}
\der\theta^1=&2\Om_1\dz\theta^1+\Theta\dz\theta^2+S_4\theta^1\dz\theta^5+S_4\theta^1\dz\theta^7-S_4\theta^2\dz\theta^4-S_4\theta^2\dz\theta^6-\\&\theta^4\dz\theta^7-\theta^5\dz\theta^6,\\
\der\theta^2=&2\Om_1\dz\theta^2-2\theta^5\dz\theta^7,\\
\der\theta^3=&2\Om_1\dz\theta^3+2\Theta\dz\theta^1-2S_4\theta^1\dz\theta^4-2S_4\theta^1\dz\theta^6+\theta^2\dz\theta^5+\theta^2\dz\theta^7+2S_4\theta^3\dz\theta^5+\\&2S_4\theta^3\dz\theta^7-2\theta^4\dz\theta^6,\\
\der\theta^4=&2\Om_1\dz\theta^4+\Theta\dz\theta^5+\epsilon\theta^1\dz\theta^5+\epsilon \theta^1\dz\theta^7-S_4\theta^4\dz\theta^7+\theta^4\dz\theta^8-\\&S_4\theta^5\dz\theta^6+\theta^6\dz\theta^8,\\
\der\theta^5=&2\Om_1\dz\theta^5+\epsilon\theta^2\dz\theta^5+\epsilon \theta^2\dz\theta^7-2S_4\theta^5\dz\theta^7+\theta^5\dz\theta^8+\theta^7\dz\theta^8,\\
\der\theta^6=&-2\Om_1\dz\theta^4+\Theta\dz\theta^7-\epsilon \theta^1\dz\theta^5-\epsilon \theta^1\dz\theta^7+4S_4\theta^4\dz\theta^5+5S_4\theta^4\dz\theta^7-\theta^4\dz\theta^8-\\&3S_4\theta^5\dz\theta^6
+4S_4\theta^6\dz\theta^7-\theta^6\dz\theta^8,\\
\der\theta^7=&-2\Om_1\dz\theta^5-\epsilon\theta^2\dz\theta^5-\epsilon\theta^2\dz\theta^7+2S_4\theta^5\dz\theta^7-\theta^5\dz\theta^8-\theta^7\dz\theta^8,\\
\der\theta^8=&2\Om_1\dz\theta^8+2\epsilon \theta^5\dz\theta^7+4S_4\theta^5\dz\theta^8+4S_4\theta^7\dz\theta^8,\\
\der\Om_1=&0,\\
\der\Theta=&-(\epsilon+3S_4^2)(\theta^4\dz\theta^5+\theta^4\dz\theta^7-\theta^5\dz\theta^6+\theta^6\dz\theta^7),\\
\der S_4=&0,\\
\epsilon=\pm&1.\end{aligned}
\label{sca}\ee
One can check that this system is \emph{differentially closed}, i.e. that applying the exterioror differential $\mathrm{d}$ on both sides of the above equations does \emph{not} bring any compatibility conditions.

Thus every Monge system satisfying these equations has 10-dimensional group of symmetries. The group manifold is just $Q$ with forms $(\theta^1,\theta^2,\dots,\theta^8,\Om_1,\Theta)$ being a basis of the Maurer Cartan forms. The structure constants of the Lie algebra of the symmetry group in this basis are easilly read of from (\ref{sca}).

To construct Monge ODEs corresponding to Monge systems (\ref{sca}) with a pair of numbers $(\epsilon,S_4)\in \bbZ_2\times\bbR$ being invariants we do as follows.

Take a Monge structure corresponding to the Monge ODEs given in Proposition \ref{I1I2} with function
$$h(\dot{x}_1)=\dot{x}_1^s,\quad\quad s=\mathrm{const}\in\bbR.$$
Then, according to this Proposition (formula (\ref{dd1})), this Monge structure has $I_1\neq 0$ iff
$s^2(s-1)^2(s-2)^2(s-3)(s+1)\neq 0$ and it has $I_2=0$. Moreover, aplying formula (\ref{dd}) to the function $h=\dot{x}_1^s$ we see that this Monge structure has
$$S_4=\frac{s-1}{\sqrt{\epsilon(s+1)(s-3)}},\quad \epsilon=\mathrm{sign}\big((s+1)(s-3)\big).$$
Thus $S_4$ is a \emph{constant} and the corresponding EDS describes a Monge system with 10-symmetries as in (\ref{sca}).

If $\epsilon=-1$, then $s\in]-1,3[$. A quick look at a function $S_4(s)=\frac{s-1}{\sqrt{-(s+1)(s-3)}}$ shows that the image of the interval $]-1,3[$ is $\bbR$, $S_4(]-1,3[)=\bbR$. Thus all pairs $(-1,S_4)\in Z_2\times\bbR$ are invariants of a Monge system given in Proposistion \ref{I1I2} with $h=\dot{x}_1^s$, $s\in]-1,3[$. Likewise one can convince himself that the pairs $(1,S_4)$ with $S_4\in]-\infty,-1[\cup]1,+\infty[$ can be realized as invariants of Monge systems corresponding to the Monge ODEs given in Proposition \ref{I1I2} having $h=\dot{x}_1^s$, $s\in]-\infty,-1[\cup]3,+\infty[$.

We failed to find explicit realization, in terms of Monge ODEs, of pairs $(1,S_4)$ with $S_4\in[-1,1]$ as invariants. It would be nice to have them. 
\end{proof}  
\subsection{Covariantly constant torsion and curvature of $\mathfrak{h}_0$-connection} Given a Monge structure with $I_1\neq 0$ and $I_2=0$ we have an $\mathfrak{h}_0$-valued connection $(\Gamma^i{}_j)$ as given in Theorem \ref{dzie}. We also have its torsion $(T^i)$ and curvature $(K^i{}_j)$. These define respective torsion coefficients $T^i{}_{jk}$ and curvature coefficients $K^i{}_{jkl}$ given by:
$$T^i=\tfrac12 T^i{}_{jk}\theta^j\dz\theta^k,\quad\quad K^i{}_j=K^i{}_{jkl}\theta^k\dz\theta^l.$$
We say that the Monge structure with $I_1\neq 0$, $I_2=0$ has \emph{covariantly constant torsion} iff
$$DT^i{}_{jk}=\der T^i{}_{jk}+\Gamma^i{}_lT^l{}_{jk}-\Gamma^l{}_jT^i{}_{lk}-\Gamma^l{}_kT^i{}_{jl}=0.$$
Likewise, we say that it has \emph{covariantly constant curvature} iff 
$$DK^i{}_{jkl}=\der K^i{}_{jkl}+\Gamma^i{}_mK^m{}_{jkl}-\Gamma^m{}_jK^i{}_{mkl}-\Gamma^m{}_kK^i{}_{jml}-\Gamma^m{}_lK^i{}_{jkm}=0.$$
We have the following proposition.
\begin{proposition}
  Every Monge structure with $I_1\neq 0$ and $I_2=0$ which has covariantly constant torsion has $\der S_4=0$, and thus has 10-dimensional group of local symmetries.

  Every Monge structure with $I_1\neq 0$ and $I_2=0$ which has covariantly constant curvature has $\der S_5=0$. The corresponding EDS for Monge structure with covariantly constant curvature is (\ref{ssc}) with $\der S_5=0$ and is differentially closed.

  Structures with covariantly constant torsion have covariantly constant curvature, but not the other way arround.
\end{proposition}
  \begin{proof}
    By a direct calculation one shows that the condition $DT^i{}_{jk}=0$ applied to the torsion $(T^i)$ from Theorem \ref{dzie} gives is equivalent to \be\epsilon+3S_4^2+S_5=0\label{bui}\ee This, when compared with (\ref{ss5}), gives $\der S_4=0$. Because of (\ref{bui}) we have $S_5=-(\epsilon+S_4^2)$, so in such a case  we end up with the EDS (\ref{sca}) having 10 symmetries and decsribed by Theorem \ref{10}.
    
    Likewise, Bianchi identities for the EDS (\ref{ssc}), applied to $\der S_4$ from (\ref{ss5}) show that $\der S_5\dz(\theta^5+\theta^7)=0$. This means that $\der S_5=Z(\theta^5+\theta^7)$ with $Z$ a function on $Q$. Now the left hand side of the condition $DK^i{}_{jkl}=0$ can be calculated directly for the curvature $(K^i{}_j)$ from Theorem \ref{dzie}. It shows that $DK^i{}_{jkl}=0$ if and only $Z=0$. Thus we have $\der S_5=0$. One can see by a direct calculation that condition $\der S_5=0$ differentially closes the EDS (\ref{ssc}) without stronger assumption that $\der S_4=0$. 
    \end{proof}
  \subsection{ Reduction to nine dimensions: case $I_1\neq 0$, $I_2\neq 0$}
  We now return to the EDS (\ref{ss2})-(\ref{ss3}) corresponding to generic case of Monge structures with $I_1\neq 0$, namely those that have $I_2\neq 0$. If $I_2\neq 0$ formula (\ref{ss17}) enables for the next normalization
  \be S_1=-1\quad\Longleftrightarrow\quad b=-\frac{2I_2u^2}{I_1}.\label{norb}\ee
  This reduces the system (\ref{ss2})-(\ref{ss3}) to \emph{nine} dimensions. In particular the form $\Om_1$ becomes dependent on ($\theta^1,\theta^2,\dots,\theta^8,\Theta)$ and reads:
  $$\Om_1=-\tfrac43\epsilon (1+\frac{R}{I_2^2})\theta^2+\tfrac14 S_2(\theta^5+\theta^7).$$
  A convenient coordinate system on $\mathcal U$ is given by coordinates $(y_1,y_2,\dots,y_8)$ as given in Theorem \ref{exa} on the base $M$, and a coordinate $c$ - the remaining of the three group $H$ parameters $(a,b,c)$ after imposition of the normalizations $a=\frac{(\epsilon I_1)^{1/4}}{2u}$, $b=-\frac{2I_2u^2}{I_1}$.

We have the following proposition.
\begin{proposition}\label{411}
    Every Monge structure with $I_1\neq 0$ and $I_2\neq 0$ uniquley defines a 9-dimensional subset $\mathcal U$ of $Q$, which is locally fibred ${\mathcal U}\to M$ over the 8dimensional manifold $M$ on which the Monge structure is defined. The EDS corresponding to these Monge structures is given in terms of a unique coframe $(\theta^1,\theta^2,\dots,\theta^8,\Theta)$ and reads: 
 \be
\begin{aligned}
\der\theta^1=&\Theta\dz\theta^2+S_6\theta^1\dz\theta^2+\tfrac12(S_2+2S_4)\theta^1\dz\theta^5+\\&\tfrac12(2S_4-S_2)\theta^1\dz\theta^7+\theta^1\dz\theta^8-(S_2+S_4)\theta^2\dz\theta^4-S_4\theta^2\dz\theta^6-\\&\theta^4\dz\theta^7-\theta^5\dz\theta^6,\\
\der\theta^2=&-\tfrac12S_2\theta^2\dz\theta^5-\tfrac12 S_2\theta^2\dz\theta^7-2\theta^5\dz\theta^7,\\
\der\theta^3=&2\Theta\dz\theta^1-2(S_2+S_4)\theta^1\dz\theta^4-2S_4\theta^1\dz\theta^6-\\&S_6\theta^2\dz\theta^3+\theta^2\dz\theta^5+\theta^2\dz\theta^7+\tfrac12(3S_2+4S_4)\theta^3\dz\theta^5+\\&\tfrac12(4S_4-S_2)\theta^3\dz\theta^7+2\theta^3\dz\theta^8-2\theta^4\dz\theta^6,\\
\der\theta^4=&\Theta\dz\theta^5-\epsilon S_2\theta^1\dz\theta^2-(\tfrac13\epsilon+\tfrac12S_3)\theta^1\dz\theta^5+\epsilon \theta^1\dz\theta^7-\\&\tfrac12(3S_3+2S_6)\theta^2\dz\theta^4-\tfrac12S_2\theta^4\dz\theta^5-\tfrac12(S_2+2S_4)\theta^4\dz\theta^7-\\&S_4\theta^5\dz\theta^6+\theta^6\dz\theta^8,\\
\der\theta^5=&-(\tfrac13\epsilon+2S_3+S_6) \theta^2\dz\theta^5+\epsilon \theta^2\dz\theta^7-\tfrac12(S_2+4S_4)\theta^5\dz\theta^7-\\&\theta^5\dz\theta^8+\theta^7\dz\theta^8,\\
\der\theta^6=&\Theta\dz\theta^7+\epsilon S_2\theta^1\dz\theta^2+3\epsilon \theta^1\dz\theta^5+(\tfrac53\epsilon-\tfrac12S_3) \theta^1\dz\theta^7+\\&(3S_3+S_6)\theta^2\dz\theta^4+\tfrac32S_3\theta^2\dz\theta^6+\tfrac12(9S_2+8S_4)\theta^4\dz\theta^5+\\&\tfrac12(3S_2+10S_4)\theta^4\dz\theta^7+3\theta^4\dz\theta^8-3(S_2+S_4)\theta^5\dz\theta^6
+\\&4S_4\theta^6\dz\theta^7+2\theta^6\dz\theta^8,\\
\der\theta^7=&(3\epsilon+3S_3+S_6)\theta^2\dz\theta^5+(\tfrac53\epsilon+S_3) \theta^2\dz\theta^7+\\&\tfrac12(4S_4-3S_2)\theta^5\dz\theta^7+3\theta^5\dz\theta^8+\theta^7\dz\theta^8,\\
\der\theta^8=&2\epsilon S_2\theta^2\dz\theta^5-(3S_3+S_6)\theta^2\dz\theta^8+2\epsilon \theta^5\dz\theta^7+\\&\tfrac12(9S_2+8S_4)\theta^5\dz\theta^8+\tfrac12(S_2+8S_4)\theta^7\dz\theta^8,\\
\der\Theta=&\Theta\dz\theta^8+S_7\theta^1\dz\theta^2+(2\epsilon S_2-S_{10})\theta^1\dz\theta^5-S_8\theta^1\dz\theta^7-\\&(3S_3+S_6)\theta^1\dz\theta^8+S_{10}\theta^2\dz\theta^4+S_8\theta^2\dz\theta^6+\\&(6\epsilon +3S_2^2-9S_3+6S_2 S_4+S_9-4S_6)\theta^4\dz\theta^5+
\end{aligned}\label{ss18}\ee
$$\begin{aligned}
&(\tfrac43\epsilon -\tfrac52S_3+3S_2S_4+S_9)\theta^4\dz\theta^7+\tfrac12(9S_2+8S_4)\theta^4\dz\theta^8+\\&(\tfrac23\epsilon+\tfrac52S_3-3S_2S_4-S_9)\theta^5\dz\theta^6+S_9\theta^6\dz\theta^7+\tfrac12(S_2+8S_4)\theta^6\dz\theta^8.
\end{aligned}$$
Here the functions $S_2,S_3,S_4$ are functions from (\ref{fsa}) with $b$ as in (\ref{norb}), and $S_6,S_7,\dots,S_{10}$ are some other functions on $\mathcal U$. All the functions $S_i$, $i=2,3,4,6,7,8,9,10$, do not depend on the fiber coordinate $c$, $\frac{\partial S_i}{\partial c}=0$. In particular,
$$S_6=\frac{8R+8I_2^2-3\epsilon S_3 I_2^2}{3\epsilon I_2^2},$$
and we also have
$$\begin{aligned}
  \der S_2=&2S_2\theta^8+\tfrac12(3S_2S_3+2S_8+2S_2S_6-2S_{10})\theta^2+\\&(4\epsilon+\tfrac32 S_2^2-6S_3+2S_2S_4-2S_6)\theta^5+(2S_2S_4-\tfrac12 S_2^2-2S_3-\tfrac43\epsilon)\theta^7,\\
  \der S_3=&2(3S_3+S_6)\theta^8+(\tfrac{128}{9}+2\epsilon S_2^2-\tfrac{48}{9}\epsilon S_3+\tfrac12 S_3^2-2S_7-\tfrac{48}{9}\epsilon S_6+S_3 S_6)\theta^2+\\&(2S_{10}-\tfrac23\epsilon S_2+\tfrac12 S_2S_3-\tfrac{16}{3}\epsilon S_4+S_3S_4)\theta^5+\\&(\tfrac23\epsilon S_2-\tfrac12S_2S_3-\tfrac{16}{3}\epsilon S_4+S_3S_4+2S_8)\theta^7,\\
  \der S_4=&\tfrac32(S_2+4S_4)\theta^8-(\epsilon S_2+\tfrac32 S_3S_4+S_8)\theta^2+\\&(S_9+3S_4^2+6S_2S_4-3S_3-\epsilon)\theta^5+(\epsilon+3S_4^2+S_9)\theta^7.
\end{aligned}$$
  \end{proposition}
\begin{proof}By a straightforward algebra applied to the EDS (\ref{ss2})-(\ref{ss3}). \end{proof}

Now let us consider two vector subspaces $\mathfrak{t}_\epsilon$, $\epsilon=\pm1$, in $\spa(3,\bbR)$, which are defined by the following linear relations
\be\begin{aligned}
  b^1{}_1=&\tfrac{32}{9}a^1{}_1,\quad b^1{}_2=0,\quad b^2{}_2=0,\\
  c^1{}_1=&-\tfrac23\epsilon a^1{}_1,\quad c^2{}_1=0,\quad c^2{}_2=-c-\epsilon a^1{}_1,\\
  q^1=&\tfrac13\epsilon(3p^1+u^1),\quad q^2=0,\quad q=3c,\\
  p=&c,\\
  v^1=&\epsilon(3 u^1-\frac53 p^1),\quad v^2=0,
\end{aligned}\label{paar}\ee
between the parameters of algebra $\spa(3,\bbR)$ as given in (\ref{sp3r}). It follows that each of the vector subspaces $\mathfrak{t}_\epsilon\subset\spa(3,\bbR)$ is a Lie subalgebra of $\spa(3,\bbR)$. 

It is now convenient to introduce the following 1-form $\varpi$ on $\mathcal U$ with values in $\mathfrak{t}_\epsilon$:
\be\varpi=\bma
-\tfrac23\epsilon\theta^2&\Theta+\tfrac13\epsilon\theta^1&\tfrac13\epsilon(\theta^5-3\theta^7)&\tfrac{32}{9}\theta^2&0&\tfrac13\epsilon(9\theta^5+5\theta^7\\
0&-\epsilon\theta^2-\theta^8&0&0&0&0\\
\theta^7&\theta^6&\theta^8&\tfrac13\epsilon(9\theta^5+5\theta^7)&0&3\theta^8\\
\theta^2&\theta^1&\theta^5&\tfrac23\epsilon\theta^2&0&-\theta^7\\
\theta^1&\theta^3-\tfrac14\epsilon(\theta^5+\theta^7)&\theta^4&-\Theta-\tfrac13\epsilon\theta^1&\epsilon \theta^2+\theta^8&-\theta^6\\
\theta^5&\theta^4&\theta^8&\tfrac13\epsilon(3\theta^7-\theta^5)&0&-\theta^8
\ema.\label{cc9}\ee
To see that this has values in $\mathfrak{t}_\epsilon$ it is enough to take 
$$\begin{aligned}
  a^1{}_2=&\theta^1,\quad a^1{}_1=\theta^2,\quad a^2{}_2=\theta^3-\tfrac14\epsilon(\theta^5+\theta^7),\\
  c^1{}_2=&\Theta+\tfrac13\epsilon\theta^1,\quad c=\theta^8,\\
  p^1=&-\theta^7,\quad p^2=-\theta^6,\\
  u^1=&\theta^5,\quad u^2=\theta^4,
\end{aligned}$$
in (\ref{sp3r}) and to take the rest of the parameters in (\ref{sp3r}) as in (\ref{paar}). 

It then follows that $\varpi$ is a Cartan $\mathfrak{t}_\epsilon$-valued connection on a Cartan bundle ${\mathcal U}\to M$, whose curvature
\be\Omega=\der\varpi+\varpi\dz\varpi,\label{cc10}\ee
is
\be\Omega=\bma
-\tfrac23\epsilon A^1{}_1&C^1{}_2&\tfrac13\epsilon(U^1+3P^1)&\tfrac{32}{9}A^1{}_1&0&\tfrac13\epsilon(9U^1-5P^1)\\
0&-\epsilon A^1{}_1-C&0&0&0&0\\
-P^1&-P^2&C&\tfrac13\epsilon(9U^1-5P^1)&0&3C\\
A^1{}_1&A^1{}_2&U^1&\tfrac23\epsilon A^1{}_1&0&P^1\\
A^1{}_2&A^2{}_2&U^2&-C^1{}_2&\epsilon A^1{}_1+C&P^2\\
U^1&U^2&C&-\tfrac13\epsilon(3P^1+U^1)&0&-C
\ema,\label{cc11}\ee
where:
\be\begin{aligned}
  A^1{}_1=&-\tfrac12 S_2\theta^2\dz(\theta^5+\theta^7),\\
  A^1{}_2=&(S_6-2\epsilon)\theta^1\dz\theta^2+\tfrac12(S_2+2S_4)\theta^1\dz\theta^5+\tfrac12(2S_4-S_2)\theta^1\dz\theta^7-\\&(S_2+S_4)\theta^2\dz\theta^4-S_4\theta^2\dz\theta^6,\\
  A^2{}_2=&-2(S_2+S_4)\theta^1\dz\theta^4-2S_4\theta^1\dz\theta^6+(2\epsilon-S_6)\theta^2\dz\theta^3-\\&(\tfrac16+\tfrac14\epsilon S_3)\theta^2\dz(\theta^5+\theta^7)+\tfrac12(3S_2+4S_4)\theta^3\dz\theta^5+\tfrac12(4S_4-S_2)\theta^3\dz\theta^7+\tfrac12\epsilon S_2\theta^5\dz\theta^7,\\
  U^1=&(\tfrac23\epsilon-2S_3-S_6)\theta^2\dz\theta^5-\tfrac12(S_2+4S_4)\theta^5\dz\theta^7,\\
  U^2=&-\epsilon S_2\theta^1\dz\theta^2-(\tfrac13\epsilon+\tfrac12S_3)\theta^1\dz\theta^5+\tfrac12(2\epsilon-3S_3-2S_6)\theta^2\dz\theta^4-\\&\tfrac12 S_2\theta^4\dz\theta^5-\tfrac12(S_2+2S_4)\theta^4\dz\theta^7-S_4\theta^5\dz\theta^6,\\
  C^1{}_2=&\tfrac13(3S_7+\epsilon S_6-11)\theta^1\dz\theta^2+(\tfrac13\epsilon S_4+\tfrac{13}{6}\epsilon S_2-S_{10})\theta^1\dz\theta^5+(\tfrac13\epsilon S_4-\tfrac16\epsilon S_2-S_8)\theta^1\dz\theta^7-\\&(3S_3+S_6)\theta^1\dz\theta^8+\tfrac13(3S_{10}-\epsilon S_2-\epsilon S_4)\theta^2\dz\theta^4+\tfrac13(3S_8-\epsilon S_4)\theta^2\dz\theta^6+\\&(3\epsilon+3S_2^2-9S_3+6 S_2S_4+S_9-4S_6)\theta^4\dz\theta^5+(S_9+3S_2S_4-\tfrac52S_3-\tfrac23\epsilon)\theta^4\dz\theta^7+\\&\tfrac12(9S_2+8 S_4)\theta^4\dz\theta^8-(S_9+3S_2S_4-\tfrac52S_3-\tfrac23\epsilon)\theta^5\dz\theta^6+(\epsilon+ S_9)\theta^6\dz\theta^7+\\&\tfrac12(S_2+8 S_4)\theta^6\dz\theta^8,\\
  C=&2\epsilon S_2\theta^2\dz\theta^5-(3S_3+S_6)\theta^2\dz\theta^8+\tfrac12(9S_2+8S_4)\theta^5\dz\theta^8+\tfrac12(S_2+8S_4)\theta^7\dz\theta^8,\\
  P^1=&-(3S_3+S_6)\theta^2\dz\theta^5-(\tfrac23\epsilon+S_3)\theta^2\dz\theta^7+\tfrac12(3S_2-4S_4)\theta^5\dz\theta^7,\\
   P^2=&-\epsilon S_2\theta^1\dz\theta^2+(\tfrac13\epsilon+\tfrac12S_3)\theta^1\dz\theta^7-(3S_3+S_6)\theta^2\dz\theta^4-\tfrac12(2\epsilon+3 S_3)\theta^2\dz\theta^6-\\&\tfrac12(9S_2+8S_4)\theta^4\dz\theta^5-\tfrac12(3S_2+10S_4)\theta^4\dz\theta^7+3(S_2+S_4)\theta^5\dz\theta^6-4S_4\theta^6\dz\theta^7.\\
  \end{aligned}\label{cc12}\ee
We summerize the above considerations in the following theorem.
\begin{theorem}
  Every Monge structure with $I_1\neq 0$ and $I_2\neq 0$ locally uniquley defines a 9-dimensional Cartan bundle
  $$\bbR\to {\mathcal U}\to M$$ with Cartan connection $\varpi$ as in (\ref{cc9}) defined in terms of a rigid coframe $(\theta^1,\theta^2,\dots,$ $\theta^8,\Theta)$ satisfying the EDS (\ref{ss18}) from Proposition \ref{411}. The curvature of this connection given explicitly in (\ref{cc10})-(\ref{cc11}) provides all basic local invariants of Monge structures with $I_1\neq 0$ and $I_2\neq 0$.

  The EDS lives on 9-dimensional manifold $\mathcal U$ and no Cartan reduction of this system to dimension 8 is possible.
  \end{theorem}
\begin{proof} Since the theorem is a summary of the considerations preceeding it no further justification of its validity is needed, except the statement of impossibility of further Cartan reduction to dimension 8. This is true because \emph{none} of the curvatures functions of the connection $\varpi$ appearing in the curvature $\Omega$ depends on the fiber coordinate $c$. 
  \end{proof}
\subsection{The flat model here and nine dimensional symmetry group}
 It follows that the only possibility for Monge structures that have fivefold multiple root of their main invariant $T(b_6)$ to have precisely 9-dimensional transtive symmetry group is when $I_1\neq 0$, $I_2\neq 0$ and when all the curvature functions $S_2,S_3,S_4,S_6...,S_{10}$ in the EDS (\ref{ss18}) are constants. A quick isnpection of the integrability conditions shows that even weaker conditions, namely 
 $$\der S_2=0,\quad \mathrm{and}\quad \der S_4=0,$$
 imply that
 $$\der S_3=\der S_5=\der S_6=\der S_7=\der S_8=\der S_9=\der S_{10}=0.$$
 Actually we have the following theorem:
 \begin{theorem}\label{sy9}
   Every Monge strucure satisfying an ODEs (\ref{ss18}) with $$\der S_2=\der S_4=0$$ has
   $$S_2=S_4=S_8=S_{10}=0,\,\,\mathrm{and}\,\, S_3=-\tfrac23\epsilon,\,\,\mathrm{and}\,\,  S_6=2\epsilon,\,\,\mathrm{and}\,\, S_7=3,\,\,\mathrm{and}\,\, S_9=-\epsilon.$$
   Thus, modulo local equivalence, there are only two such structures, corresponding to $\epsilon=1$ or $\epsilon=-1$, and their inavariant coframes satisfy the following EDS: 
 \be
\begin{aligned}
\der\theta^1=&\Theta\dz\theta^2+2\epsilon \theta^1\dz\theta^2+\theta^1\dz\theta^8-\theta^4\dz\theta^7-\theta^5\dz\theta^6,\\
\der\theta^2=&-2\theta^5\dz\theta^7,\\
\der\theta^3=&2\Theta\dz\theta^1-2\epsilon\theta^2\dz\theta^3+\theta^2\dz\theta^5+\theta^2\dz\theta^7+2\theta^3\dz\theta^8-2\theta^4\dz\theta^6,\\
\der\theta^4=&\Theta\dz\theta^5+\epsilon \theta^1\dz\theta^7-\epsilon\theta^2\dz\theta^4+\theta^6\dz\theta^8,\\
\der\theta^5=&-\epsilon\theta^2\dz\theta^5+\epsilon \theta^2\dz\theta^7-\theta^5\dz\theta^8+\theta^7\dz\theta^8,\\
\der\theta^6=&\Theta\dz\theta^7+3\epsilon \theta^1\dz\theta^5+2\epsilon\theta^1\dz\theta^7-\epsilon\theta^2\dz\theta^6+3\theta^4\dz\theta^8+2\theta^6\dz\theta^8,\\
\der\theta^7=&3\epsilon\theta^2\dz\theta^5+\epsilon\theta^2\dz\theta^7+3\theta^5\dz\theta^8+\theta^7\dz\theta^8,\\
\der\theta^8=&2\epsilon \theta^5\dz\theta^7,\\
\der\Theta=&\Theta\dz\theta^8+3\theta^1\dz\theta^2+3\epsilon\theta^4\dz\theta^5+2\epsilon\theta^4\dz\theta^7-\epsilon\theta^6\dz\theta^7.\end{aligned}
\label{ss180}\ee
The two nonequivalent structures are characterized as the only ones that have vanishing curvature
$$\Omega=0$$
of the corresponding Cartan connection $\varpi$. For them the  base manifold $M$ of the Cartan bundle $\bbR\to{\mathcal U}\to M$ is a homogeneous space $M={\mathcal U}/\bbR$, and the structures have strictly 9-dimensional group of local symmetries. The 9-dimensional group of symmetries is locally homeomorphic to $\mathcal U$, which gets equipped with a Lie group structure with the coframe $(\theta^1,\theta^2,\dots,\theta^8,\Theta)$ beeing a basis of its Maurer-Cartan form $\varpi$. 
 \end{theorem}

To get an explicit realization of a Monge system with nine symmetries we take
\be h(y_2,y_7)=\tfrac13\frac{y_2^3}{y_7^2}\label{choice}\ee
in the solutions (\ref{systemu}) described in Theorem \ref{exa}. We then have the following expressions
\be\begin{aligned}
  I_1=&\frac{16(y_1y_7-y_4^2)}{y_7^6},\quad I_2=\frac{2}{y_7^3},\quad I_3=0,\quad R=-\frac{2}{y_7^6},\\
  u=&\frac{1}{y_7^2},\quad v_2=-\frac{4y_4}{y_7^3},\quad z_4=\frac{24y_4(3y_1y_7-8 y_4^2)}{y_7^5}
  \end{aligned}\label{bla}\ee
for the invariants (\ref{in1}) and the relevant structural functions (\ref{sysc3u1}). Thus we have structures with $\epsilon=1$ in the regions where $y_1y_7>y_4^2$, $y_7\neq 0$ and with $\epsilon=-1$ in the regions where $y_1y_7<y4^2$, $y_7\neq 0$. Since in these domains we also have $I_2\neq 0$, we can ask for the curvature functions $S_2,S_3,S_4,S_6,S_7,S_8,S_9,S_{10}$. Using the definitions (\ref{fsa})) and relations (\ref{bla}) we easilly find that the considered structures have $S_2=S_4=0$. Thus, applying Theorem \ref{sy9}, we see that the choice of a function $h=h(y_2,y_7)$ as in (\ref{choice}) corresponds precisely to the Monge structures having flat Cartan connection $\varpi$. Thus we have the following proposition. 
\begin{proposition}
  The two Monge structures, which have the main invariant $T(b_6)$ with a root of multiplicity five and, which have the group of local symmetries of dimension precisely equal to nine, admit local coordinates $(y_1,y_2,\dots,y_8)$ in which the Monge coframe 1-forms $(\om^1,\om^2,\dots,\om^8)$ satisfying (\ref{sysc3u}) reads:
  $$\begin{aligned}
   \om^1=&\der y_8-y_5\der y_2+y_3\der y_4,\\
   \om^2=&\der y_7-2y_4\der y_2,\\
   \om^3=&\der y_6+2y_3\der y_5-\frac{y_2}{y_7^2}\der y_7,\\
   \om^4=&\der y_5-y_1\der y_3,\\
   \om^5=&\der y_4-y_1\der y_2,\\
   \om^6=&\der y_3,\\
   \om^7=&\der y_2,\\
   \om^8=&\der y_1.
    \end{aligned}
$$
  The structures with $\epsilon=1$ correspond to the regions $y_1y_7>y_4^2$, $y_7\neq 0$, and the structures with $\epsilon=-1$ to the regions $y_1y_7>y_4^2$, $y_7\neq 0$.
\end{proposition}

\end{document}